\documentclass[11pt,letterpaper]{amsart}
\usepackage[T1]{fontenc}
\usepackage[latin9]{inputenc}
\usepackage{geometry}
\usepackage{wrapfig}
\usepackage{multicol}
\usepackage{enumerate}   
\usepackage{graphicx}
\usepackage{soul}
\usepackage{xcolor}
\usepackage{amssymb}

\newtheorem{theorem}{Theorem}[section]
\newtheorem{lemma}[theorem]{Lemma}
\newtheorem{proposition}[theorem]{Proposition}
\newtheorem{corollary}[theorem]{Corollary}
{ \theoremstyle{definition}
\newtheorem{definition}[theorem]{Definition}}
{ \theoremstyle{remark}
\newtheorem{remark}[theorem]{Remark}}
\usepackage{placeins}
\usepackage{cite}
\usepackage{caption}
\usepackage{enumerate}
\usepackage{afterpage}
\usepackage{enumitem}
\usepackage{bmpsize}
\usepackage{hyperref}
\usepackage{tabu}
\usepackage{enumitem}   
\numberwithin{equation}{section}
\usepackage{stmaryrd}

\newcommand{\N}{\mathbb{N}}
\newcommand{\Z}{\mathbb{Z}}
\newcommand{\R}{\mathbb{R}}
\newcommand{\E}{\mathbb{E}}

\newcommand{\CL}{\mathcal{L}}

\title{Characterization of $H$-Brownian Gibbsian line ensembles}
\date{\today}
\author{Evgeni Dimitrov}

\begin{document}

\maketitle 

\begin{abstract}
In this paper we show that an $H$-Brownian Gibbsian line ensemble is completely characterized by the finite-dimensional marginals of its lowest indexed curve for a large class of interaction Hamiltonians $H$. A particular consequence of our result is that the KPZ line ensemble is the unique line ensemble that satisfies the $H_1$-Brownian Gibbs property with $H_1(x) = e^x$ and whose lowest indexed curve is equal to the Cole-Hopf solution to the narrow wedge KPZ equation.
\end{abstract}

\tableofcontents

%-------------------------------------------------------------------------------------------------------------------------------------------------------------------------------------------------
% Section 1
%
%-------------------------------------------------------------------------------------------------------------------------------------------------------------------------------------------------
\section{Introduction and main result}\label{Section1}

%-------------------------------------------------------------------------------------------------------------------------------------------------------------------------------------------------
% Section 1.1
%
%-------------------------------------------------------------------------------------------------------------------------------------------------------------------------------------------------
\subsection{Gibbsian line ensembles}\label{Section1.1}In the past two decades there has been an increased interest in studying {\em Gibbsian line ensembles} due to their occurrence in integrable probability and the Kardar-Parisi-Zhang (KPZ) universality class. Informally, a Gibbsian line ensemble can be thought of as a collection of labeled random walks or Brownian motions, whose joint law is reweighed by a Radon-Nikodym derivative proportional to the exponential of the sum of local interaction energies between consecutively labeled curves. We refer to the introduction of \cite{BCDpaperII} for a more extensive discussion of general Gibbsian line ensembles and their applications. Below we focus on a specific class of Gibbsian line ensembles, whose underlying paths are Brownian and whose curves interact through a Hamiltonian function $H$. We begin by introducing a bit of notation.

A {\em $\Sigma$-indexed line ensemble} is a collections of curves $\mathcal{L} = \{ \mathcal{L}_i \}_{i \in \Sigma}$ all defined on the same probability space, where $\Sigma \subset \mathbb{Z}$ and for each $i \in \Sigma$ the $\mathcal{L}_i$ is a continuous function on a fixed interval $\Lambda  \subset \mathbb{R}$. In this paper, we will almost exclusively consider the case when $\Sigma = \llbracket 1, N \rrbracket$, where we write $\llbracket p, q \rrbracket = \{p, p+1, \dots, q\}$ and $N \in \mathbb{N} \cup \infty$ (if $N = \infty$ we adopt the convention $\Sigma = \mathbb{N}$). We say that the $\llbracket 1, N \rrbracket$-indexed line ensemble $\mathcal{L}$ satisfies the {\em $H$-Brownian Gibbs property} if for any $k_1, k_2 \in \Sigma$ with $k_1\leq k_2$ and $[a,b] \subset \Lambda$ the law of the curves $\mathcal{L}_{k_1}, \dots, \mathcal{L}_{k_2}$ on the interval $[a,b]$ is that of $k = k_2 - k_1 + 1$ independent Brownian bridges $\{ \mathcal{Q}_i \}_{i = k_1}^{k_2}$ between the points $(a, \mathcal{L}_{i}(a))$ and $(b, \mathcal{L}_{i}(b))$ for $i = k_1, \dots, k_2$, whose law has been reweighed by the following Radon-Nikodym derivative
\begin{equation}\label{S1E1}
W_H(\mathcal{Q}_{k_1}, \dots, \mathcal{Q}_{k_2}) \propto \exp \left( - \sum_{i = k_1 - 1}^{k_2} \int_a^b H \left( \mathcal{Q}_{i+1}(u) - \mathcal{Q}_i(u)\right)du \right),
\end{equation}
where we have $\mathcal{Q}_{k_1-1} = \mathcal{L}_{k_1-1}$ and $\mathcal{Q}_{k_2 + 1} = \mathcal{L}_{k_2+1}$. 

There are a few subtleties about the above definition. For example, one needs to make a convention about what happens when $k_1 = 1$ or $k_2 = N$ (assuming $N < \infty$) and we will assume in these cases that $\mathcal{L}_0 = \infty$, while $\mathcal{L}_{N+1} = -\infty$. In addition, one should investigate the well-posedness of the Radon-Nikodym derivative implicitly defined in (\ref{S1E1}) (for example if $H: [-\infty, \infty) \rightarrow [0, \infty)$ is continuous the Radon-Nikodym is well-defined). All of these issues will be addressed in Section \ref{Section2.1} where the formal definition is given, and here we proceed with our (somewhat informal) discussion of $H$-Brownian Gibbsian line ensembles.\\

In plain words, a $\llbracket 1, N \rrbracket$-indexed $H$-Brownian Gibbsian line ensemble $\mathcal{L} = \{\mathcal{L}_i \}_{i = 1}^N$ is a collection of $N$ random curves, whose local structure is that of independent Brownian bridges, such that the $i$-th curve $\mathcal{L}_i$ interacts only with its neighbors $\mathcal{L}_{i-1}$ and $\mathcal{L}_{i+1}$ as described in (\ref{S1E1}). For example, if $H(x) \equiv \mathsf{const}$ for $x \in \Lambda$ then the curves do not interact at all and $N$ independent standard Brownian motions form a  $\llbracket 1, N \rrbracket$-indexed $H$-Brownian Gibbsian line ensemble on $\Lambda = [0, \infty)$. 

Another natural choice for $H$ is to set $H(x) = 0$ if $x < 0$ and $H(x) = \infty$ if $x \geq 0$. Notice that for this choice of $H$ we have that $W_H$ in (\ref{S1E1}) equals $1$ if $\mathcal{Q}_{k_1}(x) > \cdots > \mathcal{Q}_{k_2}(x)$ for all $x \in [a,b]$ and otherwise it equals $0$. In this case the law of $\mathcal{L}$ is supported on curves such that $\mathcal{L}_1(x) > \mathcal{L}_2(x) > \cdots$ for all $x \in \Lambda$ and the local structure of $\mathcal{L}$ is that of independent Brownian bridges, which have been conditioned to never intersect. For this special choice of $H$, the terms {\em Brownian Gibbs property} and {\em Brownian Gibbsian line ensembles} are used instead, cf. \cite{CorHamA}. For finite $N$ {\em Dyson Brownian motion} \cite{Dys62} is a $\llbracket 1, N \rrbracket$-indexed Brownian Gibbsian line ensemble on $\Lambda = (0,\infty)$. An example of a $\llbracket 1, N \rrbracket$-indexed Brownian Gibbsian line ensemble with $N =\infty$ and $\Lambda = \mathbb{R}$ is given by the {\em parabolic Airy line ensemble} $\mathcal{L}^{Airy}$, see Figure \ref{S1_1}. 
\begin{figure}[ht]
\begin{center}
  \includegraphics[scale = 0.7]{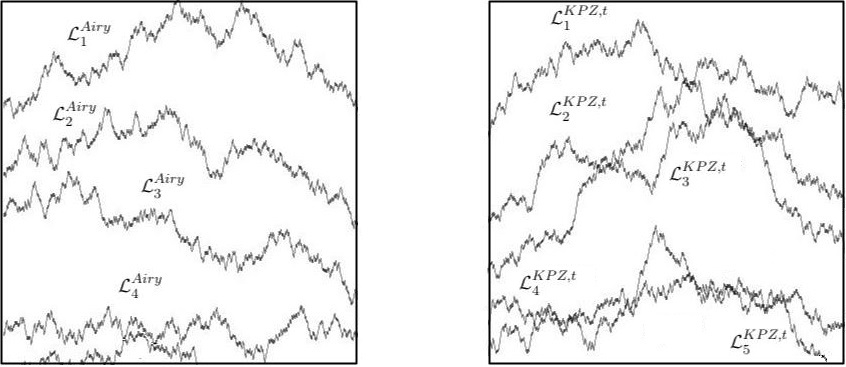}
  \vspace{-2mm}
  \caption{The left picture depicts the parabolic Airy line ensemble $\mathcal{L}^{Airy} = \{\mathcal{L}^{Airy}_i\}_{i = 1}^\infty$ and the right depicts the KPZ$_t$ line ensemble $\mathcal{L}^{KPZ,t} = \{\mathcal{L}^{KPZ,t}_i\}_{i = 1}^\infty$.}
  \label{S1_1}
  \end{center}
\end{figure}

The parabolic Airy line ensemble is a random collection of curves $\mathcal{L}^{Airy} = \{\mathcal{L}^{Airy}_i\}_{i = 1}^\infty$ such that:
\begin{enumerate}[leftmargin =  20mm ]
\item[{ \em    (Airy-1)}] If $\mathcal{A}^{Airy} = \{\mathcal{A}_i^{Airy}\}_{i = 1}^\infty$ is defined by $\mathcal{A}^{Airy}_i(t) = 2^{1/2} \mathcal{L}^{Airy}_i(t) + t^2$ for $i \in \mathbb{N}$, then the lowest indexed curve $\mathcal{A}_1^{Airy}$ is equal in distribution to the {\em Airy process} from \cite{Spohn};
\item[{ \em    (Airy-2)}] $\mathcal{L}^{Airy}$ satisfies the Brownian Gibbs property, i.e. the $H$-Brownian Gibbs property with $H(x) = 0$ if $x < 0$ and $H(x) = \infty$ if $x \geq 0$.
\end{enumerate}
The $\llbracket 1, N \rrbracket$-indexed line ensemble $\mathcal{A}^{Airy}$ is called the {\em Airy line ensemble} and was constructed in \cite{CorHamA} as the weak $N \rightarrow \infty$ limit of $N$ non-intersecting Brownian bridges. It is a stationary line ensemble, whose top curve is given by the Airy process, introduced in \cite{Spohn}. The parabolic Airy line ensemble $\mathcal{L}^{Airy}$ is obtained from $\mathcal{A}^{Airy}$ by subtracting $t^2$ and dividing by $2^{1/2}$. The parabolic shift by $-t^2$ is what ensures that $\mathcal{L}^{Airy}$ satisfies the Brownian Gibbs property, and the division by $2^{1/2}$ is to make the paths of $\mathcal{L}^{Airy}$ be Brownian bridges with diffusion parameter $1$.

Yet another natural choice of $H$, which is closer to the topic of the present paper, is to take $H(x) = H_t(x):= e^{t^{1/3}x}$ for $t > 0$. One can think of this family of interactions as interpolating between the $H_0(x) \equiv 1$ and $H_\infty(x) = \infty \cdot {\bf 1}\{ x> 0\}$, which are the two cases we considered above. The curves of a $\llbracket 1, N \rrbracket$-indexed $H_t$-Brownian Gibbsian line ensemble are allowed to go out of order (i.e. cross); however, they are exponentially penalized for doing so. For finite $N$ the free energies of the {\em O'Connell-Yor semi-discrete Brownian polymer} \cite{CorHamK, OC, OCY} form a $\llbracket 1, N \rrbracket$-indexed $H_1$-Brownian Gibbsian line ensemble on $[0,\infty)$. An example of a $\llbracket 1, N \rrbracket$-indexed $H_t$-Brownian Gibbsian line ensemble with $N =\infty$ and $\Lambda = \mathbb{R}$ is given by the KPZ$_t$ line ensemble $\mathcal{L}^{KPZ,t}$, see Figure \ref{S1_1}. For $t \geq 1$ the KPZ$_t$ line ensemble is a random collection of curves $\mathcal{L}^{KPZ,t} = \{\mathcal{L}^{KPZ,t}_i\}_{i = 1}^\infty$ such that:
\begin{enumerate}[leftmargin =  20mm ]
\item[{ \em    (KPZ$_t$-1)}] The lowest indexed curve $\mathcal{L}^{KPZ,t}_1$ is equal in distribution to the time $t$ Cole-Hopf solution of the KPZ equation 
\begin{equation}\label{KPZEq}
\partial_tH(x,t) = \frac{1}{2} \partial^2_x \mathcal{H}(x,t) + \frac{1}{2} [ \partial_x \mathcal{H}(x,t)]^2 +  \xi(x,t),
\end{equation}
with narrow wedge initial data;
\item[{ \em    (KPZ$_t$-2)}] $\mathcal{L}^{KPZ,t} $ satisfies the $H$-Brownian Gibbs property with $H(x) =  e^{t^{1/3}x}$.
\end{enumerate}
In (\ref{KPZEq}) the letters $t$ and $x$ denote time and space, and $\xi$ is space-time Gaussian white noise, so that formally $\mathbb{E} \left[ \xi(x,t) \xi(y,s)\right] = \delta(t-s)\delta(x-y).$ For more background and history of the KPZ equation we refer the reader to the surveys  \cite{CU2, QS}. The KPZ$_t$ line ensemble was constructed in \cite{CorHamK} as the weak limit of the free energies of the O'Connell-Yor semi-discrete Brownian polymer.\\

The above few paragraphs introduced our main objects of interest -- the $H$-Brownian Gibbsian line ensembles, and gave some natural examples. We next turn to formulating our results.

In the recent joint work with K. Matetski \cite{DM20}, the author showed that a Brownian Gibbsian line ensemble is completely characterized by the finite-dimensional distribution of its lowest indexed curve. In particular, the parabolic Airy line ensemble $\mathcal{L}^{Airy}$ is the unique $\mathbb{N}$-indexed line ensemble on $\mathbb{R}$ that satisfies conditions (Airy-1) and (Airy-2) above. In the present paper we extend this result to $H$-Brownian Gibbsian line ensembles. The main result of the paper is as follows.

\begin{theorem}\label{ThmMainS1}  Let $\Sigma = \{ 1,2, \dots,  N \}$ with $N \in \mathbb{N}$ or $\Sigma =\mathbb{N}$, and let $\Lambda \subset \mathbb{R}$ be an interval. Suppose further that $H: [-\infty, \infty) \rightarrow [0, \infty)$ is a continuous function, which is convex and such that for some fixed $\lambda > 0$ one has for any $M \in \mathbb{N}$ that
\begin{equation}\label{S1E1.5}
\limsup_{y \rightarrow \infty} \sup_{x \in [-M, M]} \left| \frac{H(x+y)}{H(y)} - e^{\lambda x} \right| = 0.
\end{equation}
 Let $\mathcal{L}^1$ and $\mathcal{L}^2$ be $\Sigma$-indexed line ensembles on $\Lambda$ satisfying the $H$-Brownian Gibbs property with laws $\mathbb{P}_1$ and $\mathbb{P}_2$, respectively. Suppose further that for every $k\in \mathbb{N}$,  $t_1 < t_2 < \cdots < t_k$ with $t_i \in \Lambda$ and $x_1, \dots, x_k \in \mathbb{R}$ we have that 
\begin{equation*}
\mathbb{P}_1 \left( \mathcal{L}^1_1(t_1) \leq x_1, \dots,\mathcal{L}^1_1(t_k) \leq x_k  \right) =\mathbb{P}_2 \left( \mathcal{L}^2_1(t_1) \leq x_1, \dots,\mathcal{L}^2_1(t_k) \leq x_k  \right).
\end{equation*}
Then we have that $\mathbb{P}_1 = \mathbb{P}_2$.
\end{theorem}
\begin{remark}
In plain words, Theorem \ref{ThmMainS1} states that if two line ensembles both satisfy the $H$-Brownian Gibbs property and their lowest indexed curves have the same finite-dimensional distributions, then they have the same distribution as line ensembles, provided that $H$ satisfies a few technical conditions. 
\end{remark}
\begin{remark}
Theorem \ref{ThmMainS1} is formulated slightly more generally after introducing some necessary notation as Theorem \ref{ThmMain} in the main text.
\end{remark}
\begin{remark}
Since the functions $H_t(x) = e^{t^{1/3}x}$ for $t > 0$ satisfy the conditions in the Theorem \ref{ThmMainS1}, we conclude that for any $t \geq 1$ the KPZ$_t$ line ensemble $\mathcal{L}^{KPZ,t}$ is the unique $\mathbb{N}$-indexed line ensemble on $\mathbb{R}$ that satisfies conditions (KPZ$_t$-1) and (KPZ$_t$-2). We mention that Theorem \ref{ThmMainS1} {\em does not} apply to the parabolic Airy line ensemble since (for example) the function $H_\infty(x) = \infty \cdot {\bf 1}\{ x> 0\}$ is not continuous. In addition, Theorem \ref{ThmMainS1} {\em does not} apply to the case of $N$ independent Brownian motions since $H_0(x) =1$ does not satisfy $\lim_{y \rightarrow \infty} H_0(1 + y)/H_0(y) = e^{\lambda}$ for any $\lambda > 0$.
\end{remark}
\begin{remark}
The class of functions $H$ in the statement of Theorem \ref{ThmMainS1} is dictated by the particular approach we take in the proof, which goes through constructing a certain class of observables of the top $k$ curves that probe the $(k+1)\textsuperscript{st}$. It seems possible that one can handle more general functions $H$ by modifying these observables, but it is presently not clear how to accomplish this. We also mention that in addition to $H(x) = e^{\lambda x}$ for $\lambda >0$, Theorem \ref{ThmMainS1} holds when $H(x) = f(x) e^{\lambda x} + g(x)$ is non-negative, continuous and convex, $f(x)$ is a polynomial and $g(x)$ grows slower than $e^{\lambda x}$ near infinity. For example one can take $H(x) = e^{\lambda x} + x^2\cdot {\bf 1}\{x \geq 0\} \mbox{ or } H(x) = C_0 e^{\lambda x} + \sum_{i = 1}^k C_i e^{\lambda_i x},$
where  $\lambda > \max_{1 \leq i \leq k} \lambda_i$ and $\lambda_i,C_i  >0$ for all $i$. One can also take $H(x) = (x^2 + 4) e^x$. 
\end{remark}

In the next few pages we will explain the main ideas behind the proof of Theorem \ref{ThmMainS1}, or rather its generalization Theorem \ref{ThmMain} in the text. We will try to illustrate how these ideas are different from those in \cite{DM20}, which (informally) correspond to the case $H(x) = H_\infty(x) = \infty \cdot {\bf 1}\{ x> 0\}$. We will also explain the origin of the assumptions on $H$ in Theorems \ref{ThmMainS1} and \ref{ThmMain}. 

The philosophy of the proof of Theorem \ref{ThmMain} is similar to \cite{DM20}. The argument is inductive and one roughly shows that if two $H$-Brownian Gibbsian line ensembles $\mathcal{L}^1$ and $\mathcal{L}^2$ have the same finite-dimensional distributions when restricted to their lowest-labeled $k$ curves then the same is true for when they are restricted to their lowest-labeled $k+1$ curves. The difficulty lies in establishing the induction step, since we are assuming the statement of the theorem for the base case $k = 1$. In going from $k$ to $k+1$ we use equality of laws of $\{\mathcal{L}^1_i\}_{i = 1}^k$ and $\{\mathcal{L}^2_i\}_{i = 1}^k$, available by the induction hypothesis. The key idea is to construct a family of observables, which are measurable with respect to the lowest-labeled $k$ curves, but which probe the $(k+1)\textsuperscript{st}$ one. Informally speaking, the law of $\{\mathcal{L}^v_i\}_{i = 1}^k$ for $v \in \{1,2\}$  is that of $k$ independent Brownian bridges, which interact with each other and the $(k+1)\textsuperscript{st}$ curve $\mathcal{L}^v_{k+1}$ through the Hamiltonian $H$ as in (\ref{S1E1}). Then, the observables we construct measure the difference in the local behavior between $\{\mathcal{L}^v_i\}_{i = 1}^k$ and that of $k$ independent Brownian bridges that continue to interact with each other but {\em not} with $\mathcal{L}^v_{k+1}$.

The differences with \cite{DM20} lie in the choice of observables and also in determining general conditions on $H$, for which the above strategy would work. To illustrate the differences, let us explain how one establishes the equality in law of $\mathcal{L}^1_2(t)$ and $\mathcal{L}^2_2(t)$ for a fixed $t \in \Lambda$ in \cite{DM20} (which is the case $H(x) = H_\infty(x) = \infty \cdot {\bf 1}\{ x> 0\}$) and in our setup. In \cite{DM20} the authors considered the following family of observables for $v \in \{1,2\}$
\begin{equation}\label{S1Obs1}
q_n^v = \mathbb{E}_{\mathbb{P}_v} \left[ \frac{{\bf 1} \{ \mathcal{L}_1^v(t) \leq x \}}{\mathbb{P}_{free}^{v, n} (\mathcal{Q}_{bot}(t) \leq x)} \right],
\end{equation}
where $x \in \mathbb{R}$ is fixed, $t \in [a_n, b_n]$ and $b_n - a_n \rightarrow 0$. In the above equation $\mathbb{P}_{free}^{v, n}$ denotes the law of a Brownian bridge with diffusion parameter $1$ between the points $(a_n, \mathcal{L}_1^v(a_n))$ and $(b_n, \mathcal{L}_1^v(b_n))$ and $\mathcal{Q}_{bot}$ has law $\mathbb{P}_{free}^{v,n}$, see Figure \ref{S1_2}. 
\begin{figure}[ht]
\begin{center}
  \includegraphics[scale = 0.47]{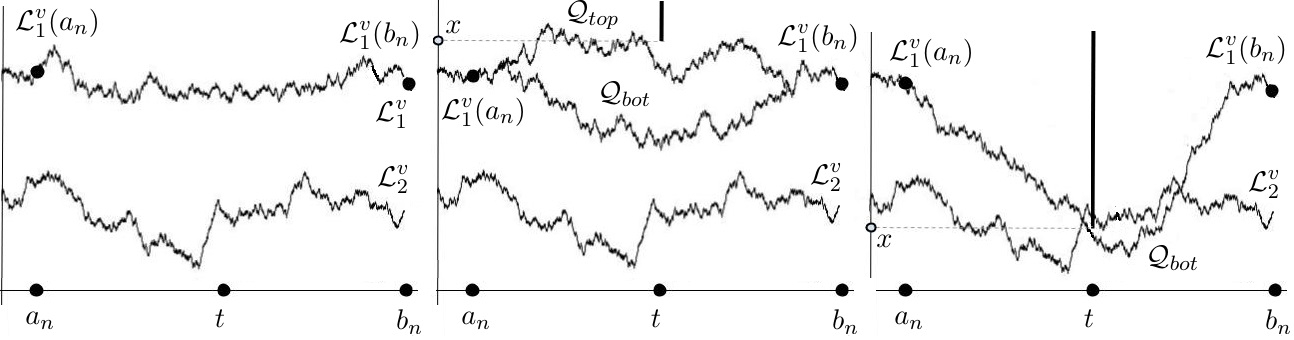}
  \vspace{-2mm}
  \caption{The left picture depicts the lowest-indexed two curves $\mathcal{L}^v_1$ and $\mathcal{L}^v_2$ of the line ensemble $\mathcal{L}^v$. The middle picture shows a level $x$ that is above $\mathcal{L}^v_2(t)$ and samples of $\mathcal{Q}_{top}$ from $\mathbb{P}_{H_\infty}^{v,n} $ and $\mathcal{Q}_{bot}$ from $\mathbb{P}_{free}^{v,n}$ such that $\mathcal{Q}_{top}(t) \leq x$ and $\mathcal{Q}_{bot}(t) \leq x$. Both $\mathcal{Q}_{top}$ and $\mathcal{Q}_{bot}$ are Brownian bridges between the points $(a_n, \mathcal{L}^v_1(a_n))$ and $(b_n, \mathcal{L}^v_1(b_n))$ with $\mathcal{Q}_{top}$ being conditioned to stay above $\mathcal{L}^v_2$ on $[a_n, b_n]$, and $\mathcal{Q}_{bot}$ being free to fall below it. The picture on the right depicts $\mathcal{Q}_{bot}$ when $x < \mathcal{L}^v_2(t)$. In this case there is no $\mathcal{Q}_{top}$ with $\mathcal{Q}_{top}(t) \leq x$, since $\mathcal{Q}_{top}$ is supported on curves above $\mathcal{L}^v_2$ on $[a_n, b_n]$. }
  \label{S1_2}
  \end{center}
\end{figure}

Since the function inside the expectation only depends on the lowest-labeled curve $\mathcal{L}_1^v$ we have that $q_n^1 = q_n^2$ for all $n \in \mathbb{N}$. Upon applying the $H_{\infty}$-Brownian Gibbs property one obtains that 
$$q_n^v = \mathbb{E}_{\mathbb{P}_v} \left[ \frac{\mathbb{P}_{H_\infty}^{v,n} (\mathcal{Q}_{top}(t) \leq x)}{\mathbb{P}_{free}^{v,n} (\mathcal{Q}_{bot}(t) \leq x)} \right],$$
where we write $\mathbb{P}_{H}^{v,n}$ for the law of a Brownian bridge with diffusion parameter $1$ between the points $(a_n, \mathcal{L}_1^v(a_n))$ and $(b_n, \mathcal{L}_1^v(b_n))$, which interacts with $\mathcal{L}^v_2$ on the interval $[a_n, b_n]$ through the Hamiltonian $H$ as in (\ref{S1E1}), and $\mathcal{Q}_{top}$ has law $\mathbb{P}_{H}^{v,n}$.

What the authors of \cite{DM20} were able to show (modulo some technicalities) is that as $n \rightarrow \infty$ one has $\mathbb{P}_v$-almost surely
\begin{equation}\label{S1E2}
 \frac{\mathbb{P}_{H_{\infty}}^{v,n} (\mathcal{Q}_{top}(t) \leq x)}{\mathbb{P}_{free}^{v,n} (\mathcal{Q}_{bot}(t) \leq x)} \rightarrow {\bf 1} \{ \mathcal{L}_2^v(t) \leq x\}.
\end{equation}
In addition, the authors showed that the random curves $\mathcal{Q}_{top}$ and $\mathcal{Q}_{bot}$ can be coupled so that $\mathcal{Q}_{top}(s) \geq \mathcal{Q}_{bot}(s)$ for all $s \in [a_n, b_n]$, see Figure \ref{S1_2}. Consequently, the quantities on the left side of (\ref{S1E2}) are bounded by $1$ and the bounded convergence theorem implies
\begin{equation}\label{S1E3}
\mathbb{P}_1(\mathcal{L}_2^1(t) \leq x) = \lim_{n\rightarrow \infty} q_n^1 =\lim_{n\rightarrow \infty} q_n^2  = \mathbb{P}_2(\mathcal{L}_2^2(t) \leq x).
\end{equation}
This proves the equality in law of $\mathcal{L}_2^1(t) $ and $\mathcal{L}_2^2(t)$.

In the setup of Theorem \ref{ThmMain} the above choice of observables will not yield anything, because $\mathbb{P}_v$-almost surely
\begin{equation}\label{S1E3.5}
\frac{\mathbb{P}_{H}^{v,n} (\mathcal{Q}_{top}(t) \leq x)}{\mathbb{P}_{free}^{v,n} (\mathcal{Q}_{bot}(t) \leq x)} \rightarrow 1.
\end{equation}
The reason behind this stark difference with (\ref{S1E2}) can be understood by looking at the Radon-Nikodym derivative of $\mathcal{Q}_{top}$ with respect to a Brownian bridge, which is given by 
\begin{equation}\label{S1E4}
\frac{d \mathbb{P}_{H}^{v,n} }{d \mathbb{P}_{free}^{v,n} }  (\mathcal{Q}) = W^{v,n}_H(\mathcal{Q}) \propto \exp \left( -  \int_{a_n}^{b_n} H \left(  \mathcal{L}_2^v(u) -\mathcal{Q}(u) \right)du \right).
\end{equation}
If $H = H_\infty$ we see that the above Radon-Nikodym derivative is $0$ if $\mathcal{Q}(t) < \mathcal{L}_2^v(t)$ regardless of how small $b_n - a_n$ is and if $\mathcal{Q}(t) > \mathcal{L}_2^v(t)$ it is asymptotically (as $n \rightarrow \infty$) equal to $1$. This dichotomy is what ultimately fuels the limit in (\ref{S1E2}) to an indicator function. On the other hand, for a continuous function $H$ if $b_n - a_n \rightarrow 0$ then we always have $W^{v,n}_H(\mathcal{Q}) \rightarrow 1$, and this ultimately results in (\ref{S1E3.5}).

Faced with the above difficulty, we needed to find a different set of observables from those in (\ref{S1Obs1}) that would work in the case of a continuous Hamiltonian $H$. The main idea behind our new observables is that instead of pushing the lowest-labeled curve below a fixed level $x$ as is done in (\ref{S1Obs1}), we push it  below a varying with $n$ level $-A_n$ where $A_n \rightarrow \infty$ as $n \rightarrow \infty$. For technical reasons, that will become clearer later in the paper, it is more convenient to push the lowest-labeled curve below $-A_n$ at two points $c_n, d_n$ such that $a_n < c_n < t < d_n < b_n$ and both $b_n - d_n$, $c_n - a_n$ are tiny compared to $d_n - c_n$. The resulting observables are then given by
\begin{equation}\label{S1Obs2}
p_n^v = \mathbb{E}_{\mathbb{P}_v} \left[ \frac{{\bf 1} \{ \mathcal{L}_1^v(c_n) \leq -A_n \} \cdot {\bf 1} \{ \mathcal{L}_1^v(d_n) \leq -A_n \}}{\mathbb{P}_{free}^{v,n} (\mathcal{Q}_{bot}(c_n) \leq  - A_n \mbox{ and } \mathcal{Q}_{bot}(d_n) \leq  - A_n)} \right].
\end{equation}
As before, we have that the function inside the expectation only depends on the lowest-labeled curve $\mathcal{L}_1^v$ and so $p_n^1 = p_n^2$ for all $n \in \mathbb{N}$. Upon applying the $H$-Brownian Gibbs property one obtains that 
$$p_n^v = \mathbb{E}_{\mathbb{P}_v} \left[ \frac{\mathbb{P}_{H}^{v,n} (\mathcal{Q}_{top}(c_n) \leq  - A_n \mbox{ and } \mathcal{Q}_{top}(d_n) \leq  - A_n)}{\mathbb{P}_{free}^{v,n}  (\mathcal{Q}_{bot}(c_n) \leq  - A_n \mbox{ and } \mathcal{Q}_{bot}(d_n) \leq  - A_n)} \right].$$
See Figure \ref{S1_3} for a pictorial depiction of some of the above notation.
\begin{figure}[ht]
\begin{center}
  \includegraphics[scale = 0.50]{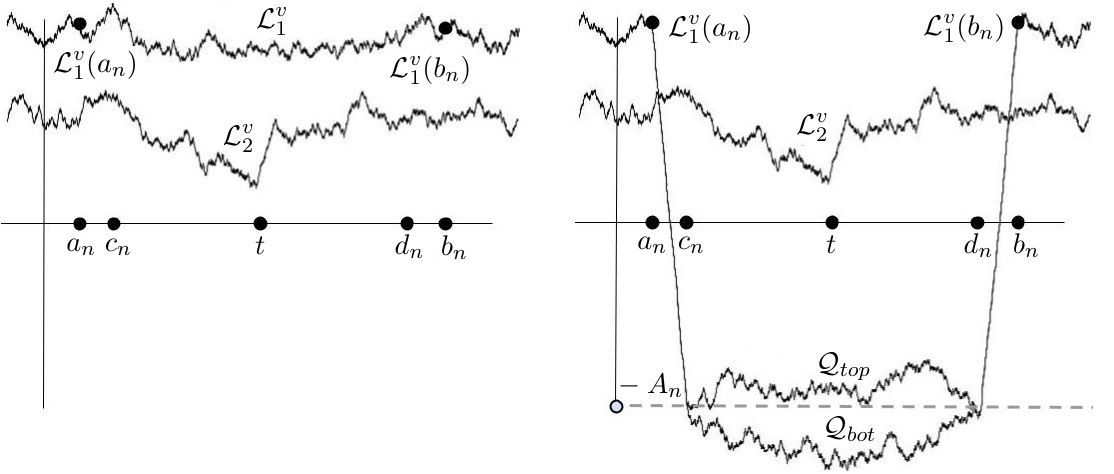}
  \vspace{-2mm}
  \caption{The left picture depicts the lowest-indexed two curves $\mathcal{L}^v_1$ and $\mathcal{L}^v_2$ of the line ensemble $\mathcal{L}^v$. The right picture shows the level $-A_n$ and samples of $\mathcal{Q}_{top}$ from $\mathbb{P}_{H}^{v,n} $ and $\mathcal{Q}_{bot}$ from $\mathbb{P}_{free}^{v,n}$ such that $\mathcal{Q}_{top}(c_n),\mathcal{Q}_{top}(d_n)  \leq -A_n$ and $\mathcal{Q}_{bot}(c_n),\mathcal{Q}_{bot}(d_n) \leq -A_n$. Both $\mathcal{Q}_{top}$ and $\mathcal{Q}_{bot}$ are Brownian bridges between the points $(a_n, \mathcal{L}^v_1(a_n))$ and $(b_n, \mathcal{L}^v_1(b_n))$ with $\mathcal{Q}_{top}$ interacting with $\mathcal{L}^v_2$ on $[a_n, b_n]$ through the Hamiltonian $H$ as in (\ref{S1E1}), and $\mathcal{Q}_{bot}$ not interacting with it at all.  }
  \label{S1_3}
  \end{center}
\end{figure}

As will be carefully explained in Section \ref{Section3}, one has the general statement that if a Brownian bridge with diffusion parameter $1$ on $[a_n, b_n]$ has bounded (in $n$) endpoints and is forced to dip below $-A_n$ at times $c_n, d_n$ then it will stay close to the level $-A_n$ on the whole interval $[c_n, d_n]$, see the right part of Figure \ref{S1_3}. On the other hand, since $b_n - a_n \rightarrow 0$ as $n \rightarrow \infty$ and $t \in [a_n, b_n]$, we have that $\mathcal{L}^v_n(s) \approx \mathcal{L}^v_n(t)$ for all $s \in [a_n, b_n]$. The last two observations and the fact that $d_n - c_n \approx b_n -a_n$ suggest that the Radon-Nikodym derivative in (\ref{S1E4}) satisfies
\begin{equation}\label{S1E5}
 W^{v,n}_H(\mathcal{Q}) \propto \exp \left( -  \int_{a_n}^{b_n} H \left(  \mathcal{L}_2^v(u) -\mathcal{Q}(u) \right)du \right) \approx \exp \left( -  [b_n - a_n ]H \left(  \mathcal{L}_2^v(t) + A_n \right) \right).
\end{equation}
We now have two sequences $\delta_n = b_n - a_n$ and $A_n$ that we can control in the hopes of making $p_n^v$ converge as $n \rightarrow \infty$ to an observable of $\mathcal{L}_2^v(t)$. If $\lim_{n \rightarrow \infty} \delta_n H(x + A_n) = G(x)$ for some function $G$ in a sufficiently uniform way in $x$, then one can use (\ref{S1E5}) to essentially show that $\mathbb{P}_v$-almost surely
\begin{equation}\label{S1E6}
\lim_{n \rightarrow \infty} \frac{\mathbb{P}_{H}^{v,n} (\mathcal{Q}_{top}(c_n) \leq  - A_n \mbox{ and } \mathcal{Q}_{top}(d_n) \leq  - A_n)}{\mathbb{P}_{free}^{v,n}  (\mathcal{Q}_{bot}(c_n) \leq  - A_n \mbox{ and } \mathcal{Q}_{bot}(d_n) \leq  - A_n)} = \exp( - G(\mathcal{L}_2^v(t))),
\end{equation}
which upon taking expectations would imply
\begin{equation}\label{S1E7}
\mathbb{E}_{\mathbb{P}_1}\left[\exp( - G(\mathcal{L}_2^1(t)) \right] = \lim_{n\rightarrow \infty} p_n^1 =\lim_{n\rightarrow \infty} p_n^2  = \mathbb{E}_{\mathbb{P}_2}\left[\exp( - G(\mathcal{L}_2^2(t))) \right] .
\end{equation}
If we can show (\ref{S1E7}) for a sufficiently rich family of functions $G$, we would conclude the equality in law of $\mathcal{L}_2^1(t)$ and $\mathcal{L}_2^2(t)$.

We are now faced with the question of finding sequences, $A_n, \delta_n$ and functions $G,H$ that satisfy $\lim_{n \rightarrow \infty} \delta_n H(x + A_n) = G(x)$. Taking $x =0$, we see that we can basically only take $\delta_n = \beta/ H(A_n)$, where $\beta = G(0)$. This suggests $\lim_{y \rightarrow \infty}  H(x + y)/ H(y) = G(x)/ G(0)$ for all $x \in \mathbb{R}$, which implies 
\begin{equation*}
\begin{split}
\frac{G(x+y)}{ G(0)} = \lim_{n \rightarrow \infty}  \frac{H(x + y + A_n)}{H(A_n)}  = \lim_{n \rightarrow \infty} \frac{ H(x + y + A_n)}{ H(x+ A_n )} \cdot \lim_{n \rightarrow \infty} \frac{H(x + A_n)}{H(A_n)}  = \frac{G(y)}{ G(0)} \cdot \frac{G(x)}{ G(0)}.
\end{split}
\end{equation*}
The above equation shows that the only possible choices of $G(x)$ are $G(x) = \beta \cdot e^{\lambda x}$ for some $\lambda, \beta \in \mathbb{R}$. The $\lambda$ is fixed as it must satisfy $\lim_{y \rightarrow \infty} H(y + x)/H(y) = e^{\lambda x}$, while any $\beta > 0$ can be accommodated by taking $\delta_n = \beta/ H(A_n) $. We mention that since $\delta_n \rightarrow 0+$ we must have $\lim_{y \rightarrow \infty} H(y + x)/H(y) = e^{\lambda x}$  with $\lambda > 0$. The last paragraph explains the origin of (\ref{S1E1.5}). 

The reason we need to further assume that $H$ is convex is subtle. Even if $H$ is not convex one would still be able to obtain (\ref{S1E6}) if one assumes that $H$ is continuous and satisfies (\ref{S1E1.5}); however, to deduce (\ref{S1E7}) we need to apply some form of the dominated convergence theorem. If $H$ is convex one can show that $\mathcal{Q}_{top}$ and $\mathcal{Q}_{bot}$ can be coupled so that $\mathcal{Q}_{top}(s) \geq \mathcal{Q}_{bot}(s)$ for all $s \in [a_n ,b_n]$ -- this is done using certain monotone coupling lemmas, which are presented in Section \ref{Section3.4} and date back to \cite{CorHamK}. In particular, for convex $H$ the terms on the left side of (\ref{S1E6}) are all bounded by $1$ and so (\ref{S1E7}) follows from (\ref{S1E6})  and the bounded convergence theorem. We mention that this is the only place where the convexity of $H$ is utilized in the argument, but possibly one can relax this condition by finding a different way of deducing (\ref{S1E7}) from (\ref{S1E6}) -- we do not pursue this direction in the paper.

Once we know that (\ref{S1E7}) holds for $G(x) = \beta \cdot e^{\lambda x}$ for a fixed $\lambda > 0$ and all $\beta > 0$ we can deduce the equality of law of $\mathcal{L}_2^1(t)$ and $\mathcal{L}_2^2(t)$ -- see Lemma \ref{PropFD} for a generalization of this fact. Overall, we see that if one chooses observables as in (\ref{S1Obs2}) to probe the second curve, the assumptions on $H$ in Theorems \ref{ThmMainS1} and \ref{ThmMain} are essentially necessary. We believe that one can possibly relax the continuity and convexity assumptions on $H$ and still obtain (\ref{S1E6}), but we do not know how to obtain (\ref{S1E7}) as a consequence without these assumptions. It is also possible that one can remove the assumption (\ref{S1E1.5}); however, that would require working with a completely different set of observables from (\ref{S1Obs2}), and we are presently not aware of other candidates.

The above description of the main argument is of course quite reductive and the full argument, presented in Section \ref{Section4.1} for the special case we discussed above and Section \ref{Section4.2} in full generality, relies on various technical statements and definitions that are given in Sections \ref{Section2} and \ref{Section3}. We remark that some of the results we establish in these two sections have appeared in earlier studies on $H$-Brownian Gibbsian line ensembles; however, we could not always find complete proofs of them. We have thus opted to fill in the gaps in the proofs of these statements in the literature and this work is the content of the (somewhat) technical Section \ref{Section5}.\\

We end this section with a brief discussion of some of the motivation behind our work. Our interest in Theorem \ref{ThmMainS1} is twofold. Firstly, $H$-Brownian Gibbsian line ensembles have become central objects in probability theory and understanding their structure is an important area of research. As mentioned earlier, Dyson Brownian motion, the parabolic Airy line ensemble, the O'Connell-Yor semi-discrete Brownian polymer and the KPZ$_t$ line ensemble are all examples of $H$-Brownian Gibbsian line ensembles. The second, and more essential, reason we believe Theorem \ref{ThmMainS1} to be important is that it can be used as a tool for proving KPZ universality for various models in integrable probability. We elaborate on the second point below.

The KPZ$_1$ line ensemble is believed to be a universal scaling limit for various models that belong to the KPZ universality class, see \cite{CU2} for an expository review of this class. Theorem \ref{ThmMainS1} paves a way to showing uniform convergence to the KPZ$_1$ line ensemble for various models, which naturally have the structure of a line ensemble and for which the finite dimensional marginals of the top curve can be accessed. The primary examples, we are interested in applying this approach to, come from the {\em Macdonald processes} \cite{BorCor}, and include the {\em Hall-Littlewood processes} \cite{CD, ED}, the {\em $q$-Whittaker processes} \cite{BCF}, and the {\em log-gamma polymer} \cite{COSZ, Sep12}. Each of the models we listed naturally has the structure of a line ensemble with a Gibbs property, which can be found for the Hall-Littlewood process in \cite{CD} and for the log-gamma polymer in \cite{BCDpaperII, Wu19}. If we denote by $\{L^N_i\}_{i = 1}^\infty$ the discrete line ensemble associated to one of the above models and by $\{\mathcal{L}_i^{KPZ,1}\}_{i = 1}^\infty$ the KPZ$_1$ line ensemble, the proposed program for establishing the uniform convergence of $\{L^N_i \}_{i = 1}^\infty$  to $\{\mathcal{L}_i^{KPZ,1}\}_{i = 1}^\infty$ goes through the following steps:
\begin{enumerate}
\item show that as $N \rightarrow \infty$ the lowest indexed curve $L^N_1$ converges in the sense of finite dimensional distributions to $\mathcal{L}^{KPZ,1}_1$, which is the time $1$ Cole-Hopf solution to the KPZ equation (\ref{KPZEq});
\item show that $\left\{L^N_i \right\}_{i = 1}^\infty$ form a tight sequence of line ensembles and that every subsequential limit enjoys the $H$-Brownian Gibbs property with $H(x) = e^{x}$;
\item use the characterization of Theorem \ref{ThmMainS1} to prove that all subsequential limits are given by $\{\mathcal{L}_i^{KPZ,1}\}_{i = 1}^\infty$.
\end{enumerate}
One reason we are optimistic that our proposed program has a good chance of establishing convergence to the KPZ$_1$ line ensemble is that the first step in the program has been completed for several models from integrable probability including the asymmetric simple exclusion process (ASEP) \cite{BerGia} and the stochastic six vertex model \cite{CGST}, as well as for a large class of polymer models in \cite{AKQ}. In addition, there has been substantial progress on establishing tightness for sequences of line ensembles in various related contexts -- see e.g. \cite{BCDpaperII, caputo2019, CIW, CD,  CHH19, CorHamK, DREU,DNV19, DV18, Wu19}. Another reason we are optimistic about our proposed program is that its analogue from \cite{DM20} has already been applied in the preprint \cite{QuaSar} to show the convergence of the KPZ$_t$ line ensemble to the parabolic Airy line ensemble as $t\rightarrow \infty$.

It is worth pointing out that for the log-gamma polymer model the above program appears to have been completed in \cite{Wu19}. More specifically, \cite[Proposition 3.4.9]{Wu19} should imply the finite dimensional convergence in the first step, and \cite[Theorem 3.4.11]{Wu19} should imply the second step. The author believes that once the results from the thesis \cite{Wu19} have been verified and published, Theorem \ref{ThmMainS1} would improve the tightness result from \cite[Theorem 3.4.11]{Wu19} to a uniform convergence result to the KPZ$_1$ line ensemble.

Even beyond the above program, we believe that Theorem \ref{ThmMainS1} will be useful in reducing some of the work in showing convergence of various models to the KPZ$_1$ line ensemble, and is an important result that furthers our understanding of Gibbsian line ensembles in general.

%-------------------------------------------------------------------------------------------------------------------------------------------------------------------------------------------------
% Section 1.2
%
%-------------------------------------------------------------------------------------------------------------------------------------------------------------------------------------------------
\subsection{Outline of the paper}\label{Section1.2} The structure of this paper is as follows. In Section \ref{Section2} we introduce several important definitions, which are used throughout the paper. In particular, we  define $H$-Brownian bridge line ensembles and introduce the (partial) $H$-Brownian Gibbs property. The main result of this paper is stated in this section as Theorem \ref{ThmMain}. In Section \ref{Section3} we collect several technical lemmas involving estimates on Brownian bridges and $H$-Brownian bridge line ensembles as well as two monotone coupling lemmas. In Section \ref{Section4} we prove Theorem \ref{ThmMain} and in Section \ref{Section5} we prove the technical lemmas from Sections \ref{Section2.3} and \ref{Section3.4}.\\

%-------------------------------------------------------------------------------------------------------------------------------------------------------------------------------------------------
% Section 1.3
%
%-------------------------------------------------------------------------------------------------------------------------------------------------------------------------------------------------
\noindent {\bf Acknowledgments.} The author would like to thank Ivan Corwin and Alan Hammond for their useful comments on earlier drafts of this paper. The author is partially supported by the Minerva Foundation Fellowship.

%----------------------------------------------------------------------------------------------------------------------------------------------------------------------------------------------------------------------------
%
%     Section 2
%
%----------------------------------------------------------------------------------------------------------------------------------------------------------------------------------------------------------------------------
\section{Definitions, main result and basic lemmas}\label{Section2}

In this section we introduce the basic definitions that are necessary for formulating our main result, given in Section \ref{Section2.2} below.

%----------------------------------------------------------------------------------------------------------------------------------------------------------------------------------------------------------------------------
%
%     Section 2.1
%
%----------------------------------------------------------------------------------------------------------------------------------------------------------------------------------------------------------------------------
\subsection{Line ensembles and the (partial) $H$-Brownian Gibbs property}\label{Section2.1}
In order to state our main results, we need to introduce some notation as well as the notions of a {\em line ensemble} and the {\em (partial) $H$-Brownian Gibbs property}. Our exposition in this section closely follows that of  \cite[Section 2]{CorHamK} and \cite[Section 2]{DM20}. 

Given two integers $p \leq q$, we let $\llbracket p, q \rrbracket$ denote the set $\{p, p+1, \dots, q\}$. Given an interval $\Lambda \subset \mathbb{R}$, we endow it with the subspace topology of the usual topology on $\mathbb{R}$. We let $(C(\Lambda), \mathcal{C})$ denote the space of continuous functions $f: \Lambda \rightarrow \mathbb{R}$ with the topology of uniform convergence over compacts, see \cite[Chapter 7, Section 46]{Munkres}, and Borel $\sigma$-algebra $\mathcal{C}$. Given a set $\Sigma \subset \mathbb{Z}$, we endow it with the discrete topology and denote by $\Sigma \times \Lambda$ the set of all pairs $(i,x)$ with $i \in \Sigma$ and $x \in \Lambda$ with the product topology. We also denote by $\left(C (\Sigma \times \Lambda), \mathcal{C}_{\Sigma}\right)$ the space of continuous functions on $\Sigma \times \Lambda$ with the topology of uniform convergence over compact sets and Borel $\sigma$-algebra $\mathcal{C}_{\Sigma}$. We will typically take $\Sigma = \llbracket 1, N \rrbracket$ (we use the convention $\Sigma = \mathbb{N}$ if $N = \infty$) and then we write  $\left(C (\Sigma \times \Lambda), \mathcal{C}_{|\Sigma|}\right)$ in place of $\left(C (\Sigma \times \Lambda), \mathcal{C}_{\Sigma}\right)$.
The following defines the notion of a line ensemble.
\begin{definition}\label{DefLE}
Let $\Sigma \subset \mathbb{Z}$ and $\Lambda \subset \mathbb{R}$ be an interval. A {\em $\Sigma$-indexed line ensemble $\mathcal{L}$} is a random variable defined on a probability space $(\Omega, \mathcal{F}, \mathbb{P})$ that takes values in $\left(C (\Sigma \times \Lambda), \mathcal{C}_{\Sigma}\right)$. Intuitively, $\mathcal{L}$ is a collection of random continuous curves (sometimes referred to as {\em lines}), indexed by $\Sigma$,  each of which maps $\Lambda$ in $\mathbb{R}$. We will often slightly abuse notation and write $\mathcal{L}: \Sigma \times \Lambda \rightarrow \mathbb{R}$, even though it is not $\mathcal{L}$ which is such a function, but $\mathcal{L}(\omega)$ for every $\omega \in \Omega$. For $i \in \Sigma$ we write $\mathcal{L}_i(\omega) = (\mathcal{L}(\omega))(i, \cdot)$ for the curve of index $i$ and note that the latter is a map $\mathcal{L}_i: \Omega \rightarrow C(\Lambda)$, which is $(\mathcal{C}, \mathcal{F})-$measurable.
\end{definition}

We next turn to formulating the partial $H$-Brownian Gibbs property -- we do this in Definition \ref{DefPBGP} after introducing some relevant notation. If $W_t$ denotes a standard one-dimensional Brownian motion, then the process
$$\tilde{B}(t) =  W_t - t W_1, \hspace{5mm} 0 \leq t \leq 1,$$
is called a {\em Brownian bridge (from $\tilde{B}(0) = 0$ to $\tilde{B}(1) = 0 $) with diffusion parameter $1$.} 

Given $a, b,x,y \in \mathbb{R}$ with $a < b$, we define a random variable on $(C([a,b]), \mathcal{C})$ through
\begin{equation}\label{BBDef}
B(t) = (b-a)^{1/2} \cdot \tilde{B} \left( \frac{t - a}{b-a} \right) + \left(\frac{b-t}{b-a} \right) \cdot x + \left( \frac{t- a}{b-a}\right) \cdot y, 
\end{equation}
and refer to the law of this random variable as a {\em Brownian bridge (from $B(a) = x$ to $B(b) = y$) with diffusion parameter $1$.} Given $k_1, k_2 \in \mathbb{Z}$ with $k_1 \leq k_2$ and $\vec{x}, \vec{y} \in \mathbb{R}^{k_2 - k_1 + 1}$, we let $\mathbb{P}^{k_1, k_2, a,b, \vec{x},\vec{y}}_{free}$ denote the law of $k_2 - k_1 + 1$ independent Brownian bridges $\{B_i: [a,b] \rightarrow \mathbb{R} \}_{i = k_1}^{k_2}$ from $B_i(a) = x_i$ to $B_i(b) = y_i$ all with diffusion parameter $1$. We will denote the expectation with respect to $\mathbb{P}^{k_1, k_2, a,b, \vec{x},\vec{y}}_{free}$ by $\mathbb{E}^{k_1, k_2, a,b, \vec{x},\vec{y}}_{free}.$

The following definition introduces the notion of an $H$-Brownian bridge line ensemble, which in plain words can be understood as a random ensemble of $k$ independent Brownian bridges, whose adjacent curves interact with each other and with two boundary functions $f$ and $g$ through a Hamiltonian function $H$.
\begin{definition}\label{DefAvoidingLaw}
A {\em Hamiltonian} $H$ is any continuous function $H: [-\infty, \infty) \rightarrow [0, \infty)$. Let $k_1, k_2 \in \mathbb{Z}$ with $k_1 \leq k_2$, and $\vec{x}, \vec{y} \in \mathbb{R}^{k_2 - k_1 + 1}$, $a,b \in \mathbb{R}$ with $a < b$, and $f: [a,b] \rightarrow (-\infty, \infty]$ and $g: [a,b] \rightarrow [-\infty, \infty)$ be two continuous functions. 

With the above data we define the {\em $\llbracket k_1, k_2 \rrbracket \times (a,b)$-indexed $H$-Brownian bridge line ensemble with entrance data $\vec{x}$, exit data $\vec{y}$ and boundary data $(f,g)$} to be the $\Sigma$-indexed line ensemble $\mathcal{Q}$ with $\Sigma = \llbracket k_1, k_2\rrbracket$ on $\Lambda = [a,b]$ whose law $\mathbb{P}_H^{k_1, k_2, a,b, \vec{x}, \vec{y}, f,g}$ is given in terms of the following Radon-Nikodym derivative (with respect to the free Brownian bridge line ensemble $\mathbb{P}^{k_1, k_2, a,b, \vec{x},\vec{y}}_{free}$):
$$ \frac{d \mathbb{P}_H^{k_1, k_2, a,b, \vec{x}, \vec{y}, f,g}}{d\mathbb{P}^{k_1, k_2, a,b, \vec{x},\vec{y}}_{free}} (\mathcal{Q}_{k_1}, \dots, \mathcal{Q}_{k_2}) = \frac{W_H^{k_1, k_2, a,b, \vec{x}, \vec{y}, f,g} (\mathcal{Q}_{k_1}, \dots, \mathcal{Q}_{k_2}) }{Z_H^{k_1, k_2, a,b, \vec{x}, \vec{y}, f,g} }.$$
Here we call $\mathcal{Q}_{k_1 - 1} = f$, $\mathcal{Q}_{k_2 + 1} = g$ and define the {\em Boltzmann weight}
$$W_H^{k_1, k_2, a,b, \vec{x}, \vec{y}, f,g} (\mathcal{Q}_{k_1}, \dots, \mathcal{Q}_{k_2}) = \exp \left( - \sum_{i = k_1 - 1}^{k_2} \int_a^b H \left( \mathcal{Q}_{i+1}(u) - \mathcal{Q}_i(u)\right)du \right),$$
and the {\em normalizing constant}
\begin{equation}\label{S2E1}
Z_H^{k_1, k_2, a,b, \vec{x}, \vec{y}, f,g} = \mathbb{E}^{k_1, k_2, a,b, \vec{x},\vec{y}}_{free}\left[ W_H^{k_1, k_2, a,b, \vec{x}, \vec{y}, f,g} (\mathcal{Q}_{k_1}, \dots, \mathcal{Q}_{k_2})  \right].
\end{equation}
On the right side of (\ref{S2E1}) we have that the $\mathcal{Q}_{k_1}, \dots, \mathcal{Q}_{k_2}$ are distributed according to $\mathbb{P}^{k_1, k_2, a,b, \vec{x},\vec{y}}_{free}$. Notice that since $H: [-\infty, \infty) \rightarrow [0, \infty)$ is continuous, we have that $W_H^{k_1, k_2, a,b, \vec{x}, \vec{y}, f,g}$ is a continuous function on $C(\Sigma \times \Lambda)$, taking values in $(0,1]$. In particular, $Z_H^{k_1, k_2, a,b, \vec{x}, \vec{y}, f,g}  \in (0,1)$ and the above Radon-Nikodym derivative, and consequently $\mathbb{P}_H^{k_1, k_2, a,b, \vec{x}, \vec{y}, f,g}$, are well-defined. We will write $\mathbb{E}_H^{k_1, k_2, a,b, \vec{x}, \vec{y}, f,g}$ for the expectation with respect to $\mathbb{P}_H^{k_1, k_2, a,b, \vec{x}, \vec{y}, f,g}$.
\end{definition}

The following definition introduces the notion of the partial $H$-Brownian Gibbs property.
\begin{definition}\label{DefPBGP}
Fix a set $\Sigma = \llbracket 1 , N \rrbracket$ with $N \in \mathbb{N}$ or $N  = \infty$ and an interval $\Lambda \subset \mathbb{R}$.  A $\Sigma$-indexed line ensemble $\mathcal{L}$ on $\Lambda$ is said to satisfy the {\em partial $H$-Brownian Gibbs property} if and only if for any finite $K = \{k_1, k_1 + 1, \dots, k_2 \} \subset \Sigma$ with $k_2 \leq N - 1$ (if $\Sigma \neq \mathbb{N}$), $[a,b] \subset \Lambda$ and any bounded Borel-measurable function $F: C(K \times [a,b]) \rightarrow \mathbb{R}$ we have $\mathbb{P}$-almost surely
\begin{equation}\label{PBGPTower}
\mathbb{E} \left[ F(\mathcal{L}|_{K \times [a,b]}) {\big \vert} \mathcal{F}_{ext} (K \times (a,b))  \right] =\mathbb{E}_{H}^{k_1, k_2, a,b, \vec{x}, \vec{y}, f, g} \bigl[ F({\mathcal{Q}}) \bigr],
\end{equation}
where we have $D_{K,a,b} = K \times (a,b)$ and $D_{K,a,b}^c = (\Sigma \times \Lambda) \setminus D_{K,a,b}$,
$$\mathcal{F}_{ext} (K \times (a,b)) = \sigma \left \{ \mathcal{L}_i(s): (i,s) \in D_{K,a,b}^c \right\}$$
is the $\sigma$-algebra generated by the variables in the brackets above, and $ \mathcal{L}|_{K \times [a,b]}$ denotes the restriction of $\mathcal{L}$ to the set $K \times [a,b]$. On the right side of (\ref{PBGPTower}),  $\mathcal{Q} = (\mathcal{Q}_{k_1}, \dots, \mathcal{Q}_{k_2})$ is the $\llbracket k_1, k_2 \rrbracket \times (a,b)$-indexed $H$-Brownian bridge line ensemble with entrance data $\vec{x}$, exit data $\vec{y}$ and boundary data $(f,g)$ as in Definition \ref{DefAvoidingLaw}. Here $\vec{x} = (\mathcal{L}_{k_1}(a), \dots, \mathcal{L}_{k_2}(a))$, $\vec{y} = (\mathcal{L}_{k_1}(b), \dots, \mathcal{L}_{k_2}(b))$, $f = \mathcal{L}_{k_1 - 1}[a,b]$ (the restriction of $\mathcal{L}$ to the set $\{k_1 - 1 \} \times [a,b]$) with the convention that $f = \infty$ if $k_1 - 1 \not \in \Sigma$, and $g = \mathcal{L}_{k_2 +1}[a,b]$. 

If we furthermore have that $N = \infty$ or $N < \infty$ and equation (\ref{PBGPTower}) holds when $k_2 = N$ and $g = -\infty$, then we say that $\Lambda$ satisfies the {\em $H$-Brownian Gibbs property} (i.e. we drop ``partial'').
\end{definition}

\begin{remark}\label{RPBGP}
The $H$-Brownian Gibbs property is the same as the one introduced in \cite{CorHamK}, and the partial $H$-Brownian Gibbs property is related to it but is slightly different as we explain here. Assuming that $\Sigma = \mathbb{N}$ the two definitions are equivalent. However, if $\Sigma = \{1, \dots, N\}$ with $1 \leq N < \infty$, then a line ensemble that satisfies the $H$-Brownian Gibbs property also satisfies the partial $H$-Brownian Gibbs property, but the reverse need not be true. Specifically, the $H$-Brownian Gibbs property allows for the possibility that $k_2 = N$ in Definition \ref{DefPBGP} and in this case the convention is that $g = -\infty$. A distinct advantage of working with the partial $H$-Brownian Gibbs property instead of the $H$-Brownian Gibbs property is that the former is stable under projections, while the latter is not. Specifically, if $1 \leq M \leq N$ and $\mathcal{L}$ is a $\llbracket 1, N\rrbracket$-indexed line ensemble on $\Lambda$ that satisfies the partial $H$-Brownian Gibbs property, and $\tilde{\mathcal{L}}$ is obtained from $\mathcal{L}$ by projecting on $(\mathcal{L}_1, \dots, \mathcal{L}_M)$ then the induced law on $\tilde{\mathcal{L}}$ also satisfies the partial $H$-Brownian Gibbs property as a $\llbracket 1, M \rrbracket$-indexed line ensemble on $\Lambda$. Later in the text some of our arguments rely on an induction on $N$, for which having this projectional stability becomes important. This is why we choose to work with the partial $H$-Brownian Gibbs property instead of the $H$-Brownian Gibbs property.
\end{remark}

\begin{remark} \label{RemMeas}
Observe that if $N = 1$ then the conditions in Definition \ref{DefPBGP} become void. I.e., any line ensemble with one line satisfies the partial $H$-Brownian Gibbs property. 

It is perhaps worth explaining why equation (\ref{PBGPTower}) makes sense. Firstly, since $\Sigma \times \Lambda$ is locally compact, we know by \cite[Lemma 46.4]{Munkres} that $\mathcal{L} \rightarrow \mathcal{L}|_{K \times [a,b]}$ is a continuous map from $C(\Sigma \times \Lambda)$ to $C(K \times [a,b])$, so that the left side of (\ref{PBGPTower}) is the conditional expectation of a bounded measurable function, and is thus well-defined. A more subtle question is why the right side of (\ref{PBGPTower})  is $\mathcal{F}_{ext} (K \times (a,b))$-measurable. In fact we will show in Lemma \ref{LemmaMeasExp} that the right side is measurable with respect to the $\sigma$-algebra 
$$ \sigma \left\{ \mathcal{L}_i(s) : \mbox{  $i \in K$ and $s \in \{a,b\}$, or $i \in \{k_1 - 1, k_2 +1 \}$ and $s \in [a,b]$} \right\}.$$
\end{remark}

%----------------------------------------------------------------------------------------------------------------------------------------------------------------------------------------------------------------------------
%
%     Section 2.2
%
%----------------------------------------------------------------------------------------------------------------------------------------------------------------------------------------------------------------------------
\subsection{Main result}\label{Section2.2} In this section we formulate the main result of the paper. We continue with the same notation as in Section \ref{Section2.1} and additionally introduce the following basic definition.

\begin{definition}\label{S2Halph} Let $\lambda > 0$ be given. We consider the following conditions for a function $H: [-\infty, \infty) \rightarrow [0,\infty)$ 
\begin{enumerate}
\item $H$ is continuous;
\item $H$ is convex;
\item for any $M \in \mathbb{N}$ one has $\limsup_{y \rightarrow \infty} \sup_{x \in [-M, M]} \left| \frac{H(x+y)}{H(y)} - e^{\lambda x} \right| = 0.$
\end{enumerate}
If $H$ satisfies all three conditions above we say that $H$ is a {\em $\lambda$-exponential Hamiltonian}.
\end{definition}

\begin{theorem}\label{ThmMain}  Let $\Sigma = \llbracket 1, N \rrbracket$ with $N \in \mathbb{N}$ or $N =\infty$, and let $\Lambda \subset \mathbb{R}$ be an interval. Fix $\lambda > 0$ and suppose that $H$ is a $\lambda$-exponential Hamiltonian in the sense of Definition \ref{S2Halph}. Suppose that $\mathcal{L}^1$ and $\mathcal{L}^2$ are $\Sigma$-indexed line ensembles on $\Lambda$ that satisfy the partial $H$-Brownian Gibbs property with laws $\mathbb{P}_1$ and $\mathbb{P}_2$ respectively. Suppose further that for every $k\in \mathbb{N}$,  $t_1 < t_2 < \cdots < t_k$ with $t_i \in \Lambda$ and $x_1, \dots, x_k \in \mathbb{R}$ we have that 
\begin{equation*}
\mathbb{P}_1 \left( \mathcal{L}^1_1(t_1) \leq x_1, \dots,\mathcal{L}^1_1(t_k) \leq x_k  \right) =\mathbb{P}_2 \left( \mathcal{L}^2_1(t_1) \leq x_1, \dots,\mathcal{L}^2_1(t_k) \leq x_k  \right).
\end{equation*}
Then we have that $\mathbb{P}^1 = \mathbb{P}^2$.
\end{theorem}
In plain words, Theorem \ref{ThmMain} states that if two line ensembles both satisfy the partial $H$-Brownian Gibbs property and have the same finite-dimensional distributions of the top curve, then they have the same distribution as line ensembles, provided that $H$ is a $\lambda$-exponential Hamiltonian.

One of the assumptions in Theorem \ref{ThmMain} is that $\mathcal{L}^1$ and $\mathcal{L}^2$ have the same number of curves $N$, and a natural question is whether this condition can be relaxed. That is, can two $H$-Brownian Gibbsian line ensembles with a {\em different} number of curves have the same finite-dimensional distributions of the top curve? The answer to this question is negative and we isolate this statement in the following corollary.
\begin{corollary}\label{CorMain2}  Let $\Sigma_1 = \llbracket 1, N_1 \rrbracket$ with $N_1 \in \mathbb{N}$ and $\Sigma_2 = \llbracket 1, N_2 \rrbracket$ with $N_2 \in \mathbb{N}$ or $N_2 = \infty$ such that $N_2 > N_1$. In addition, let $\Lambda \subset \mathbb{R}$ be an interval, fix $\lambda > 0$ and suppose that $H$ is a $\lambda$-exponential Hamiltonian in the sense of Definition \ref{S2Halph}. Suppose that $\mathcal{L}^i$ are $\Sigma_i$-indexed line ensembles on $\Lambda$ for $i = 1,2$ such that $\mathcal{L}^1$ satisfies the $H$-Brownian Gibbs property and $\mathcal{L}^2$ satisfies the partial $H$-Brownian Gibbs property with laws $\mathbb{P}_1$ and $\mathbb{P}_2$, respectively. Then there exist $k\in \mathbb{N}$,  $t_1 < t_2 < \cdots < t_k$ with $t_i \in \Lambda$ and $x_1, \dots, x_k \in \mathbb{R}$ such that
\begin{equation*}
\mathbb{P}_1 \left( \mathcal{L}^1_1(t_1) \leq x_1, \dots,\mathcal{L}^1_1(t_k) \leq x_k  \right) \neq \mathbb{P}_2 \left( \mathcal{L}^2_1(t_1) \leq x_1, \dots,\mathcal{L}^2_1(t_k) \leq x_k  \right).
\end{equation*}
\end{corollary}
\begin{remark} It is important that $\mathcal{L}^1$ satisfies the usual rather than the partial $H$-Brownian Gibbs property in Corollary \ref{CorMain2}. Indeed, otherwise one could take $\mathcal{L}^2$ and project this line ensemble to its top $N_1$ curves. The resulting $\Sigma_1$-indexed line ensemble on $\Lambda$ will have the same top curve distribution as $\mathcal{L}^2$ and also satisfy the partial $H$-Brownian Gibbs property -- see Remark \ref{RPBGP}. In a sense $\mathcal{L}^1$ can be understood as a line ensemble with $N_1+1$ curves with the $(N_1+1)\textsuperscript{st}$ curve sitting at $-\infty$, while $\mathcal{L}^2$ has a $(N_1+1)\textsuperscript{st}$ curve that is finite-valued and the question that Corollary \ref{CorMain2} answers in the affirmative is whether we can distinguish between these two cases using only the top curve of the line ensemble. 
\end{remark}

%----------------------------------------------------------------------------------------------------------------------------------------------------------------------------------------------------------------------------
%
%     Section 2.3
%
%----------------------------------------------------------------------------------------------------------------------------------------------------------------------------------------------------------------------------
\subsection{Basic lemmas}\label{Section2.3} 
In this section we present two lemmas that are related to line ensembles. As their proofs are somewhat technical, we postpone them until Section \ref{Section5.1}. 

\begin{definition}\label{ContFunInf} We endow $(-\infty, \infty]$ and $[-\infty, \infty)$ with the usual topology, i.e. the one for which $\{ (x-1/n, x+1/n)\}_{n \in \mathbb{N}}$ forms a local base at a point $x \in (-\infty, \infty)$, $\{ (n, \infty]\}_{n \in \mathbb{N}}$ is a local base of $\infty$ and $\{ [-\infty, -n)\}_{n \in \mathbb{N}}$ is a local base of $-\infty$. Let $C^+([a,b])$ be the space of continuous functions $f: [a,b] \rightarrow (-\infty, \infty]$ and $C^-([a,b])$ the space of continuous functions on $g: [a,b] \rightarrow [-\infty, \infty)$. We endow $C^+([a,b])$ with the topology where $f_n \rightarrow f$ if and only if $e^{-f_n} \rightarrow e^{-f}$ uniformly on $[a,b]$. Similarly, we endow $C^-([a,b])$ with the topology where $f_n \rightarrow f$ if and only if $e^{f_n} \rightarrow e^{f}$ uniformly on $[a,b]$.
\end{definition}

The following lemma explains why equation (\ref{PBGPTower}) makes sense, see also Remark \ref{RemMeas}.
\begin{lemma}\label{LemmaMeasExp} Assume the same notation as in Definitions \ref{DefAvoidingLaw} and \ref{ContFunInf} and suppose that $h : C( \llbracket k_1, k_2 \rrbracket \times [a,b]) \rightarrow \mathbb{R}$ is a bounded Borel-measurable function. Let $k = k_2 - k_1 + 1$, and
\begin{equation*}
\begin{split}
&S = \{ (\vec{x}, \vec{y}, f,g) \in \mathbb{R}^k \times  \mathbb{R}^k \times C^+([a,b]) \times C^-([a,b]) \}, \\
\end{split}
\end{equation*}
where $S$ is endowed with the product topology and corresponding Borel $\sigma$-algebra. Then the function $G_h: S \rightarrow \mathbb{R}$, given by
\begin{equation}\label{S2MeasExpFun}
\begin{split}
&G_h(\vec{x}, \vec{y}, f,g) := \mathbb{E}_{H}^{k_1, k_2, a,b,\vec{x}, \vec{y},f,g}[h(\mathcal{Q})],
\end{split}
\end{equation}
is bounded and measurable.
\end{lemma}

\begin{lemma}\label{PropFD}  
 Let $\Sigma = \llbracket 1, N \rrbracket$ with $N \in \mathbb{N}$ or $N =\infty$, and let $\Lambda \subset \mathbb{R}$ be an interval. Suppose that $\mathcal{L}^1$ and $\mathcal{L}^2$ are $\Sigma$-indexed line ensembles on $\Lambda$ with laws $\mathbb{P}_1$ and $\mathbb{P}_2$ respectively. Suppose further that for every $k\in \mathbb{N}$,  $ t_1 < t_2 < \cdots < t_k$ with $t_i \in \Lambda^o$ (the interior of $\Lambda$) for $i = 1, \dots, k$; $n_1, \dots, n_k \in \Sigma$ and $z_1, \dots, z_k \in \mathbb{N}$ we have
\begin{equation}\label{S2FDStrict}
\mathbb{E}_{\mathbb{P}_1} \left[ \prod_{ i =1}^k \exp \left( - z_i \exp \left( \mathcal{L}^1_{n_i}(t_i)\right) \right) \right] =\mathbb{E}_{\mathbb{P}_2} \left[ \prod_{ i =1}^k \exp \left( - z_i \exp \left( \mathcal{L}^2_{n_i}(t_i)\right) \right)  \right].
\end{equation}
Then we have that $\mathbb{P}_1 = \mathbb{P}_2$.
\end{lemma}

%----------------------------------------------------------------------------------------------------------------------------------------------------------------------------------------------------------------------------
%
%     Section 3
%
%----------------------------------------------------------------------------------------------------------------------------------------------------------------------------------------------------------------------------
\section{Key lemmas}\label{Section3} In the following section we isolate several technical lemmas, which will be required in the proof of Theorem \ref{ThmMain}. The main result of the section, Lemma \ref{BridgeLemma}, is presented in Section \ref{Section3.1} and is followed by a brief description of the method of proof. The complete proof is given in Section \ref{Section3.2}. In Section \ref{Section3.3} we formulate two results for $H$-Brownian bridge line ensembles, with the first being a generalization of Lemma \ref{BridgeLemma}. In Section \ref{Section3.4} we present two monotone coupling lemmas for $H$-Brownian bridge line ensembles, whose proof is deferred to Section \ref{Section5}.

%-------------------------------------------------------------------------------------------------------------------------------------------------------------------------------------------------
% Section 3.1
%
%-------------------------------------------------------------------------------------------------------------------------------------------------------------------------------------------------
\subsection{Estimates on Brownian bridges}\label{Section3.1} If $W_t$ is a standard Brownian motion, we recall that $\tilde{B}(t) = W_t - t W_1$ for $0 \leq t \leq 1$ is a standard Brownian bridge. Given $a, b,x,y \in \mathbb{R}$ with $a < b$, we recall from (\ref{BBDef}) that
\begin{equation*}
B(t) = (b-a)^{1/2} \cdot \tilde{B} \left( \frac{t - a}{b-a} \right) + \left(\frac{b-t}{b-a} \right) \cdot x + \left( \frac{t- a}{b-a}\right) \cdot y, 
\end{equation*}
is the Brownian bridge (from $B(a) = x$ to $B(b) = y$) with diffusion parameter $1$, and we denote its law by $\mathbb{P}_{free}^{a,b,x,y}$. 

The following lemma is the main result of the section.
\begin{lemma}\label{BridgeLemma} Let us fix $ \lambda, M > 0$, $x,y \in [-M,M]$. We let $a_n, b_n$ be sequences such that $n^{-5/4} \leq b_n \leq n^{-3/4}$ and $a_n = b_n + n^{-2}$ for $n \in \mathbb{N}$. We also set $W_n = \lambda^{-1} \log n$, $V_n = n^{-1/3}$ for $n \in \mathbb{N}$. Let $B_n$ be the Brownian bridge with law $\mathbb{P}_{free}^{-a_n, a_n,x,y}$ and denote the events:
\begin{enumerate}
\item $A_n = \{ -W_n - 1 \leq B_n(t) \leq M + 1 \mbox{ for $-a_n \leq t \leq -b_n $} \}$;
\item $D_n = \{ -W_n - 1 \leq B_n(t) \leq M + 1 \mbox{ for $b_n \leq t \leq a_n$} \}$;
\item $E_n = \{  -W_n - 2V_n \leq B_n(t) \leq -W_n + 2V_n \mbox{ for $-b_n \leq t \leq b_n$} \}$;
\item $F_n = \{B_n(-b_n) \leq -W_n \mbox{ and } B_n(b_n) \leq -W_n \} .$
\end{enumerate}
Then there exist positive constants $C,c$, depending on $\lambda, M$ such that for all $n \geq 1$ one has
\begin{equation}\label{BBEM}
\mathbb{P}_{free}^{-a_n, a_n, x,y} ( A^c_n \cup D^c_n \cup E^c_n | F_n) \leq C e^{-c n^{1/12} }.
\end{equation}
\end{lemma}

Let us explain the content of Lemma \ref{BridgeLemma} and the key ideas of the proof, see also Figure \ref{S3_1}. In the lemma we are given a sequence of Brownian bridges $B_n$ that are defined on small intervals $[-a_n, a_n]$. Inside the intervals $[-a_n, a_n]$ there are slightly smaller intervals $[-b_n, b_n]$ such that $a_n -b_n = n^{-2}$ is tiny compared to $b_n$. The event $F_n$ is the event that the Brownian bridge $B_n$ is lower than $-W_n = - \lambda^{-1} \log n$ both at times $-b_n$ and $b_n$. The starting and ending location of $B_n$ are points $x,y$ which are assumed to lie inside of a fixed window $[-M,M]$. The event $F_n$ is a tail event and very unlikely for large $n$ as the Brownian bridge is required to go very low in a tiny amount of time from $-a_n$ to $-b_n$ and again from time $a_n$ backwards to $b_n$. The main result of the lemma, equation (\ref{BBEM}), says that all three events $A_n$, $D_n$ and $E_n$ are extremely likely to occur if the event $F_n$ occurs.
\begin{figure}[ht]
\begin{center}
  \includegraphics[scale = 0.7]{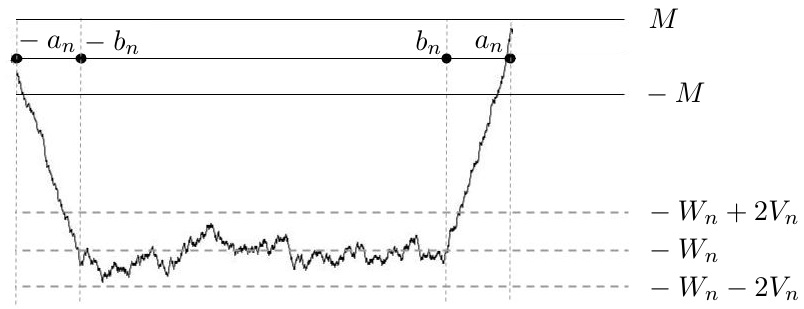}
  \caption{The figure shows $B_n$ for an outcome in the event $A_n \cap D_n \cap E_n \cap F_n$. }
  \label{S3_1}
  \end{center}
\end{figure}

To see why $A_n, D_n$ and $E_n$ are likely given $F_n$, let us consider the event $G_n = G_n^+ \cup G_n^-$, where 
$$G^+_n = \{ B_n(b_n) \leq -W_n - V_n \} \mbox{ and } G^-_n =\{ B_n(-b_n) \leq -W_n - V_n \}.$$
On the event $F_n$ one can show that the event $G_n$ is extremely unlikely. The reason behind this is that on $F_n$ the bridge $B_n$ is already forced to go below $-W_n$ at the times $\pm b_n$, while on $G_n$ it is forced at one of these two times to go even lower, namely below $-W_n - V_n$. Even though the bridge is forced to go only a small additional distance lower, this extra bit of deviation of the bridge from its mean comes at a great energetic cost making it unlikely. Once we have shown that $G_n$ is unlikely given $F_n$, it suffices to show that $A_n, D_n, E_n$ are all very likely conditional on $F_n \cap G_n^c$. 

Conditional on the values $B_n(-b_n)$ and $B_n(b_n)$ the law of the three pieces of $B_n$ on $[-a_n, -b_n]$, $[-b_n, b_n]$ and $[b_n, a_n]$ becomes the law of three independent Brownian bridges from $x$ to $B_n(-b_n)$, from $B_n(-b_n)$ to $B_n(b_n)$ and from $B_n(b_n)$ to $y$, respectively, all with diffusion parameters $1$. On the event $F_n \cap G_n^c$ we have that 
$$-W_n \geq B_n(-b_n), B_n(b_n) \geq - W_n- V_n,$$
and the bridge from $x$ to $B_n(-b_n)$ closely follows the segment connecting the points $(-a_n, x)$ and $(-b_n, B(-b_n))$. By our assumption that $x \in [-M, M]$ and $-W_n  \geq B_n(-b_n) \geq -W_n - V_n$ (which holds on  $F_n \cap G_n^c$) we have that this segment lies well below the horizontal level $M+1$ and above the level $-W_n - 1$. This means that on $F_n \cap G_n^c$ the bridge from $x$ to $B_n(-b_n)$ is unlikely to dip below $-W_n-1$ and to rise above $M+ 1$, making $A_n$ very likely. An analogous argument shows that $D_n$ is also very likely on $F_n \cap G_n^c$. 

Finally, the bridge from $B_n(-b_n)$ to $B_n(b_n)$ on the event $F_n \cap G_n^c$ starts and ends between the levels $-W_n$ and $-W_n - V_n$. Since this bridge is on the interval $[-b_n, b_n]$, it has fluctuations of order $\sqrt{b_n}$, which means that it is unlikely to deviate too much from the straight segment connecting $(-b_n, B_n(-b_n))$ and $(b_n, B_n(b_n))$, which makes the event $E_n$ very likely conditional on $F_n \cap G_n^c$. 

The above few paragraphs sketch the main steps of the proof of Lemma \ref{BridgeLemma}, and in Section \ref{Section3.2} we supply the necessary details to make the argument complete.\\

We end this section by recalling a few basic results about Gaussian distributions and Brownian bridges. We denote by $\Phi(x)$ and $\phi(x)$ the cumulative distribution function and density of a standard normal random variable. The following lemma can be found in \cite[Section 4.2]{MZ}. 
\begin{lemma}\label{LemmaI1}
There is a constant $c_0 > 1$ such that for all $x \geq 0$ we have
\begin{equation}\label{LI2}
 \frac{1}{c_0(1+x)} \leq \frac{1 - \Phi(x)}{\phi(x)} \leq \frac{c_0}{1 +x},
\end{equation}
\end{lemma}

The following result can be found in \cite[Chapter 4, Eq.~3.40]{KS}. 
\begin{lemma}\label{LemmaBBmax} Let $a \in \mathbb{R}$, $T > 0$ and $\beta > \max(a, 0)$. Let $B: [0,T] \rightarrow \mathbb{R}$ denote a Brownian bridge from $B(0) = 0$ to $B(T) = a$ with diffusion parameter $1$. Then we have
\begin{equation*}
\mathbb{P}^{0,T,0,a}_{free}\left( \max_{0 \leq t \leq T} B(t) \geq \beta \right) = \mathbb{P}^{0,T,0,-a}_{free}\left( \min_{0 \leq t \leq T} B(t) \leq -\beta \right)  = e^{-2\beta (\beta -a)/T}.
\end{equation*}
\end{lemma}

%-------------------------------------------------------------------------------------------------------------------------------------------------------------------------------------------------
% Section 3.2
%
%-------------------------------------------------------------------------------------------------------------------------------------------------------------------------------------------------
\subsection{Proof of Lemma \ref{BridgeLemma}}\label{Section3.2} 
Throughout this section we will denote $\mathbb{P}_{free}^{-a_n, a_n, x,y}$ by $\mathbb{P}$ for simplicity. We also adopt the same notation as in the statement of the lemma. Let 
$$ G_n = G_n^+ \cup G_n^- \mbox{,  where }G^+_n = \{ B_n(b_n) \leq -W_n - V_n \} \mbox{ and } G^-_n =\{ B_n(-b_n) \leq -W_n - V_n \}.$$

We will show that there exist positive constants $C_i, c_i > 0$ for $i = 1,\dots, 4$ such that
\begin{equation}\label{BBE1}
\frac{\mathbb{P}(G_n \cap F_n)}{\mathbb{P}( F_n)} \leq C_1 e^{-c_1 n^{1/12} },
\end{equation}
\begin{equation}\label{BBE2}
\frac{\mathbb{P}(A^c_n \cap G^c_n \cap F_n)}{\mathbb{P}( F_n)} \leq C_2 e^{-c_2 n^{1/12} }, \hspace{5mm} \frac{\mathbb{P}(D^c_n \cap G^c_n \cap F_n)}{\mathbb{P}( F_n)} \leq C_3 e^{-c_3 n^{1/12} },
\end{equation}
\begin{equation}\label{BBE4}
\frac{\mathbb{P}(E^c_n \cap G^c_n \cap F_n)}{\mathbb{P}( F_n)} \leq C_4 e^{-c_4 n^{1/12} }.
\end{equation}
If equations (\ref{BBE1}), (\ref{BBE2}) and (\ref{BBE4}) are all true we conclude that 
\begin{equation*}
\mathbb{P}( A^c_n \cup D^c_n \cup E^c_n |F_n) = \frac{\mathbb{P}\left( (A^c_n \cup D^c_n \cup E^c_n) \cap F_n \right)}{\mathbb{P}( F_n)} \leq \sum_{i = 1}^4 C_i e^{-c_i n^{1/12} },
\end{equation*}
which implies (\ref{BBEM}) with $C = C_1 + C_2 +C_3 + C_4$ and $c = \min(c_1,c_2,c_3, c_4)$. 

We have thus reduced the proof of the lemma to showing that (\ref{BBE1}), (\ref{BBE2}) and (\ref{BBE4}) all hold, which we do in the steps below.\\

{\bf \raggedleft Step 1.} In this step we prove (\ref{BBE2}). As the two inequalities are quite similar we will only establish the first.
 By \cite[Eq.~6.28, p. 359]{KS} we know that the random vector $(B_n(-b_n), B_n(b_n)) \in \mathbb{R}^2$ has the following density function
\begin{equation}\label{BBDensity2}
 \frac{p(a_n - b_n; x, x_1) p(2b_n; x_1, x_2) p(a_n - b_n; x_2, y)}{p(2a_n;x,y)}, \mbox{ where } p(t; x,y) = \frac{e^{-(x-y)^2/2t}}{\sqrt{2\pi t}},
\end{equation}
and moreover conditional on $B_n(-b_n) = x_1$ and $B_n(b_n) = x_2$ the law of $B_n$ on the three intervals $[-a_n, -b_n]$, $[-b_n, b_n]$ and $[b_n, a_n]$ becomes the law of three independent Brownian bridges from $x$ to $x_1$, from $x_1$ to $x_2$ and from $x_2$ to $y$, respectively, all with diffusion parameters $1$. Let us denote 
$$h_n(x,x_1) = \mathbb{P}_{free}^{-a_n, -b_n, x, x_1} \left( -W_n - 1 \leq Q(t) \leq M + 1 \mbox{ for $-a_n \leq t \leq -b_n $} \right),$$
where $Q$ has law $\mathbb{P}_{free}^{-a_n, -b_n, x, x_1}.$ Then we have
\begin{equation}\label{BBE2R1}
\begin{split}
&\frac{\mathbb{P}(A^c_n \cap G^c_n \cap F_n)}{\mathbb{P}( F_n)} = \\
& \frac{\int^{-W_n}_{-W_n -V_n} \int^{-W_n}_{-W_n-V_n} [1 - h_n(x,x_1)]p(a_n - b_n; x, x_1) p(2b_n; x_1, x_2) p(a_n - b_n; x_2, y) dx_1 dx_2}{\int^{-W_n}_{-\infty} \int^{-W_n}_{-\infty}p(a_n - b_n; x, x_1) p(2b_n; x_1, x_2) p(a_n - b_n; x_2, y) dx_1 dx_2}.
\end{split}
\end{equation}
 From equation (\ref{BBE2R1}) we see that to prove the first inequality in (\ref{BBE2}) it suffices to show that there are constants $C_2, c_2 > 0$ such that for $x \in [-M, M]$ and $x_1 \in [-W_n - V_n, -W_n ]$ we have
\begin{equation}\label{BBE2R2}
\begin{split}
1 - h_n(x,x_1) \leq C_2 e^{-c_2 n^{1/12} } .
\end{split}
\end{equation}

We observe that we have the following tower of inequalities for $x \in [-M, M]$, $x_1 \in [-W_n - V_n, -W_n]$ and $W_n \geq M$
\begin{equation}\label{BBE2R3}
\begin{split}
&1 - h_n(x,x_1) \leq \\
& \mathbb{P}_{free}^{-a_n, -b_n, x, x_1} \left( \max_{t \in [-a_n, -b_n]} Q(t) > M+ 1 \right) + \mathbb{P}_{free}^{-a_n, -b_n, x, x_1} \left( \min_{t \in [-a_n, -b_n]} Q(t) < -W_n-1\right)  = \\
& e^{ - \frac{2(M + 1 - x)(M + 1 -x_1)}{2[ a_n - b_n]} }+ e^{ - \frac{2(  W_n + 1 + x)(W_n +1 + x_1)}{2[a_n - b_n]}} \leq e^{ - n^{2}} + e^{ - n^{2}( 1 -V_n) } . \\
\end{split}
\end{equation}
Let us briefly elaborate on (\ref{BBE2R3}). The first inequality follows by subadditivity and the definition of $h_n(x,x_1)$. The equality on the second line follows from Lemma \ref{LemmaBBmax}, and the last inequality follows from the fact that $a_n - b_n = n^{-2}$ and our assumptions $x \in [-M, M]$, $x_1 \in [-W_n-V_n, -W_n]$, $W_n \geq M$. Using that $V_n = n^{-1/3}$, we see that (\ref{BBE2R3}) implies (\ref{BBE2R2}), which concludes the proof of first inequality in (\ref{BBE2}). The second one is handled analogously.\\

{\bf \raggedleft Step 2.} In this step we prove (\ref{BBE4}). As explained in Step 1, we have that $(B_n(-b_n), B_n(b_n))$ has density (\ref{BBDensity2}) and conditional on $B_n(-b_n) = x_1$ and $B_n(b_n) = x_2$ the law of $B_n$ on the three intervals $[-a_n, -b_n]$, $[-b_n, b_n]$ and $[b_n, a_n]$ becomes the law of three independent Brownian bridges from $x$ to $x_1$, from $x_1$ to $x_2$ and from $x_2$ to $y$, respectively, all with diffusion parameters $1$. Let us denote 
$$g_n(x_1,x_2) = \mathbb{P}_{free}^{-b_n, b_n, x_1, x_2} \left(-W_n - 2V_n \leq Q(t) \leq -W_n + 2 V_n \mbox{ for $-b_n \leq t \leq b_n $} \right),$$
where $Q$ has law $\mathbb{P}_{free}^{-b_n, b_n, x_1, x_2} .$ Then we have
\begin{equation}\label{BBE4R1}
\begin{split}
&\frac{\mathbb{P}(E^c_n \cap G^c_n \cap F_n)}{\mathbb{P}( F_n)} = \\
& \frac{\int^{-W_n}_{-W_n - V_n} \int^{-W_n}_{-W_n-V_n} [1 - g_n(x_1,x_2)]p(a_n - b_n; x, x_1) p(2b_n; x_1, x_2) p(a_n - b_n; x_2, y) dx_1 dx_2}{\int^{-W_n}_{-\infty} \int^{-W_n}_{-\infty}p(a_n - b_n; x, x_1) p(2b_n; x_1, x_2) p(a_n - b_n; x_2, y) dx_1 dx_2}.
\end{split}
\end{equation}
 From equation (\ref{BBE4R1}) we see that to prove (\ref{BBE4}) it suffices to show that there are constants $C_4, c_4 > 0$ such that for $x_1, x_2 \in [-W_n - V_n, -W_n]$ we have
\begin{equation}\label{BBE4R2}
\begin{split}
1 - g_n(x_1,x_2) \leq C_4 e^{-c_4 n^{1/12} }.
\end{split}
\end{equation}

We observe that we have the following tower of inequalities for $x_1, x_2 \in [-W_n-V_n, -W_n]$
\begin{equation}\label{BBE4R3}
\begin{split}
&1 - g_n(x_1,x_2) \leq \mathbb{P}_{free}^{-b_n, b_n, x_1, x_2} \left( \max_{t \in [-b_n, b_n]} Q(t) > -W_n+ 2V_n \right) + \\
&  \mathbb{P}_{free}^{-b_n, b_n, x_1, x_2} \left( \inf_{t \in [-b_n, b_n]} Q(t) < -W_n - 2V_n \right)  = \\
& e^{ - \frac{2(-W_n + 2V_n - x_1)(-W_n +2V_n - x_2)}{2b_n} }+ e^{ - \frac{2( x_1 + W_n + 2V_n )(x_2 + W_n + 2V_n)}{2b_n}} \leq 2 e^{-V_n^2/b_n}. \\
&
\end{split}
\end{equation}
Let us briefly elaborate on (\ref{BBE4R3}). The first inequality follows by subadditivity and the definition of $g_n(x_1,x_2)$. The equality on the second line follows from Lemma \ref{LemmaBBmax}, and the last inequality follows from our assumption that $x_1, x_2 \in [-W_n-V_n, -W_n]$. Using that $n^{-5/4} \leq b_n \leq  n^{-3/4}$ and $V_n = n^{-1/3}$, we see that (\ref{BBE4R3}) implies (\ref{BBE4R2}), which concludes the proof of (\ref{BBE4}).\\

{\bf \raggedleft Step 3.} In this step we prove (\ref{BBE1}). As explained in Step 1, we have that $(B_n(-b_n), B_n(b_n))$ has density (\ref{BBDensity2}).
Combining the latter with $\mathbb{P}( G_n \cap F_n) \leq \mathbb{P}( G^+_n \cap F_n) + \mathbb{P}( G^-_n \cap F_n)$ we conclude
\begin{equation}\label{BBE1R1}
\begin{split}
\frac{\mathbb{P}(G_n \cap F_n)}{\mathbb{P}( F_n)} \leq &\frac{\int^{-W_n}_{-\infty} \int^{-W_n - V_n}_{-\infty}p(a_n - b_n; x, x_1) p(2b_n; x_1, x_2) p(a_n - b_n; x_2, y) dx_1 dx_2}{\int^{-W_n}_{-\infty} \int^{-W_n}_{-\infty}p(a_n - b_n; x, x_1) p(2b_n; x_1, x_2) p(a_n - b_n; x_2, y) dx_1 dx_2} + \\
& \frac{\int^{-W_n - V_n}_{-\infty} \int^{-W_n}_{-\infty }p(a_n - b_n; x, x_1) p(2b_n; x_1, x_2) p(a_n - b_n; x_2, y) dx_1 dx_2}{\int^{-W_n}_{-\infty} \int^{-W_n}_{-\infty}p(a_n - b_n; x, x_1) p(2b_n; x_1, x_2) p(a_n - b_n; x_2, y) dx_1 dx_2}.
\end{split}
\end{equation}
We now proceed to find a lower bound for the denominators on the right side of (\ref{BBE1R1}) and upper bounds for the numerators.
\vspace{1mm}

We first note that if $W_n \geq M$ and $x_1 \in [-W_n-V_n/10, -W_n]$ we have
\begin{equation}\label{BBE1R3}
p(a_n - b_n; x, x_1)  \geq \frac{\exp \left( - (W_n  + V_n/10 +  x)^2/[ 2 (a_n -b_n)] \right)}{\sqrt{2\pi (a_n - b_n)}}.
\end{equation}
Analogously, if $W_n \geq M$ and $x_2 \in [-W_n-V_n/10, -W_n]$ we have
\begin{equation}\label{BBE1R4}
p(a_n - b_n; x_2, y)  \geq \frac{\exp \left( - (W_n  + V_n/10 +  y)^2/[ 2 (a_n -b_n)] \right)}{\sqrt{2\pi (a_n - b_n)}}.
\end{equation}
We also have for $x_1, x_2 \in [-W_n - V_n/10, -W_n ]$ that 
\begin{equation}\label{BBE1R5}
\begin{split}
p(2b_n; x_1, x_2) \geq \frac{e^{-(V_n/5)^2/2b_n}}{\sqrt{4\pi b_n} }.
\end{split}
\end{equation}
Combining (\ref{BBE1R3}), (\ref{BBE1R4}) and (\ref{BBE1R5}) we conclude that 
\begin{equation}\label{BBE1R6}
\begin{split}
&\int^{-W_n}_{-\infty} \int^{-W_n}_{-\infty}p(a_n - b_n; x, x_1) p(2b_n; x_1, x_2) p(a_n - b_n; x_2, y) dx_1 dx_2 \geq \frac{V_n^2}{100} \cdot \frac{e^{-(V_n/5)^2/2b_n}}{\sqrt{4\pi b_n} } \cdot \\
&\frac{\exp \left( - (W_n + V_n/10 +  x)^2/[ 2 (a_n -b_n)] \right)}{\sqrt{2\pi (a_n - b_n)}} \cdot \frac{\exp \left( - (W_n  + V_n/10 +  y)^2/[ 2 (a_n -b_n)] \right)}{\sqrt{2\pi (a_n - b_n)}}.
\end{split}
\end{equation}
This gives us a lower bound for the denominators on the right side of (\ref{BBE1R1}).
\vspace{1mm}

We now proceed to obtain upper bounds for the numerators on the right side of (\ref{BBE1R1}). Since
\begin{equation*}
\begin{split}
p(2b_n; x_1, x_2) \leq \frac{1}{\sqrt{4\pi b_n} },
\end{split}
\end{equation*}
we obtain for $W_n \geq M$ that
\begin{equation}\label{BBE1R8}
\begin{split}
&\int^{-W_n}_{-\infty} \int^{-W_n - V_n}_{-\infty }p(a_n - b_n; x, x_1) p(2b_n; x_1, x_2) p(a_n - b_n; x_2, y) dx_1 dx_2 \leq \\
& \frac{1}{\sqrt{4\pi b_n} } \cdot \int^{-W_n}_{-\infty}p(a_n - b_n; x, x_1) dx_1 \cdot \int^{-W_n - V_n}_{-\infty } p(a_n - b_n; x_2, y)  dx_2 = \\
&\frac{1}{\sqrt{4\pi b_n} } \cdot \left[1 - \Phi \left( \frac{W_n + x}{\sqrt{a_n -b_n}} \right) \right]\left[1 - \Phi \left( \frac{W_n +V_n +y}{\sqrt{a_n -b_n}} \right) \right] \leq \\
&\frac{1}{\sqrt{4\pi b_n} }  \cdot \frac{c_0^2}{\left(1 + \frac{W_n - M}{\sqrt{a_n -b_n}} \right)^2} \cdot  \phi \left(  \frac{W_n + x}{\sqrt{a_n -b_n}} \right) \cdot \phi \left(  \frac{W_n + V_n + y}{\sqrt{a_n -b_n}} \right),
\end{split}
\end{equation}
where in the last inequality we used Lemma \ref{LemmaI1} ($c_0$ is as in that lemma) as well as the fact that $x,y \in [-M, M]$. One analogously obtains that if $W_n \geq M$ then
\begin{equation}\label{BBE1R9}
\begin{split}
&\int^{-W_n - V_n}_{-\infty} \int^{-W_n}_{-\infty }p(a_n - b_n; x, x_1) p(2b_n; x_1, x_2) p(a_n - b_n; x_2, y) dx_1 dx_2 \leq \\
&\frac{1}{\sqrt{4\pi b_n} }  \cdot \frac{c_0^2}{\left(1 + \frac{W_n - M}{\sqrt{a_n -b_n}} \right)^2} \cdot  \phi \left(  \frac{W_n +V_n + x}{\sqrt{a_n -b_n}} \right) \cdot \phi \left(  \frac{W_n + y}{\sqrt{a_n -b_n}} \right).
\end{split}
\end{equation}

Combining (\ref{BBE1R8}), (\ref{BBE1R9}) with (\ref{BBE1R6}) and (\ref{BBE1R1}) we obtain for $W_n \geq 4 M$ that
\begin{equation*}
\begin{split}
&\frac{\mathbb{P}(G_n \cap F_n)}{\mathbb{P}( F_n)} \leq \frac{100}{V_n^2} \cdot e^{V_n^2/50 b_n}  \cdot \frac{2\pi (a_n -b_n) c_0^2}{\left(1 + \frac{W_n - M}{\sqrt{a_n -b_n}} \right)^2} \cdot \exp \left(  \frac{(W_n +V_n/10 +  x)^2 + (W_n +V_n/10 + y)^2}{ 2 (a_n -b_n)}  \right) \cdot  \\
& \Bigg{[} \exp \left( - \frac{(W_n + x)^2 + (W_n +V_n +  y)^2}{ 2 (a_n -b_n)}  \right) + \exp \left(  -\frac{(W_n + V_n +  x)^2 + (W_n + y)^2}{ 2 (a_n -b_n)}  \right)\Bigg{]}  = \\
&\frac{100}{V_n^2} \cdot e^{V_n^2/50 b_n}  \cdot \frac{2\pi (a_n -b_n) c_0^2}{\left(1 + \frac{W_n - M}{\sqrt{a_n -b_n}} \right)^2} \cdot \Bigg{[} \exp \left(- \frac{0.98 V_n^2 + V_n[ 1.6 W_n - 0.2 x  + 1.8y]}{2(a_n - b_n)}\right)  +  \\
& \exp \left(- \frac{0.98 V_n^2 + V_n[ 1.6 W_n - 0.2 y  + 1.8x]}{2(a_n - b_n)}\right) \Bigg{]} \leq \\
&  \frac{200}{V_n^2} \cdot e^{V_n^2/50 b_n}  \cdot \frac{2\pi (a_n -b_n) c_0^2}{\left(1 + \frac{W_n - M}{\sqrt{a_n -b_n}} \right)^2} \cdot \exp \left(- \frac{0.98 V_n^2 + V_n W_n }{2(a_n - b_n)}\right).
\end{split}
\end{equation*}
In the last inequality we used that $0.6 W_n -2M \geq 0$ and $x,y \in [-M,M]$. 

Since $V_n = n^{-1/3}$, $W_n =\lambda^{-1} \log n$ and $a_n - b_n = n^{-2}$, while $n^{-5/4} \leq b_n \leq n^{-3/4}$ we see that the last inequality implies that for some large enough constant $C_1$ and all $n \in \mathbb{N}$ we have 
$$\frac{\mathbb{P}(G_n \cap F_n)}{\mathbb{P}( F_n)} \leq C_1 \exp \left( - (1/2 \lambda) n^{5/3} \log n \right),$$
which certainly implies (\ref{BBE1}). \\

%-------------------------------------------------------------------------------------------------------------------------------------------------------------------------------------------------
% Section 3.3
%
%-------------------------------------------------------------------------------------------------------------------------------------------------------------------------------------------------
\subsection{Estimates on $H$-Brownian bridge line ensembles}\label{Section3.3} In this section we present two results for $H$-Brownian bridge line ensembles as in Definition \ref{DefAvoidingLaw}. The first is an extension of Lemma \ref{BridgeLemma} to line ensembles with a $\lambda$-exponential Hamiltonian $H$ as in Definition \ref{S2Halph}, while the second is a simple convergence result for a special observable of an $H$-Brownian bridge ensemble over an asymptotically shrinking interval.

\begin{corollary}\label{BridgeLemmaV2} Fix $\lambda >0$ and suppose that $H$ is a $\lambda$-exponential Hamiltonian as in Definition \ref{S2Halph}. In addition, fix $k \in \mathbb{N}$, $M > 0$, $\vec{x},\vec{y} \in [-M, M]^k$. We let $a_n, b_n$ be sequences such that $n^{-5/4} \leq b_n \leq n^{-3/4}$ and $a_n = b_n + n^{-2}$ for $n \in \mathbb{N}$. We also set $W_n = \lambda^{-1} \log n$, $V_n = n^{-1/3}$ for $n \in \mathbb{N}$. Let $\mathcal{Q}^n$ be the $\llbracket 1, k \rrbracket \times (-a_n, a_n)$-indexed $H$-Brownian bridge line ensemble with law $\mathbb{P}_{H}^{1,k, -a_n, a_n,\vec{x},\vec{y}, \infty, -\infty}$ as in Definition \ref{DefAvoidingLaw}. and denote the events:
\begin{enumerate}
\item $A_n = \{ -W_n - 1 \leq \mathcal{Q}^n_i(t) \leq M + 1 \mbox{ for $-a_n \leq t \leq -b_n $ and $i = 1,\dots, k$} \}$;
\item $D_n = \{ -W_n - 1 \leq \mathcal{Q}^n_i(t) \leq M + 1 \mbox{ for $b_n \leq t \leq a_n$ and $i = 1,\dots, k$} \}$;
\item $E_n = \{  -W_n - 2V_n \leq \mathcal{Q}^n_i(t) \leq -W_n + 2V_n \mbox{ for $-b_n \leq t \leq b_n$ and $i = 1,\dots, k$} \}$;
\item $F_n = \{\mathcal{Q}^n_i(-b_n) \leq -W_n \mbox{ and } \mathcal{Q}^n_i(b_n) \leq -W_n \mbox{ for $i = 1,\dots, k$} \} .$
\end{enumerate}
Then there exist positive constants $C,c$, depending on $ k, M$ and $H$ such that for all $n \geq 1$ 
\begin{equation}\label{BBEMV2}
\mathbb{P}_{H}^{1,k, -a_n, a_n,\vec{x},\vec{y}, \infty, -\infty} ( A^c_n \cup D^c_n \cup E^c_n | F_n) \leq C e^{-c n^{1/12} }.
\end{equation}
\end{corollary}
\begin{proof} The idea of the proof is to express the probability on the left side of (\ref{BBEMV2}) in terms of free Brownian bridges and then apply Lemma \ref{BridgeLemma}. In the remainder, we write $\mathbb{P}_{H}^{n}$ in place of $\mathbb{P}_{H}^{1,k, -a_n, a_n,\vec{x},\vec{y}, \infty, -\infty}$ to ease the notation. For clarity we split the proof into two steps.\\

{\bf \raggedleft Step 1.} Let $(\Omega, \mathcal{F}, \mathbb{P})$ denote a probability space that supports $k$ i.i.d. Brownian bridges $\{ \tilde{B}_i \}_{i = 1}^{k}$ (from $\tilde{B}_i(0) = 0$ to $\tilde{B}_i(1) = 0$) with diffusion parameter $1$. We denote the expectation operator with respect to $\mathbb{P}$ by $\mathbb{E}$. We also define for $n \in \mathbb{N}$ and $i = 1, \dots, k$ the bridges on $[-a_n, a_n]$
\begin{equation*}
B_i^n(t) = (2a_n)^{1/2} \cdot \tilde{B}_i \left( \frac{t + a_n}{2a_n} \right) + \left(\frac{a_n-t}{2a_n} \right) \cdot x_{i} + \left( \frac{t+ a_n}{2a_n}\right) \cdot y_{i }.
\end{equation*} 
Let $\tilde{A}_n, \tilde{D}_n, \tilde{E}_n$ and $\tilde{F}_n$ be the sets in $\mathcal{F}$ that have the same form as $A_n, D_n, E_n, F_n$ as in the statement of the lemma but with $\mathcal{Q}^n_i$ replaced with $B_i^n$. We then have from Definition \ref{DefAvoidingLaw} that 

\begin{equation}\label{S33E1}
\begin{split}
&\mathbb{P}_H^n ( A^c_n \cup D^c_n \cup E^c_n | F_n) = \frac{\mathbb{P}_H^n ( (A^c_n \cup D^c_n \cup E^c_n) \cap F_n) }{\mathbb{P}_H^n ( F_n)} = \frac{\mathbb{E}[ {W}^n_H(\mathcal{B}^n) \cdot {\bf 1}_{\tilde{F}_n} \cdot {\bf 1}_{\tilde A^c_n \cup \tilde D^c_n \cup \tilde E^c_n}]}{\mathbb{E}[ {W}^n_H(\mathcal{B}^n) \cdot {\bf 1}_{\tilde{F}_n}]},
\end{split}
\end{equation}
where $\mathcal{B}^n = (B_1^n, \dots, B_k^n)$ is the $\llbracket 1, k \rrbracket$-indexed line ensemble on $[-a_n, a_n]$ formed by the Brownian bridges and 
$$W_H^n(\mathcal{B}^n) = \exp \left( - \sum\limits_{i = 1}^{k - 1} \int_{-a_n}^{a_n} H \left( B^n_{i+1}(u) - B^n_i(u)\right)du \right).$$

We claim that we can find a constant $\delta \in (0,1)$ (depending on $k, M$ and $H$) such that for all $n \in \mathbb{N}$ we have on $\tilde{A}_n \cap \tilde{D}_n \cap \tilde{E}_n$ that
\begin{equation}\label{S33E2}
{W}^n_H(\mathcal{B}^n) \geq \delta.
\end{equation}
We will prove (\ref{S33E2}) in the next step. Here we assume its validity and conclude the proof of (\ref{BBEMV2}). \\

Using (\ref{S33E1}) and (\ref{S33E2}) and the fact that $W^n_H(\cdot) \in [0,1]$ (as $H(\cdot) \in [0, \infty)$) we see that 
\begin{equation*}
\mathbb{P}_H^n ( \hspace{-0.25mm} A^c_n \cup D^c_n \cup E^c_n | F_n)  \leq \frac{\mathbb{P}(\tilde{F}_n \hspace{-0.25mm}\cap \hspace{-0.25mm} (\tilde A^c_n \cup \tilde D^c_n \cup \tilde E^c_n ) )}{\delta \mathbb{P}(\tilde{F}_n \hspace{-0.25mm} \cap\hspace{-0.25mm} \tilde A_n \hspace{-0.25mm} \cap\hspace{-0.25mm} \tilde D_n \hspace{-0.25mm} \cap \tilde E_n ) } = \frac{\mathbb{P}(\tilde{F}_n \hspace{-0.25mm} \cap\hspace{-0.25mm} (\tilde A^c_n \hspace{-0.25mm} \cup\hspace{-0.25mm}  \tilde D^c_n \hspace{-0.25mm}\cup\hspace{-0.25mm} \tilde E^c_n ) )}{\delta \mathbb{P}(\tilde{F}_n  )} \cdot \frac{\mathbb{P}(\tilde{F}_n  )}{ \mathbb{P}(\tilde{F}_n \cap \hspace{-0.25mm} \tilde A_n \cap\hspace{-0.25mm} \tilde D_n \hspace{-0.25mm} \cap \hspace{-0.25mm}\tilde E_n ) }.
\end{equation*}
On the other hand, from Lemma \ref{BridgeLemma} we know that there are constants $C_1, c_1 > 0$ (depending on $k, \lambda$ and $M$) such that for all $n \in \mathbb{N}$
$$\frac{\mathbb{P}(\tilde{F}_n \cap (\tilde A^c_n \cup \tilde D^c_n \cup \tilde E^c_n ) )}{ \mathbb{P}(\tilde{F}_n  ) } \leq C_1 e^{-c_1 n^{1/12}}.$$
The last two inequalities imply that we can find $C,c > 0$ that satisfy (\ref{BBEMV2}). \\

{\bf \raggedleft Step 2.} In this step we prove (\ref{S33E2}). By our assumption that $H$ is a $\lambda$-exponential Hamiltonian we know that there exists $R_1 > 0$ such that for $x \geq R_1$ we have 
\begin{equation*}
H(x) \leq  e^{3\lambda x / 2}.
\end{equation*}
From the fact that $H:[-\infty, \infty) \rightarrow [0, \infty)$ is continuous we conclude that we can find a constant $R_2 > 0$ such that for all $x \in [-\infty, \infty)$ we have
\begin{equation}\label{S33E2R2}
0 \leq H(x) \leq R_2 + e^{3\lambda x /2}.
\end{equation}
We observe that we have the following tower of inequalities for each $i = 1,  \dots, k-1$ on the event $\tilde{A}_n \cap \tilde{D}_n \cap \tilde{E}_n$
\begin{equation*}
\begin{split}
 &\int_{-a_n}^{a_n} \hspace{-1mm} H \left( B^n_{i+1}(u) - B^n_i(u)\right)du \leq  2 (a_n - b_n)  \hspace{-1mm} \left(  \hspace{-1mm} R_2 + e^{3\lambda (W_n + M + 2)/ 2} \right) +  \int_{-b_n}^{b_n} \hspace{-3mm} H \left( B^n_{i+1}(u) - B^n_i(u)\right)du = \\
& 2 n^{-2}  \left(R_2 + n^{3/2}e^{3 \lambda (M + 2)/2} \right) +  \int_{-b_n}^{b_n}  H \left( B^n_{i+1}(u) - B^n_i(u)\right)du \leq  \\
& O(n^{-1/2}) + 2b_n \left( R_2 + e^{6 \lambda V_n} \right) = O(n^{-1/2}).
\end{split}
\end{equation*}
where the constants in the big $O$ notations depend on $\lambda$ and $R_2$. Let us elaborate on the latter equation briefly. In the first inequality we used that on the event $\tilde{A}_n \cap \tilde{D}_n$ we have $B_i^n(u), B_{i+1}^n(u) \in [-W_n - 1, M + 1]$ for $u \in [-a_n, -b_n] \cup [b_n, a_n]$ and equation (\ref{S33E2R2}). In going from the first to the second line we used that $a_n -b_n = n^{-2}$ and $W_n = \lambda^{-1} \log n$. In going from the second to the third line we used that on the event $\tilde{E}_n$ we have $B_i^n(u), B_{i+1}^n(u) \in [-W_n - 2V_n, -W_n + 2V_n]$ for $u \in [-b_n, b_n]$ and equation (\ref{S33E2R2}). The last equality used that $b_n \leq n^{-3/4}$ and $V_n = n^{-1/3}$. 

The above work shows that on the event $\tilde{A}_n \cap \tilde{D}_n \cap \tilde{E}_n$ we have for some positive constant $R$ (depending on $R_2, \lambda$ and $k$) that
$${W}^n_H(\mathcal{B}^n) = \exp \left( - \sum\limits_{i = 1}^{k - 1} \int_{-a_n}^{a_n} H \left( B^n_{i+1}(u) - B^n_i(u)\right)du \right) \geq e^{-R/\sqrt{n}},$$
which clearly establishes (\ref{S33E2}). 
\end{proof}

We end this section with the following simple lemma.
\begin{lemma}\label{S3Simple} Suppose that $H: [-\infty, \infty) \rightarrow [0, \infty)$ is continuous. Fix $k \in \mathbb{N}$, $a < b$, $M > 0$ and $g \in C([a,b])$ such that $\sup_{x \in [a,b]}|g(x)| \leq M$. Then for any sequence $a_n, b_n \in [a,b]$ with $a_n < b_n$, $b_n - a_n \rightarrow 0$, and $\vec{x}^n, \vec{y}^n \in [-M, M]^k$ we have that 
\begin{equation}\label{S3Limit1}
\lim_{n \rightarrow \infty} \mathbb{E}_{H}^{1,k,a_n, b_n, \vec{x}^n, \vec{y}^n, \infty, -\infty} \left[ \exp \left( - \int_{a_n}^{b_n} H \left( g(u) - \mathcal{Q}_k(u)\right)du  \right) \right] = 1,
\end{equation}
where $(\mathcal{Q}_1, \dots, \mathcal{Q}_k)$ has law $\mathbb{P}_{H}^{1,k,a_n, b_n, \vec{x}^n, \vec{y}^n}$ as in Definition \ref{DefAvoidingLaw}. 
\end{lemma}
\begin{proof}
Let $(\Omega, \mathcal{F}, \mathbb{P})$ denote a probability space that supports $k$ i.i.d. Brownian bridges $\{ \tilde{B}_i \}_{i = 1}^{k}$ (from $\tilde{B}_i(0) = 0$ to $\tilde{B}_i(1) = 0$) with diffusion parameter $1$. We denote the expectation operator with respect to $\mathbb{P}$ by $\mathbb{E}$. We also define for $n \in \mathbb{N}$ and $i = 1, \dots, k$ the bridges on $[a_n, b_n]$
\begin{equation*}
B_i^n(t) = (b_n-a_n)^{1/2} \cdot \tilde{B}_i \left( \frac{t - a_n}{b_n-a_n} \right) + \left(\frac{b_n-t}{b_n-a_n} \right) \cdot x^n_{i} + \left( \frac{t- a_n}{b_n-a_n}\right) \cdot y^n_{i }.
\end{equation*}
We also set $B^n_0 = \infty$ and $B^n_{k+1} = g \vert_{[a_n, b_n]}$.

Notice that for $\mathbb{P}$-a.e. $\omega \in \Omega$ there exists $\infty > L(\omega) > 0$ such that 
$$\sup_{ n \geq 1} \max_{i = 1, \dots, k + 1} |B_i^n(t)| \leq L(\omega).$$
The latter together with the facts that $H$ is continuous and $b_n - a_n \rightarrow 0$ imply that for $r \in \{0, 1\}$ we have $\mathbb{P}$-a.s.
$$\lim_{ n\rightarrow \infty} \sum_{i = 0}^{k-r} \int_{a_n}^{b_n} H \left( B^n_{i+1}(u) - B^n_i(u)\right)du = 0.$$
This shows that for $r \in \{0, 1\}$ we have $\mathbb{P}$-a.s.
$$\lim_{ n\rightarrow \infty} \exp \left( - \sum_{i = 0}^{k-r} \int_{a_n}^{b_n} H \left( B^n_{i+1}(u) - B^n_i(u)\right)du \right)  = 1,$$
and so by the bounded convergence theorem (using that $H \geq 0$) we conclude for $r \in \{0,1\}$ that 
$$\lim_{n \rightarrow \infty} \mathbb{E}\left[ \exp \left( - \sum_{i = 0}^{k-r}  \int_{a_n}^{b_n} H \left( B^n_{i+1}(u) - B^n_i(u)\right)du \right)\right] = 1.$$
On the other hand, from Definition \ref{DefAvoidingLaw} we have for $n \in \mathbb{N}$ that 
\begin{equation*}
\begin{split}
&\mathbb{E}_{H}^{n} \left[ \exp \left( - \int_{a_n}^{b_n} \hspace{-3mm} H \left( g(u) - \mathcal{Q}_k(u)\right)du  \right) \right]  = \frac{\mathbb{E}\left[ \exp \left( - \sum_{i = 0}^{k} \int_{a_n}^{b_n} H \left( B^n_{i+1}(u) - B^n_i(u)\right)du \right)\right]}{\mathbb{E}\left[ \exp \left( - \sum_{i = 0}^{k - 1} \int_{a_n}^{b_n} H \left( B^n_{i+1}(u) - B^n_i(u)\right)du \right)\right]}, \\
&.
\end{split}
\end{equation*}
where we wrote $\mathbb{E}_{H}^n$ in place of $\mathbb{E}_{H}^{1,k,a_n, b_n, \vec{x}^n, \vec{y}^n, \infty, -\infty} $ to ease the notation. The last two equations imply (\ref{S3Limit1}).
\end{proof}

%-------------------------------------------------------------------------------------------------------------------------------------------------------------------------------------------------
% Section 3.4
%
%-------------------------------------------------------------------------------------------------------------------------------------------------------------------------------------------------
\subsection{Monotone coupling lemmas}\label{Section3.4} The following lemmas provide couplings of two $H$-Brownian bridge line ensembles on the same interval, which depend monotonically on their entrance, exit and boundary data.

\begin{lemma}\label{MCLxy} Assume the same notation as in Definition \ref{DefAvoidingLaw} and that $H$ is convex. Fix $k \in \mathbb{N}$, $a < b$ and two continuous function $f: [a,b] \rightarrow (-\infty, \infty]$, $g: [a,b] \rightarrow [-\infty, \infty)$ and assume that $\vec{x}, \vec{y}, \vec{x}', \vec{y}' \in \mathbb{R}^k$ are such that $x_i \leq x_i'$, $y_i \leq y_i'$ for $i = 1,\dots, k$. Then there exists a probability space $(\Omega, \mathcal{F}, \mathbb{P})$, which supports two $\llbracket 1, k \rrbracket$-indexed line ensembles $\mathcal{L}^t$ and $\mathcal{L}^b$ on $[a,b]$ such that the law of $\mathcal{L}^{t}$ {\big (}resp. $\mathcal{L}^b${\big )} under $\mathbb{P}$ is given by $\mathbb{P}_{H}^{1,k,a,b, \vec{x}', \vec{y}', f, g}$ {\big (}resp. $\mathbb{P}_{H}^{1,k,a,b, \vec{x}, \vec{y}, f, g}${\big )} and such that $\mathbb{P}$-almost surely we have $\mathcal{L}_i^t(r) \geq \mathcal{L}^b_i(r)$ for all $i = 1,\dots, k$ and $r \in [a,b]$.
\end{lemma}

\begin{lemma}\label{MCLfg} Assume the same notation as in Definition \ref{DefAvoidingLaw} and that $H$ is convex. Fix $k \in \mathbb{N}$, $a < b$,  $\vec{x}, \vec{y} \in \mathbb{R}^k$ and four continuous functions$f^t, f^b: [a,b] \rightarrow (-\infty, \infty]$, $g^t, g^b: [a,b] \rightarrow [-\infty, \infty)$ such that $f^t(r) \geq f^b(r)$ and $g^t(r) \geq g^b(r)$ for all $r \in [a,b]$. Then there exists a probability space $(\Omega, \mathcal{F}, \mathbb{P})$, which supports two $\llbracket 1, k \rrbracket$-indexed line ensembles $\mathcal{L}^t$ and $\mathcal{L}^b$ on $[a,b]$ such that the law of $\mathcal{L}^{t}$ {\big (}resp. $\mathcal{L}^b${\big )} under $\mathbb{P}$ is given by $\mathbb{P}_{H}^{1,k,a,b, \vec{x}, \vec{y}, f^t, g^t}$ {\big (}resp. $\mathbb{P}_{H}^{1,k,a,b, \vec{x}, \vec{y}, f^b, g^b}${\big )} and such that $\mathbb{P}$-almost surely we have $\mathcal{L}_i^t(r) \geq \mathcal{L}^b_i(r)$ for all $i = 1,\dots, k$ and $r \in [a,b]$.
\end{lemma}

In plain words, Lemma \ref{MCLxy} states that one can couple two line ensembles $\mathcal{L}^{t}$ and $\mathcal{L}^{b}$ of $H$-Brownian bridges, with the same boundary data $(f, g)$, in such a way that if all entrance and exit data of $\mathcal{L}^{t}$ are above the respective entrance and exit data of $\mathcal{L}^{b}$, then all curves of $\mathcal{L}^{t}$ are almost surely above the respective curves of $\mathcal{L}^{b}$. Lemma \ref{MCLfg}, states that one can couple two line ensembles $\mathcal{L}^{t}$ and $\mathcal{L}^{b}$ that have the same entrance and exit data, but the boundary data $(f^t,g^t)$ of $\mathcal{L}^{t}$ is above the boundary data $(f^b,g^b)$ of $\mathcal{L}^{b}$, in such a way that all curves of $\mathcal{L}^{t}$ are almost surely above the respective curves of $\mathcal{L}^{b}$. 

Lemmas \ref{MCLxy} and \ref{MCLfg} can be found in \cite[Section 2]{CorHamK}. The key idea behind their proof is to approximate the Brownian bridges by random walk bridges, for which constructing the monotone couplings is easier, and perform a limit transition. Since the details surrounding that limit transition are only briefly mentioned in \cite{CorHamK}, and since these lemmas are central results that will be used in the proof of Theorem \ref{ThmMain}, we included their proofs in Section \ref{Section5.3}.

%----------------------------------------------------------------------------------------------------------------------------------------------------------------------------------------------------------------------------
%
%     Section 4
%
%----------------------------------------------------------------------------------------------------------------------------------------------------------------------------------------------------------------------------
\section{Proofs of Theorem \ref{ThmMain} and Corollary \ref{CorMain2} }\label{Section4} In this section we prove Theorem \ref{ThmMain} and Corollary \ref{CorMain2}. We first state the main result of this section as Proposition \ref{PropMain} and deduce Theorem \ref{ThmMain} from it.  In Section \ref{Section4.1} we present the proof of a basic case of Proposition \ref{PropMain} to illustrate some of the key ideas and we give the full proof in Section \ref{Section4.2}. In Section \ref{Section4.3} we prove Corollary \ref{CorMain2}.

\begin{proposition}\label{PropMain}  Let $\Sigma = \llbracket 1, N \rrbracket$ with $N \in \mathbb{N}$  and $\Lambda = [a,b] \subset \mathbb{R}$. Fix $\lambda > 0$ and suppose that $H$ is a $\lambda$-exponential Hamiltonian in the sense of Definition \ref{S2Halph}. Suppose that $\mathcal{L}^1$ and $\mathcal{L}^2$ are $\Sigma$-indexed line ensembles on $\Lambda$ that satisfy the partial $H$-Brownian Gibbs property with laws $\mathbb{P}_1$ and $\mathbb{P}_2$ respectively. Suppose further that for every $k\in \mathbb{N}$,  $a < t_1 < t_2 < \cdots < t_k <  b$ and $z_1, \dots, z_k \in \mathbb{N}$ we have that 
\begin{equation}\label{FDEPM}
\mathbb{E}_{\mathbb{P}_1} \left[ \prod_{ i =1}^k \exp \left( - z_i \exp \left( \lambda \mathcal{L}^1_{1}(t_i)\right) \right) \right] =\mathbb{E}_{\mathbb{P}_2} \left[ \prod_{ i =1}^k \exp \left( - z_i \exp \left(\lambda \mathcal{L}^2_{1}(t_i)\right) \right)  \right].
\end{equation}
Then for every $k\in \mathbb{N}$,  $a  < t_1 < t_2 < \cdots < t_k <  b$, $n_1, \dots, n_k \in \llbracket 1, N \rrbracket$ and $z_1, \dots, z_k \in \mathbb{N}$ we have
\begin{equation}\label{FDEPM2}
\mathbb{E}_{\mathbb{P}_1} \left[ \prod_{ i =1}^k \exp \left( - z_i \exp \left(\lambda \mathcal{L}^1_{n_i}(t_i)\right) \right) \right] =\mathbb{E}_{\mathbb{P}_2} \left[ \prod_{ i =1}^k \exp \left( - z_i \exp \left(\lambda \mathcal{L}^2_{n_i}(t_i)\right) \right)  \right].
\end{equation}
\end{proposition}
The proof of Proposition \ref{PropMain} is given in Section \ref{Section4.2} below. In the remainder of this section we assume its validity and prove Theorem \ref{ThmMain}
\begin{proof}[Theorem \ref{ThmMain}] We assume the same notation as in Theorem \ref{ThmMain}. Let $a, b \in \Lambda$ with $a < b$ and $K \in \Sigma$ be given. We define $\pi_{[a,b]}^{\llbracket 1, K\rrbracket} : C (\Sigma \times \Lambda) \rightarrow C(\llbracket 1, K \rrbracket \times [a,b])$ through
\begin{equation}\label{ProjBox}
\pi_{[a,b]}^{\llbracket 1, K\rrbracket} (f)(i,x) = f(i,x) \mbox{ for $i = 1, \dots, K$ and $x \in [a,b]$.}
\end{equation}
Note that by Definition \ref{DefPBGP} we have that under $\mathbb{P}_v$ the $\llbracket 1, K \rrbracket$-indexed line ensembles $\pi_{[a,b]}^{\llbracket 1, K \rrbracket}(\mathcal{L}^v)$ on $[a,b]$ satisfies the partial $H$-Brownian Gibbs property, where $v \in \{1,2\}$. Here it is important that we work with the partial $H$-Brownian Gibbs property and not the usual $H$-Brownian Gibbs property, cf. Remark \ref{RPBGP}. 

Since the finite-dimensional distribution of $\mathcal{L}^1_1$ under $\mathbb{P}_1$ is the same as that of $\mathcal{L}^2_1$ under $\mathbb{P}_2$, we conclude that (\ref{FDEPM}) is satisfied for the line ensembles $\pi_{[a,b]}^{\llbracket 1, K \rrbracket}(\mathcal{L}^1)$ and $\pi_{[a,b]}^{\llbracket 1, K \rrbracket}(\mathcal{L}^2)$. Consequently, by Proposition \ref{PropMain} we conclude that for every $k\in \mathbb{N}$,  $a  < t_1 < t_2 < \cdots < t_k <  b$, $n_1, \dots, n_k \in \llbracket 1, K \rrbracket$ and $z_1, \dots, z_k \in \mathbb{R}$ we have
\begin{equation*}
\mathbb{E}_{\mathbb{P}_1} \left[ \prod_{ i =1}^k \exp \left( - z_i \exp \left( \lambda \mathcal{L}^1_{n_i}(t_i)\right) \right) \right] =\mathbb{E}_{\mathbb{P}_2} \left[ \prod_{ i =1}^k \exp \left( - z_i \exp \left( \lambda \mathcal{L}^2_{n_i}(t_i)\right) \right)  \right].
\end{equation*}
 Since $[a,b] \subset \Lambda$ and $K \in \Sigma$ were arbitrary we conclude that the latter equality holds for any $k\in \mathbb{N}$;  $ t_1 < t_2 < \cdots < t_k $, with $t_i \in \Lambda^o$ for $i =1 ,\dots, k$; $n_1, \dots, n_k \in \Sigma$ and $z_1, \dots, z_k \in \mathbb{N}$. From Lemma \ref{PropFD} we conclude that $\lambda \cdot \mathcal{L}^1$ has the same distribution as $\lambda \cdot \mathcal{L}^2$, and hence $\mathbb{P}_1 = \mathbb{P}_2$. 

\end{proof}

%----------------------------------------------------------------------------------------------------------------------------------------------------------------------------------------------------------------------------
%
%     Section 4.1
%
%----------------------------------------------------------------------------------------------------------------------------------------------------------------------------------------------------------------------------
\subsection{Basic case of Proposition \ref{PropMain}}\label{Section4.1} In this section we work under the same assumptions as in Proposition \ref{PropMain} when $N = 2$ and prove (\ref{FDEPM2}) in the simplest non-trivial case when $k = 1$ and $n_1 = 2$. As we will see, many of the key ideas that go into the proof of Proposition \ref{PropMain} are already present in this simple case. The goal is to illustrate the key arguments and constructions, so that the reader is better equipped before proceeding with the general proof in the next section.

The special case above asks us to prove that for any $t_1 \in (a,b)$ and $z_1 \in \mathbb{N}$ we have 
\begin{equation}\label{S4E1}
\mathbb{E}_{\mathbb{P}_1} \left[\exp \left( - z_1 \exp \left( \lambda \mathcal{L}^1_{2}(t_1)\right) \right) \right] = \mathbb{E}_{\mathbb{P}_2} \left[\exp \left( - z_1 \exp \left( \lambda \mathcal{L}^2_{2}(t_1)\right) \right) \right].
\end{equation}
From equation (\ref{FDEPM}) and Lemma \ref{PropFD} we know that $\mathcal{L}_1^1$ under $\mathbb{P}_1$ has the same law as $\mathcal{L}_1^2$ under $\mathbb{P}_2$ as $\{1 \}$-indexed line ensembles on $[a,b]$, or equivalently as random variables taking values in $(C([a,b]), \mathcal{C})$. In particular, we have for any bounded measurable function $G : C([a,b]) \rightarrow \mathbb{R}$ that 
\begin{equation}\label{S4E2}
\mathbb{E}_{\mathbb{P}_1} \left[ G( \mathcal{L}^1_{1}) \right] = \mathbb{E}_{\mathbb{P}_2} \left[G( \mathcal{L}^2_{1}) \right].
\end{equation}

The main idea of the argument is to find a sequence of measurable functions $G_w : C([a,b]) \rightarrow [0,\infty]$ for $w\in \mathbb{N}$ that satisfy the following 
\begin{equation}\label{S4E3}
\begin{split}
&  \mathbb{E}_{\mathbb{P}_1} \left[ G_w( \mathcal{L}^1_{1}) \right] = \mathbb{E}_{\mathbb{P}_2} \left[G_w( \mathcal{L}^2_{1}) \right] \mbox{ for $w \in \mathbb{N}$ and } \\
&\mathbb{E}_{\mathbb{P}_v} \left[\exp \left( - z_1 \exp \left( \lambda \mathcal{L}^v_{2}(t_1)\right) \right) \right]  = \lim_{w \rightarrow \infty}\mathbb{E}_{\mathbb{P}_v} \left[ G_w( \mathcal{L}^v_{1}) \right] , \mbox{ where $v \in \{1,2\}$}.
\end{split}
\end{equation}
If a sequence $G_w$ satisfying (\ref{S4E3}) exists, then clearly (\ref{S4E1}) would follow.

 In the remainder of the section we construct the sequence $G_w$ and show that it satisfies (\ref{S4E3}). Given $p,q, c,d, x,y, r \in \mathbb{R}$ with $p < c < d <q$ we define 
\begin{equation}\label{S41Fun}
F( r; p,q; c,d; x, y) = \mathbb{P}_{free}^{1,1,p,q, x,y} \left(B(c) \leq r \mbox{ and } B(d) \leq r \right),
\end{equation}
which is the probability that a Brownian bridge from $B(p) = x$ to $B(q) = y$ with diffusion parameter $1$ is below the level $r$ at times $c$ and $d$. Let us define 
\begin{equation}\label{S41Par}
\begin{split}
&\beta = z_1/2, \hspace{3mm},\hspace{3mm}A_w = \lambda^{-1} \log w, \hspace{3mm} V_w = w^{-1/3}, \hspace{3mm}  a_w = t_1 - \beta/H(A_w)- w^{-2}, \\
& b_w = t_1 + \beta /H(A_w) + w^{-2}, \hspace{3mm}c_w = t_1 - \beta/H(A_w), \hspace{3mm} d_w = t_1 + \beta /H(A_w),
\end{split}
\end{equation}
for $w \geq W_0$, where $W_0$ is sufficiently large so that $ d_w - c_w = 2\beta/ H(A_w) \in [w^{-5/4}, w^{-3/4}]$ and $a_w,b_w \in (a,b)$. Here we used the fact that $H$ is a $\lambda$-exponential Hamiltonian and that $t_1 \in (a,b)$ is in the interior of $[a,b]$. With this data we define for $f \in C([a,b])$ the function
$$G_w(f) = \frac{{\bf 1} \{ f(c_w) \leq -A_w \} \cdot {\bf 1} \{ f(d_w) \leq -A_w \} }{F( -A_w ; a_w,b_w; c_w,d_w; f(a_w), f(b_w)) }, \mbox{ where } w \geq W_0.$$
This is the choice of $G_w$ that satisfies (\ref{S4E3}). 

We first show that $G_w$ satisfies the first line in (\ref{S4E3}). For any $M \in \mathbb{N}$ we introduce the functions 
$$G_w^M(f) =  {\bf 1} \{ f(c_w) \leq -A_w \} \cdot {\bf 1} \{ f(d_w) \leq -A_w\} \cdot \min \left( M , \frac{1 }{F( -A_w; a_w,b_w; c_w,d_w; f(a_w), f(b_w)) } \right),$$
and note that $G_w^M$ are bounded measurable functions on $C([a,b])$. The boundedness of $G_w^M$ by $M$ is evident, while the measurability follows from the fact that $F( -A_w; a_w,b_w; c_w,d_w; x, y)$ is measurable in $x,y$ in view of  Lemma \ref{LemmaMeasExp}. For each $M \in \mathbb{N}$ we have from (\ref{S4E2}) that 
$$\mathbb{E}_{\mathbb{P}_1} \left[ G^M_w( \mathcal{L}^1_{1}) \right] = \mathbb{E}_{\mathbb{P}_2} \left[G^M_w( \mathcal{L}^2_{1}) \right].$$
Taking $M \rightarrow \infty$ in the last equation and applying the monotone convergence theorem, we get
$$\mathbb{E}_{\mathbb{P}_1} \left[ G_w( \mathcal{L}^1_{1}) \right] = \lim_{M \rightarrow \infty} \mathbb{E}_{\mathbb{P}_1} \left[ G^M_w( \mathcal{L}^1_{1}) \right] = \lim_{M \rightarrow \infty}\mathbb{E}_{\mathbb{P}_2} \left[G^M_w( \mathcal{L}^2_{1}) \right] =  \mathbb{E}_{\mathbb{P}_2} \left[G_w( \mathcal{L}^2_{1}) \right],$$
which establishes the first line in (\ref{S4E3}).

Establishing the second line in (\ref{S4E3}) is the hard part, and to accomplish it we first need to rewrite $\mathbb{E}_{\mathbb{P}_v} \left[ G_w( \mathcal{L}^v_{1}) \right]$ for $v \in \{1, 2\}$ using the partial $H$-Brownian Gibbs property. A technical aspect of the latter is that it requires that we work with bounded functions, and so we first rewrite $\mathbb{E}_{\mathbb{P}_v} \left[ G^M_w( \mathcal{L}^v_{1}) \right]$ and then take the $M \rightarrow \infty$ limit. By the partial $H$-Brownian Gibbs property, cf. Definition \ref{DefPBGP}, we have that 
\begin{equation*}
\begin{split}
&\mathbb{E}_{\mathbb{P}_v} \left[ G^M_w( \mathcal{L}^v_{1}) \right] =\mathbb{E}_{\mathbb{P}_v} \left[\mathbb{E}_{\mathbb{P}_v} \left[G^M_w( \mathcal{L}^v_{1}) \vert \mathcal{F}_{ext}( \{1\} \times (a_w, b_w)) \right] \right] = \\
&\mathbb{E}_{\mathbb{P}_v} \left[ \mathbb{P}^{v,w}_H (\mathcal{Q}(c_w) \leq -A_w, \mathcal{Q}(d_w) \leq -A_w) \cdot \min \left( M , \frac{1 }{F( -A_w; a_w,b_w; c_w,d_w; \mathcal{L}_1^v(a_w), \mathcal{L}^v_1(b_w)) } \right) \right],
\end{split}
\end{equation*}
where we write $\mathbb{P}^{v,w}_H$ in place of $\mathbb{P}_H^{1,1,a_w, b_w, \mathcal{L}_1^v(a_w), \mathcal{L}^v_1(b_w), \infty, \mathcal{L}^v_2[a_w,b_w]}$ to simplify the expression, and $\mathcal{Q}$ has law $\mathbb{P}^{v,w}_H$. Taking $M \rightarrow \infty$ in the last equation and applying the monotone convergence theorem again we see that for $v \in \{1, 2\}$ 
\begin{equation}\label{S4E4}
 \mathbb{E}_{\mathbb{P}_v} \left[ G_w( \mathcal{L}^v_{1}) \right] = \mathbb{E}_{\mathbb{P}_v} \left[ \frac{\mathbb{P}^{v,w}_H (\mathcal{Q}(c_w) \leq -A_w, \mathcal{Q}(d_w) \leq -A_w) }{F( -A_w; a_w,b_w; c_w,d_w; \mathcal{L}_1^v(a_w), \mathcal{L}^v_1(b_w)) }  \right].
\end{equation}

The key observation that motivates much of the proof is that $\mathbb{P}_v$-almost surely we have
\begin{equation}\label{S4GlobE1}
\lim_{w \rightarrow \infty} \frac{\mathbb{P}^{v,w}_H (\mathcal{Q}(c_w) \leq -A_w, \mathcal{Q}(d_w) \leq -A_w) }{F( -A_w; a_w,b_w; c_w,d_w; \mathcal{L}_1^v(a_w), \mathcal{L}^v_1(b_w)) }   = \exp \left( - z_1 \exp \left( \lambda \mathcal{L}^v_{2}(t_1)\right) \right) .
\end{equation}
We will justify (\ref{S4GlobE1}) properly and deduce the second line in (\ref{S4E3}) in Steps I-IV below. Before we go into the details we take the time here to give an informal description of the main ideas.\\

First of all, if (\ref{S4GlobE1}) is true then the second line of (\ref{S4E3})  would follow by the bounded convergence theorem and (\ref{S4E4}), since by Lemma \ref{MCLfg} the variables in (\ref{S4GlobE1}) are all bounded by $1$. Consequently, establishing (\ref{S4GlobE1}) is indeed all we need.

To see why (\ref{S4GlobE1}) holds we note that from Definition \ref{DefAvoidingLaw} we have
\begin{equation*}
\begin{split}
&\frac{\mathbb{P}^{v,w}_H (\mathcal{Q}(c_w) \leq -A_w, \mathcal{Q}(d_w) \leq -A_w) }{F( -A_w; a_w,b_w; c_w,d_w; \mathcal{L}_1^v(a_w), \mathcal{L}^v_1(b_w)) }  = \\
& \frac{1}{\mathbb{E}^{v,w}_{free} \left[ W_H^{v,w}(\mathcal{Q}) \right]} \cdot \frac{\mathbb{E}^{v,w}_{free} \left[ {\bf 1} \{\mathcal{Q}(c_w) \leq -A_w, \mathcal{Q}(d_w) \leq -A_w \} \cdot W_H^{v,w}(\mathcal{Q}) \right]}{F( -A_w; a_w,b_w; c_w,d_w; \mathcal{L}_1^v(a_w), \mathcal{L}^v_1(b_w)) },
\end{split}
\end{equation*}
where we have written $\mathbb{P}^{v,w}_{free}$ in place of $\mathbb{P}_{free}^{1,1,a_w, b_w, \mathcal{L}_1^v(a_w), \mathcal{L}^v_1(b_w)}$ to ease the notation and
$$W_H^{v,w}(\mathcal{Q}) = \exp \left(   -   \int_{a_w}^{b_w}\hspace{-3mm} H \left(\mathcal{L}_2^v(u) - \mathcal{Q}(u) \right)du \hspace{-1mm} \right).$$
The quantity $\mathbb{E}^{v,w}_{free} \left[ W_H^{v,w}(\mathcal{Q}) \right]$ converges almost surely to $1$, since both $\mathcal{L}_2^v$ and $\mathcal{Q}$ are of order $1$, while $b_w - a_w \rightarrow 0$ so that $W_H^{v,w}(\mathcal{Q}) \rightarrow 1$. This limit is in fact a consequence of Lemma \ref{S3Simple}. Thus we are left with showing that $\mathbb{P}_v$-almost surely
\begin{equation}\label{S4GlobE2}
\lim_{w \rightarrow \infty} \frac{\mathbb{E}^{v,w}_{free} \left[ {\bf 1} \{\mathcal{Q}(c_w) \leq -A_w, \mathcal{Q}(d_w) \leq -A_w \} \cdot W_H^{v,w}(\mathcal{Q}) \right]}{F( -A_w; a_w,b_w; c_w,d_w; \mathcal{L}_1^v(a_w), \mathcal{L}^v_1(b_w)) } = \exp \left( - z_1 \exp \left( \lambda \mathcal{L}^v_{2}(t_1)\right) \right).
\end{equation}
By Lemma \ref{BridgeLemma} we have that on the event $ \{\mathcal{Q}(c_w) \leq -A_w, \mathcal{Q}(d_w) \leq -A_w \}$ the Brownian bridge $\mathcal{Q}$ is very close to the level $-A_w$ on the whole interval $[c_w, d_w]$ and behaves reasonably on $[a_w, c_w] \cup [d_w, b_w]$ with overwhelming probability. On the other hand, since $\mathcal{L}_2^v$ is continuous and $[a_w, b_w]$ is small, one has that $\mathcal{L}_2^v(u) \approx \mathcal{L}_2^v(t_1)$ for $u \in [a_w, b_w]$ and all large $w$ (again with overwhelming probability). The latter two observations and the fact that $b_w - d_w = c_w - a_w = w^{-2}$ is much smaller than $d_w - c_w = 2\beta/ H(A_w)$ suggest that with very high probability
$$W_H^{v,w}(\mathcal{Q}) = \exp \left(   -   \int_{a_w}^{b_w}\hspace{-3mm} H \left(\mathcal{L}_2^v(u) - \mathcal{Q}(u) \right)du \hspace{-1mm} \right) \approx\exp \left(   -  2\beta \cdot  H \left(\mathcal{L}_2^v(t_1) + A_w \right)/ H(A_w) \right) .$$
Since $\lim_{ w\rightarrow \infty} H(x + A_w)/ H(A_w) = \exp(\lambda x)$ by our assumption that $H$ is a $\lambda$-exponential Hamiltonian and $A_w \rightarrow \infty$ we conclude that with high probability
$$W_H^{v,w}(\mathcal{Q}) \approx\exp \left(   -  2 \beta \cdot  \exp (\lambda \mathcal{L}_2^v(t_1) ) \right) = \exp \left(   -  z_1 \cdot \exp (\lambda \mathcal{L}_2^v(t_1)) \right),$$
where the last equality follows from $2 \beta  = z_1 $, see (\ref{S41Par}).  The last equation suggests that $W_H^{v,w}(\mathcal{Q})$ is essentially a constant with respect to $\mathbb{P}^{v,w}_{free}$ and then 
$$\mathbb{E}^{v,w}_{free} \left[ {\bf 1} \{\mathcal{Q}(c_w) \leq -A_w, \mathcal{Q}(d_w) \leq -A_w \} \cdot W_H^{v,w}(\mathcal{Q}) \right] \approx $$
$$  \exp \left(   -  z_1 \cdot \exp (\lambda \mathcal{L}_2^v(t_1)) \right) \cdot \mathbb{E}^{v,w}_{free} \left[ {\bf 1} \{\mathcal{Q}(c_w) \leq -A_w, \mathcal{Q}(d_w) \leq -A_w \} \right].$$
Since the second factor on the second line is nothing but $F( -A_w; a_w,b_w; c_w,d_w; \mathcal{L}_1^v(a_w), \mathcal{L}^v_1(b_w))$ by definition, we see that the last equation implies (\ref{S4GlobE2}). This concludes the sketch of our arguments.\\

In the remainder of the section we give the necessary details behind the above sketch and for clarity we split the argument into four steps. When we present the general proof of Proposition \ref{PropMain} later in Section \ref{Section4.2} these steps will correspond to Steps 2-5, while the discussion prior to equation (\ref{S4GlobE1}) will correspond to Step 1.\\

{\bf \raggedleft Step I.} We claim that for all $M \in \mathbb{N}$ and $v \in \{1, 2\}$ we have $\mathbb{P}_v$-almost surely
\begin{equation}\label{S41R1}
\begin{split}
&\lim_{ w \rightarrow \infty} {\bf 1 }_{E^v_M} \cdot \frac{\mathbb{E}^{v,w}_{free} \left[ {\bf 1} \{\mathcal{Q}(c_w) \leq -A_w, \mathcal{Q}(d_w) \leq -A_w \} \cdot W_H^{v,w}(\mathcal{Q}) \right]}{F( -A_w; a_w,b_w; c_w,d_w; \mathcal{L}_1^v(a_w), \mathcal{L}^v_1(b_w)) } = {\bf 1 }_{E^v_M} \cdot e^{ - z_1 \exp \left(\lambda \mathcal{L}^v_{2}(t_1)\right) } ,\\
&\mbox{ where }W_H^{v,w}(\mathcal{Q}) = \exp \left(  \hspace{-1mm} -  \hspace{-1mm} \int_{a_w}^{b_w}\hspace{-3mm} H \left(\mathcal{L}_2^v(u) - \mathcal{Q}(u) \right)du \hspace{-1mm} \right) \hspace{-1mm}\mbox{ and }   E^v_M = \left \{ \sup_{ t \in [a, b]} \max_{i = 1,2} |\mathcal{L}^v_{i}(t)| \leq M \right\}.
\end{split}
\end{equation}
In the last equation we have written $\mathbb{P}^{v,w}_{free}$ in place of $\mathbb{P}_{free}^{1,1,a_w, b_w, \mathcal{L}_1^v(a_w), \mathcal{L}^v_1(b_w)}$ to simplify the notation and $\mathcal{Q}$ is $\mathbb{P}^{v,w}_{free}$-distributed. We will prove (\ref{S41R1}) in the steps below. Here we assume its validity and conclude the proof of the second line in (\ref{S4E3}).

From Definition \ref{DefAvoidingLaw} we know that 
\begin{equation*}
\begin{split}
&\frac{\mathbb{P}^{v,w}_H (\mathcal{Q}(c_w) \leq -A_w, \mathcal{Q}(d_w) \leq -A_w) }{F( -A_w; a_w,b_w; c_w,d_w; \mathcal{L}_1^v(a_w), \mathcal{L}^v_1(b_w)) }  = \\
& \frac{1}{\mathbb{E}^{v,w}_{free} \left[ W_H^{v,w}(\mathcal{Q}) \right]} \cdot \frac{\mathbb{E}^{v,w}_{free} \left[ {\bf 1} \{\mathcal{Q}(c_w) \leq -A_w, \mathcal{Q}(d_w) \leq -A_w \} \cdot W_H^{v,w}(\mathcal{Q}) \right]}{F( -A_w; a_w,b_w; c_w,d_w; \mathcal{L}_1^v(a_w), \mathcal{L}^v_1(b_w)) }.
\end{split}
\end{equation*}
In addition, from Lemma \ref{S3Simple} we know that $\mathbb{P}_v$-almost surely we have
$$\lim_{ w \rightarrow \infty} {\bf 1 }_{E^v_M} \cdot \mathbb{E}^{v,w}_{free} \left[ W_H^{v,w}(\mathcal{Q}) \right] = {\bf 1 }_{E^v_M}.$$
Here we used that $\mathbb{P}_{H}^{1,1,p, q, x, y, \infty, -\infty} = \mathbb{P}_{free}^{1,1,p, q, x, y}$, cf. Definition \ref{DefAvoidingLaw}. Combining the last two statements with (\ref{S41R1}) we conclude that for all $M \in \mathbb{N}$ and $v \in \{1, 2\}$ we have $\mathbb{P}_v$-almost surely
$$\lim_{ w \rightarrow \infty} {\bf 1 }_{E^v_M} \cdot \frac{\mathbb{P}^{v,w}_H (\mathcal{Q}(c_w) \leq -A_w, \mathcal{Q}(d_w) \leq -A_w) }{F( -A_w; a_w,b_w; c_w,d_w; \mathcal{L}_1^v(a_w), \mathcal{L}^v_1(b_w)) } = {\bf 1 }_{E^v_M} \cdot e^{ - z_1 \exp \left(\lambda \mathcal{L}^v_{2}(t_1)\right) } .$$
Since $\mathbb{P}_v \left( \cup_{M = 1}^\infty E^v_M \right) = 1$ we conclude that $\mathbb{P}_v$-almost surely
\begin{equation}\label{S4E5}
\lim_{ w \rightarrow \infty}   \frac{\mathbb{P}^{v,w}_H (\mathcal{Q}(c_w) \leq -A_w, \mathcal{Q}(d_w) \leq -A_w) }{F( -A_w; a_w,b_w; c_w,d_w; \mathcal{L}_1^v(a_w), \mathcal{L}^v_1(b_w)) } = e^{ - z_1 \exp \left(\lambda \mathcal{L}^v_{2}(t_1)\right) }
\end{equation}
From Lemma \ref{MCLfg} (applied to $a = a_w$, $b = b_w$, $\vec{x} = \mathcal{L}_1^v(a_w)$, $\vec{y} = \mathcal{L}_1^v(b_w)$, $f^t = f^b = \infty$, $g^t = \mathcal{L}^v_2[a_w, b_w]$ and $g^b = -\infty$) we know that 
\begin{equation}\label{S4E6}
   \frac{\mathbb{P}^{v,w}_H (\mathcal{Q}(c_w) \leq -A_w, \mathcal{Q}(d_w) \leq -A_w) }{F( -A_w; a_w,b_w; c_w,d_w; \mathcal{L}_1^v(a_w), \mathcal{L}^v_1(b_w)) } \leq 1.
\end{equation}
We mention that we used in (\ref{S4E6}) that $H$ is convex, as required for the application of Lemma \ref{MCLfg}. 

Combining (\ref{S4E4}), (\ref{S4E5}), (\ref{S4E6}) and the bounded convergence theorem, we conclude the second line in (\ref{S4E3}). \\

{\bf \raggedleft Step II.} In this step we prove (\ref{S41R1}). For $w \geq W_0$ we denote the events
\begin{enumerate}
\item $\tilde{A}_w = \{ -A_w - 1 \leq \mathcal{Q}(t) \leq M + 1 \mbox{ for $a_w \leq t \leq c_w $} \}$;
\item $\tilde{D}_w = \{ -A_w - 1 \leq \mathcal{Q}(t) \leq M + 1 \mbox{ for $d_w \leq t \leq b_w$} \}$;
\item $\tilde{E}_w = \{  -A_w - 2V_w \leq \mathcal{Q}(t) \leq -A_w + 2V_w \mbox{ for $c_w \leq t \leq d_w$} \}$;
\item $\tilde{F}_w = \{\mathcal{Q}(c_w) \leq -A_w \mbox{ and } \mathcal{Q}(d_w) \leq -A_w \},$
\end{enumerate}
where we recall that $V_w = w^{-1/3}$ from (\ref{S41Par}). We claim that $\mathbb{P}_v$-almost surely
\begin{equation}\label{S41R2}
\lim_{ w \rightarrow \infty} {\bf 1 }_{E^v_M} \cdot \frac{\mathbb{E}^{v,w}_{free} \left[ {\bf 1}_{\tilde{A}_w \cap \tilde{D}_w \cap \tilde{E}_w \cap \tilde{F}_w}  \cdot W_H^{v,w}(\mathcal{Q}) \right]}{\mathbb{P}^{v,w}_{free} \left( \tilde{A}_w \cap \tilde{D}_w \cap \tilde{E}_w \cap \tilde{F}_w\right)} = {\bf 1 }_{E^v_M} \cdot e^{ - z_1 \exp \left( \lambda \mathcal{L}^v_{2}(t_1)\right) }.
\end{equation} 
We will prove (\ref{S41R2}) in the next step. Here we assume its validity and conclude the proof of (\ref{S41R1}). \\

We observe that by definition we have
\begin{equation*} 
\frac{\mathbb{E}^{v,w}_{free} \left[ {\bf 1} \{\mathcal{Q}(c_w) \leq -A_w, \mathcal{Q}(d_w) \leq -A_w \} \cdot W_H^{v,w}(\mathcal{Q}) \right]}{F( -A_w; a_w,b_w; c_w,d_w; \mathcal{L}_1^v(a_w), \mathcal{L}^v_1(b_w)) }  = \frac{\mathbb{E}^{v,w}_{free} \left[ {\bf 1}_{\tilde{F}_w} \cdot W_H^{v,w}(\mathcal{Q})  \right]}{\mathbb{P}^{v,w}_{free}(\tilde{F}_w)},
\end{equation*}
which together with the fact that $W_H^{v,w} \in [0,1]$ (as $H(x) \in [0, \infty)$) implies 

\begin{equation}\label{S41R1V1}
\begin{split}
&\liminf_{ w \rightarrow \infty} {\bf 1 }_{E^v_M} \cdot \frac{\mathbb{E}^{v,w}_{free} \left[  {\bf 1}_{\tilde{A}_w \cap \tilde{D}_w \cap \tilde{E}_w} \cdot {\bf 1}_{\tilde{F}_w} \cdot W_H^{v,w}(\mathcal{Q})  \right]}{\mathbb{P}^{v,w}_{free} \left( \tilde{A}_w \cap \tilde{D}_w \cap \tilde{E}_w \cap \tilde{F}_w\right)}\cdot \frac{\mathbb{P}^{v,w}_{free} \left( \tilde{A}_w \cap \tilde{D}_w \cap \tilde{E}_w \cap \tilde{F}_w\right)}{\mathbb{P}^{v,w}_{free}(\tilde{F}_w)}  \leq \\
&\liminf_{ w \rightarrow \infty} {\bf 1 }_{E^v_M} \cdot \frac{\mathbb{E}^{v,w}_{free} \left[ {\bf 1}_{\tilde{F}_w} \cdot W_H^{v,w}(\mathcal{Q})  \right]}{\mathbb{P}^{v,w}_{free}(\tilde{F}_w)} \leq \limsup_{ w \rightarrow \infty} {\bf 1 }_{E^v_M} \cdot \frac{\mathbb{E}^{v,w}_{free} \left[ {\bf 1}_{\tilde{F}_w} \cdot W_H^{v,w}(\mathcal{Q})  \right]}{\mathbb{P}^{v,w}_{free}(\tilde{F}_w)} \leq  \\
&\limsup_{ w \rightarrow \infty} {\bf 1 }_{E^v_M} \cdot  \frac{\mathbb{E}^{v,w}_{free} \left[  {\bf 1}_{\tilde{A}_w \cap \tilde{D}_w \cap \tilde{E}_w} \cdot {\bf 1}_{\tilde{F}_w} \cdot W_H^{v,w}(\mathcal{Q})  \right] + \mathbb{P}^{v,w}_{free} \left( (\tilde{A}^c_w \cup \tilde{D}^c_w \cup \tilde{E}^c_w) \cap \tilde{F}_w\right) }{\mathbb{P}^{v,w}_{free} \left( \tilde{A}_w \cap \tilde{D}_w \cap \tilde{E}_w \cap \tilde{F}_w\right)} .
\end{split}
\end{equation}

On the other hand, by Lemma \ref{BridgeLemma} we know that we can find constants $C, c > 0$ (depending on $\lambda, M$) such that for all $w \geq W_0$ we have on $E^v_M$ 
\begin{equation}\label{S41R2V1}
\frac{\mathbb{P}^{v,w}_{free}((\tilde{A}^c_w \cup \tilde{D}^c_w \cup \tilde{E}^c_w ) \cap \tilde{F}_w) }{\mathbb{P}^{v,w}_{free}(\tilde{F}_w)} \leq  C e^{-c w^{1/12}}.
\end{equation} 
Combining (\ref{S41R2}) with (\ref{S41R1V1}) and (\ref{S41R2V1}) we conclude (\ref{S41R1}).\\

{\bf \raggedleft Step III.} In this step we prove (\ref{S41R2}). Let us fix $\epsilon_1, \epsilon_2 \in (0,1)$. By our assumption that $H$ is a $\lambda$-exponential Hamiltonian we know that there exists $R_1 > 0$ such that for $x \geq R_1$ we have 
\begin{equation}\label{S41R2E1}
0 \leq H(x) \leq  e^{(1 + \epsilon_1) \lambda x}.
\end{equation}
From the fact that $H:[-\infty, \infty) \rightarrow [0, \infty)$ is continuous we conclude that we can find a constant $R_2 > 0$ such that for all $x \in [-\infty, \infty)$ we have
\begin{equation}\label{S41R2E1.5}
0 \leq H(x) \leq R_2 + e^{(1 + \epsilon_1) \lambda x}.
\end{equation}

Let us fix $\omega \in E_M^v$ and note that by the continuity of $\mathcal{L}^v_2$ on $[a, b]$ we can find $\delta_1$ (depending on $\omega$) such that 
\begin{equation}\label{S41R2E2}
\sup_{s,t \in [a, b], |s - t| \leq \delta_1} \left| \mathcal{L}^v_2(s) - \mathcal{L}^v_2(t) \right| < \epsilon_2.
\end{equation}
Let $W_1 \geq W_0$ be sufficiently large (depending on $\omega$) so that for $w \geq W_1$ we have
\begin{equation}\label{S41R2E3}
\begin{split}
&2 w^{-2} \cdot \sup_{ x \in [-\infty, M+A_w + 1] } H(x) \leq \epsilon_1, \hspace{5mm}  b_w - a_w \leq \delta_1, \hspace{5mm}  \\
&\left|    e^{\pm 2\lambda V_w} - 1 \right| \leq \epsilon_1, \hspace{5mm} \sup_{x \in [-M - 2, M +2]}  \left| \frac{H(x + A_w)}{H(A_w)} - e^{\lambda x} \right| \leq \epsilon_1.
\end{split}
\end{equation}
The latter inequalities hold for all large enough $w$, in view of the fact that $H$ is a $\lambda$-exponential Hamiltonian, (\ref{S41R2E1.5}) and (\ref{S41Par}), where the latter gives $A_w = \lambda^{-1} \log w$, $V_w = w^{-1/3}$, $b_w - a_w = 2 w^{-2} + 2\beta/H(A_w)$.

We claim the following inequalities for $w \geq W_1$ and $\omega \in E^v_M$ on the event $\tilde{A}_w \cap \tilde{D}_w \cap \tilde{E}_w \cap \tilde{F}_w$
\begin{equation}\label{S41R2E4}
\begin{split}
& W_H^{v,w}(\mathcal{Q}) \geq   \exp \left(  - \epsilon_1 (1 + 2\beta)  -  2 \beta (1 + \epsilon_1) \exp(\lambda \epsilon_2) \cdot  \exp  \left( \lambda  \mathcal{L}_2^v(t_1)  \right)  \right) ,
\end{split}
\end{equation}
and 
\begin{equation}\label{S41R2E5}
\begin{split}
&  W_H^{v,w}(\mathcal{Q}) \leq   \exp \left( 2\beta \epsilon_1 -  2\beta ( 1 - \epsilon_1) \exp(-\lambda \epsilon_2)  \cdot \exp  \left( \lambda   \mathcal{L}_2^v(t_1)  \right)\right) .
\end{split}
\end{equation}
We will prove (\ref{S41R2E4}) and (\ref{S41R2E5}) in the next (and final step). Here we assume their validity and conclude the proof of (\ref{S41R2}).\\

Multiplying both sides of (\ref{S41R2E4}) and (\ref{S41R2E5}) by ${\bf 1}_{\tilde{A}_w \cap \tilde{D}_w \cap \tilde{E}_w \cap \tilde{F}_w} $ and taking $\mathbb{E}^{v,w}_{free}$-expectations on both sides we conclude that for $\omega \in E_M^v$ we have
\begin{equation*}
\begin{split}
& \mathbb{P}^{v,w}_{free} \left( \tilde{A}_w \cap \tilde{D}_w \cap \tilde{E}_w \cap \tilde{F}_w\right) \cdot \exp \left(  - \epsilon_1 (1 + 2\beta)  -  2 \beta (1 + \epsilon_1) \exp(\lambda \epsilon_2) \cdot  \exp  \left( \lambda  \mathcal{L}_2^v(t_1)  \right)  \right)  \leq \\
&  \mathbb{E}^{v,w}_{free} \left[ {\bf 1}_{\tilde{A}_w \cap \tilde{D}_w \cap \tilde{E}_w \cap \tilde{F}_w}   W_H^{v,w}(\mathcal{Q}) \right] \leq  \\
& \mathbb{P}^{v,w}_{free} \left( \tilde{A}_w \cap \tilde{D}_w \cap \tilde{E}_w \cap \tilde{F}_w\right)  \cdot  \exp \left( 2\beta \epsilon_1 -  2\beta ( 1 - \epsilon_1) \exp(-\lambda \epsilon_2)  \cdot \exp  \left( \lambda   \mathcal{L}_2^v(t_1)  \right)\right).
\end{split}
\end{equation*}
Dividing the above expressions by $\mathbb{P}^{v,w}_{free} \left( \tilde{A}_w \cap \tilde{D}_w \cap \tilde{E}_w \cap \tilde{F}_w\right) $ and letting $w \rightarrow \infty$ we conclude 
\begin{equation*}
\begin{split}
&{\bf 1 }_{E^v_M} \cdot  \exp \left(  - \epsilon_1 (1 + 2\beta)  -  2 \beta (1 + \epsilon_1) \exp(\lambda \epsilon_2) \cdot  \exp  \left( \lambda  \mathcal{L}_2^v(t_1)  \right)  \right) \leq \\
& \liminf_{w \rightarrow \infty} {\bf 1 }_{E^v_M} \frac{  \mathbb{E}^{v,w}_{free} \left[ {\bf 1}_{\tilde{A}_w \cap \tilde{D}_w \cap \tilde{E}_w \cap \tilde{F}_w}   W_H^{v,w}(\mathcal{Q}) \right] }{\mathbb{P}^{v,w}_{free} \left( \tilde{A}_w \cap \tilde{D}_w \cap \tilde{E}_w \cap \tilde{F}_w\right) } \leq \limsup_{w \rightarrow \infty}  {\bf 1 }_{E^v_M} \frac{  \mathbb{E}^{v,w}_{free} \left[ {\bf 1}_{\tilde{A}_w \cap \tilde{D}_w \cap \tilde{E}_w \cap \tilde{F}_w}   W_H^{v,w}(\mathcal{Q}) \right] }{\mathbb{P}^{v,w}_{free} \left( \tilde{A}_w \cap \tilde{D}_w \cap \tilde{E}_w \cap \tilde{F}_w\right) }  \leq \\
& {\bf 1 }_{E^v_M} \cdot  \exp \left( 2\beta \epsilon_1 -  2\beta ( 1 - \epsilon_1) \exp(-\lambda \epsilon_2)  \cdot \exp  \left( \lambda   \mathcal{L}_2^v(t_1)  \right)\right).
\end{split}
\end{equation*}
Since the above holds for all $\epsilon_1, \epsilon_2 \in (0,1)$ we conclude that 
$$ \lim_{w \rightarrow \infty} {\bf 1 }_{E^v_M} \cdot \frac{  \mathbb{E}^{v,w}_{free} \left[ {\bf 1}_{\tilde{A}_w \cap \tilde{D}_w \cap \tilde{E}_w \cap \tilde{F}_w}   W_H^{v,w}(\mathcal{Q}) \right] }{\mathbb{P}^{v,w}_{free} \left( \tilde{A}_w \cap \tilde{D}_w \cap \tilde{E}_w \cap \tilde{F}_w\right) } = {\bf 1 }_{E^v_M}\cdot e^{-2\beta \exp \left(  \lambda \mathcal{L}_2^v(t_1) \right)},$$
which is the same as (\ref{S41R2}) once we recall from (\ref{S41Par}) that $\beta = z_1/2$. \\

{\bf \raggedleft Step IV.} In this final step we show that (\ref{S41R2E4}) and (\ref{S41R2E5}) hold for $w \geq W_1$ and $\omega \in E^v_M$ on the event $\tilde{A}_w \cap \tilde{D}_w \cap \tilde{E}_w \cap \tilde{F}_w$. In the sequel we assume that we have fixed $w \geq W_1$, $\omega \in E^v_M$ and an elementary outcome $\mathcal{Q}$ in $\tilde{A}_w \cap \tilde{D}_w \cap \tilde{E}_w \cap \tilde{F}_w$.

 We observe that we have the following tower of inequalities
\begin{equation*}
\begin{split}
& W_H^{v,w}(\mathcal{Q}) =   \exp \left( -\int_{a_w}^{b_w} H \left( \mathcal{L}_2^v(u) - \mathcal{Q}(u) \right)du \right) \geq  \\
&  \exp \left(  - 2 w^{-2} \sup_{ x \in [-\infty, M+A_w + 1] } H(x) - \int_{c_w}^{d_w} H \left( \mathcal{L}_2^v(u) - \mathcal{Q}(u)\right)du \right) \geq \\
&\exp \left(  - \epsilon_1 -  2 \beta \cdot \sup_{u \in [c_w, d_w]} H \left( \mathcal{L}_2^v(u) - \mathcal{Q}(u)\right)/H(A_w) \right) \geq \\
&  \exp \left( - \epsilon_1 (1 + 2\beta)  -  2 \beta \cdot \sup_{u \in [c_w, d_w]} \exp  \left( \lambda \cdot ( \mathcal{L}_2^v(u) - \mathcal{Q}(u) - A_w) \right)\right)  \geq \\
&  \exp \left( - \epsilon_1 (1 + 2\beta)  -  2 \beta \cdot  \exp  \left( \lambda \cdot ( \epsilon_2 + \mathcal{L}_2^v(t_1) + 2 V_w) \right)\right)  \geq \\
& \exp \left( - \epsilon_1 (1 + 2\beta)  -  2 \beta (1 + \epsilon_1) e^{\lambda \epsilon_2} \cdot  \exp  \left( \lambda  \mathcal{L}_2^v(t_1)  \right)\right) .
\end{split}
\end{equation*}
Let us elaborate on the last equation briefly. The first line is just the definition of $W_H^{v,w}(\mathcal{Q})$ in (\ref{S41R1}). In going from the first to the second line we used the fact that for $\omega \in E_M^v$ we have $|\mathcal{L}_2^v(u)| \leq M$ and on $\tilde{A}_w \cap \tilde{D}_w$ we have $\mathcal{Q}(u) \geq -A_w - 1$ on $[a_w, c_w] \cup [d_w ,b_w]$, while $b_w - d_w = c_w - a_w = w^{-2}$. In going from the second to the third line we used the first inequality on the first line of (\ref{S41R2E3}) and the fact that $d_w - c_w = 2\beta/H(A_w)$. In going from the third to the fourth line we used the second inequality on the second line of (\ref{S41R2E3}) and the fact that on $\tilde{E}_w$ we have $- A_w + 2V_w \leq\mathcal{Q}(u) \leq -A_w + 2V_w$ for $u \in [c_w, d_w]$ and hence $\mathcal{L}_2^v(u) - \mathcal{Q}(u) - A_w \in [-M-2, M+2]$ on $[c_w, d_w]$. In going from the fourth to the fifth line we used that on $\tilde{E}_w$ we have $\mathcal{Q}(u) \geq -A_w - 2V_w$ and also that $\mathcal{L}_2^v(t_1) + \epsilon_2 \geq \mathcal{L}_2^v(u)$ for $u \in [c_w, d_w]$ as follows from the second inequality on the first line of (\ref{S41R2E3}). In going from the fifth to the sixth line we used the first inequality in the second line of (\ref{S41R2E3}). The above tower implies (\ref{S41R2E4}).

We next observe that we have the following tower of inequalities  
\begin{equation*}
\begin{split}
&  W_H^{v,w}(\mathcal{Q}) \leq    \exp \left(  - \int_{c_w}^{d_w} H \left( \mathcal{L}_2^v(u) - \mathcal{Q}(u) \right)du \right) \leq  \\
&   \exp \left(  -  2\beta \cdot \inf_{ u \in [c_w, d_w]} H \left( \mathcal{L}_2^v(u) - \mathcal{Q}(u) \right)/H(A_w) \right) \leq  \\
&   \exp \left(  2\beta \epsilon_1 -  2\beta \cdot \inf_{ u \in [c_w, d_w]} \exp  \left( \lambda \cdot ( \mathcal{L}_2^v(u) - \mathcal{Q}(u) - A_w) \right) \right) \leq  \\
&   \exp \left(  2\beta \epsilon_1 -  2\beta \cdot \exp  \left( \lambda \cdot ( \mathcal{L}_2^v(t_1) - \epsilon_2 - 2 V_w) \right) \right) \leq  \\
&   \exp \left(  2\beta \epsilon_1 -  2\beta  ( 1 - \epsilon_1)  e^{-\lambda \epsilon_2} \cdot \exp  \left( \lambda   \mathcal{L}_2^v(t_1)  \right) \right).
\end{split}
\end{equation*}
Let us elaborate the last equation briefly. The first inequality follows from the definition of $W_H^{v,w}(\mathcal{Q})$ in (\ref{S41R1}) and the fact that $H \geq 0$. In going from the first to the second line we used that $d_w - c_w = 2\beta/ H(A_w)$. In going from the second to the third line we used the second inequality on the second line of (\ref{S41R2E3}). The latter inequality is applicable since $|\mathcal{L}_2^v(u)| \leq M$  for $\omega \in E_M^v$ and also on $\tilde{E}_w$ we have $- A_w + 2V_w \leq\mathcal{Q}(u) \leq -A_w + 2V_w$ for $u \in [c_w, d_w]$ and hence $\mathcal{L}_2^v(u) - \mathcal{Q}(u) - A_w \in [-M-2, M+2]$ on $[c_w, d_w]$. In going from the third to the fourth line we used that on $\tilde{E}_w$ we have $\mathcal{Q}(u) \leq -A_w + 2V_w$ and also that $\mathcal{L}_{2}^v(t_1) - \epsilon_2 \leq \mathcal{L}_2^v(u)$ for $u \in [c_w, d_w]$ as follows from the second inequality on the first line of (\ref{S41R2E3}). In the last inequality we used the first inequality in the second line of (\ref{S41R2E3}). The above tower implies (\ref{S41R2E5}).

As both (\ref{S41R2E4}) and (\ref{S41R2E5}) were established, we have finished our final step. This completes the proof of proposition in the basic case $k =1$ and $N = n_1= 2$.

%----------------------------------------------------------------------------------------------------------------------------------------------------------------------------------------------------------------------------
%
%     Section 4.2
%
%----------------------------------------------------------------------------------------------------------------------------------------------------------------------------------------------------------------------------
\subsection{Proof of Proposition \ref{PropMain}}\label{Section4.2} Here we present the complete proof of Proposition \ref{PropMain}. We assume the same notation as in Sections \ref{Section2} and \ref{Section3}. We proceed by induction on $N$ with the base case $N = 1$, being obvious. Suppose that we know the result for $N - 1$ and wish to prove it for $N$. 

Suppose that $k \in \mathbb{N}$ and $a = t_0 < t_1 < \cdots < t_k < t_{k+1} = b$, $n_1, \dots, n_k \in \llbracket 1, N \rrbracket$ and $z_1, \dots, z_k \in \mathbb{N}$ are all given. We wish to prove that 
\begin{equation}\label{FDEPMr1}
\mathbb{E}_{\mathbb{P}_1} \left[ \prod_{ i =1}^k \exp \left( - z_i \exp \left(\lambda \mathcal{L}^1_{n_i}(t_i)\right) \right) \right] =\mathbb{E}_{\mathbb{P}_2} \left[ \prod_{ i =1}^k \exp \left( - z_i \exp \left(\lambda \mathcal{L}^2_{n_i}(t_i)\right) \right)  \right].
\end{equation}

For clarity we split the proof into several steps. \\

{\bf \raggedleft Step 1.} We claim that there exist sequences $\{ p^v_w\}_{w = 1}^\infty$ for $v \in \{1, 2\}$ with $p^v_w \in [0,\infty]$ such that
\begin{equation}\label{Red1}
\begin{split}
&p_w^1 = p_w^2 \mbox{ for $w \in \mathbb{N}$ and } \lim_{w \rightarrow \infty} p_w^v = \mathbb{E}_{\mathbb{P}_v} \left[ \prod_{ i =1}^k \exp \left( - z_i \exp \left(\lambda \mathcal{L}^v_{n_i}(t_i)\right) \right) \right] \mbox{ for $v \in \{1, 2\}$.} 
\end{split}
\end{equation}
Since (\ref{Red1}) clearly establishes (\ref{FDEPMr1}) we only need to construct the sequences $p_w^v$. In the remainder of this step we construct $p_w^v$ and show that $p_w^1 = p_w^2$. We begin by introducing some useful notation.\\

Given points $p,q,c,d,r \in \mathbb{R}$ with $p < c< d < q$ and $\vec{x}, \vec{y} \in \mathbb{R}^{N-1}$, we define 
\begin{equation}\label{Obs1}
F(r; p,q;c,d;\vec{x}, \vec{y}):= \mathbb{P}^{1, N-1, p,q,\vec{x},\vec{y}, \infty, -\infty}_{H}\left(\mathcal{Q}_i(c) \leq r \mbox{ and } \mathcal{Q}_i(d) \leq r  \mbox{ for $i = 1, \dots, N-1$}\right),
\end{equation}
which is the probability that the curves of the $\llbracket 1, N-1 \rrbracket \times (p,q)$-indexed $H$-Brownian bridge line ensemble (as in Definition \ref{DefAvoidingLaw}) $\mathcal{Q} = (\mathcal{Q}_1, \dots, \mathcal{Q}_{N-1})$ are all below the level $r$ at times $c$ and $d$. From Lemma \ref{LemmaMeasExp} we know that $F$ is a measurable function of $\vec{x}$ and $\vec{y}$.

For $M \in \mathbb{N}$ we denote 
$$H_M(r; p,q;c,d;\vec{x},\vec{y}) = \min \left(M, \frac{1}{F(r; p,q,c,d,\vec{x},\vec{y})} \right),$$
and note that $H_M$ is a non-negative bounded measurable function of $\vec{x}$ and $\vec{y}$. 

Let $S = \{ s \in \{1, \dots, k\} : n_s = N\}$. For $w \in \mathbb{N}$ and $s \in S$ we define 
\begin{equation}\label{S42Par}
\begin{split}
&\beta_s = z_s/2, \hspace{3mm}A_w = \lambda^{-1} \log w, \hspace{3mm} V_w = w^{-1/3}, \hspace{3mm} a^s_w = t_s - \beta_s /H(A_w) - w^{-2}, \\
& b_w^s = t_s + \beta_s /H(A_w) + w^{-2}, \hspace{3mm} c_w^s = t_s - \beta_s /H(A_w), \hspace{3mm} d_w^s = t_s + \beta_s /H(A_w).
\end{split}
\end{equation}
We also fix $W_0 \in \mathbb{N}$ sufficiently large so that $w \geq W_0$ implies $d^s_w - c^s_w = 2\beta_s/ H(A_w) \in [w^{-5/4}, w^{-3/4}]$ for all $s \in S$ and $2[ w^{-2} + \beta/H(A_w)] < \min_{1 \leq i \leq k+1} (t_i - t_{i-1})$, where $\beta = \max_{i = 1}^k (z_i/2)$. Here we used the fact that $H$ is a $\lambda$-exponential Hamiltonian. 

By the induction hypothesis, we know that (\ref{FDEPMr1}) holds provided $n_1, \dots, n_k \in \llbracket 1, N-1 \rrbracket$. The latter and Lemma \ref{PropFD} imply that $\pi_{[a,b]}^{\llbracket 1, N-1 \rrbracket}(\mathcal{L}^1)$ under $\mathbb{P}_1$ and $\pi_{[a,b]}^{\llbracket 1, N-1 \rrbracket}(\mathcal{L}^2)$ under $\mathbb{P}_2$ have the same distribution as $\llbracket 1, N-1 \rrbracket$-indexed line ensembles on $[a,b]$ (we remind the reader that $\pi_{[a,b]}^{\llbracket 1, N-1 \rrbracket}$ was defined in (\ref{ProjBox})). This means that for $w \geq W_0$ 
\begin{equation}\label{Trunc1}
\mathbb{E}_{\mathbb{P}_1} \left[ G^M_w(\mathcal{L}^1; \vec{t}, \vec{n}, \vec{z}) \right] =  \mathbb{E}_{\mathbb{P}_2} \left[ G^M_w(\mathcal{L}^2; \vec{t}, \vec{n}, \vec{z}) \right],
\end{equation}
with
\begin{equation*}
\begin{split}
&G^M_w(\mathcal{L}^v; \vec{t}, \vec{n}, \vec{z}) = \prod_{s \in S^c} \exp \left( -z_s \exp \left( \lambda \mathcal{L}^v_{n_s}(t_s) \right) \right) \cdot\prod_{s \in S} \prod_{i = 1}^{N-1} {\bf 1}\{ \mathcal{L}_{i}^v(c_w^s) \leq -A_w\} {\bf 1}\{ \mathcal{L}_{i}^v(d_w^s) \leq -A_w\}  \times \\
&\prod_{s \in S} H_M\left(-A_w; a_w^s, b_w^s; c_w^s, d_w^s;\vec{x}^{s,v,w}, \vec{y}^{s,v,w} \right)  ,
\end{split}
\end{equation*}
where $\vec{x}^{s,v,w} = (\mathcal{L}_1^v(a_w^s), \dots, \mathcal{L}_{N-1}^v(a_w^s))$ and $\vec{y}^{s,v,w} = (\mathcal{L}_1^v(b_w^s), \dots, \mathcal{L}_{N-1}^v(b_w^s))$ for $v = 1,2$. Some of the notation we defined above is illustrated in Figure \ref{S4_1}.
\begin{figure}[ht]
\begin{center}
  \includegraphics[scale = 0.8]{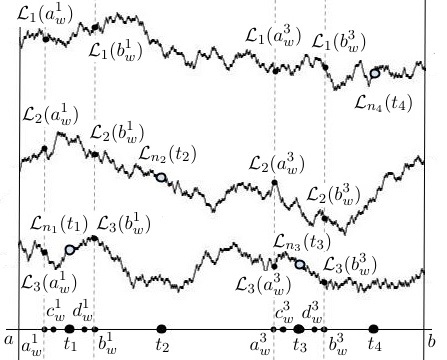}
  \caption{The figure schematically represents $\mathcal{L}^v$ where we have suppressed $v$ from the notation. In the figure, $N = 3$, $S = \{1,3\}$ and $n_1 = 3, n_2 = 2, n_3 = 3, n_4 = 1$. }
  \label{S4_1}
  \end{center}
\end{figure}

For $s \in S$ define $\mathcal{F}_{ext}^{s,w} = \mathcal{F}_{ext} (\llbracket 1, N-1 \rrbracket \times (a_w^s,b_w^s))$ as in Definition \ref{DefPBGP} and observe that by the tower property for conditional expectations and the partial $H$-Brownian Gibbs property we have

\begin{equation}\label{Trunc2}
\begin{split}
&\mathbb{E}_{\mathbb{P}_v} \left[G^M_w(\mathcal{L}^v; \vec{t}, \vec{n}, \vec{z})  \right] = \mathbb{E}_{\mathbb{P}_v} \left[ \mathbb{E}_{\mathbb{P}_v}\left[  \cdots \mathbb{E}_{\mathbb{P}_v}\left[ G^M_w(\mathcal{L}^1; \vec{t}, \vec{n}, \vec{z})\Big{|} \mathcal{F}_{ext}^{s_1,w}  \right] \cdots\Big{|} \mathcal{F}_{ext}^{s_{m},w}  \right] \right] = \\
&\mathbb{E}_{\mathbb{P}_v} \Bigg[  \prod_{s \in S^c} \exp \left( -z_s \exp \left( \lambda \mathcal{L}^v_{n_s}(t_s) \right) \right) \prod_{s \in S} H_M\left(-A_w; a_w^s, b_w^s; c_w^s, d_w^s;\vec{x}^{s,v,w}, \vec{y}^{s,v,w} \right) \times \\
&\prod_{s \in S} \mathbb{P}_{H}^{s,v,w}\left(\mathcal{Q}_{i}(c_w^s) \leq -A_w \mbox{ and } \mathcal{Q}_{i}(d_w^s) \leq - A_w\mbox{ for } i = 1, \dots, N-1 \right)  \Bigg] ,
\end{split}
\end{equation}
where $v \in \{1, 2\}$ and we have written $\mathbb{P}_{H}^{s,v,w}$ in place of $\mathbb{P}_{H}^{a^s_w, b^s_w,\vec{x}^{s,v,w},\vec{y}^{s,v,w}, \infty, \mathcal{L}_N^v[a_w^s,b_w^s]} $ to simplify the expression; note also that $(\mathcal{Q}_{1}, \dots, \mathcal{Q}_{N-1})$ is $\mathbb{P}_{H}^{s,v,w}$-distributed. In addition, $s_1, \dots, s_m$ is an enumeration of the elements of $S$ and in deriving the above expression, we also used Lemma \ref{LemmaMeasExp}, which implies that 
$ \mathbb{P}_{H}^{s,v,w}\left(\mathcal{Q}_{i}(c_w^s) \leq -A_w \mbox{ and } \mathcal{Q}_{i}(d_w^s) \leq - A_w \mbox{ for } i = 1, \dots, N-1 \right) $
is measurable with respect to the $\sigma$-algebra
$$\sigma \left\{ \mathcal{L}^v_i(t) : \mbox{  $i \in \llbracket 1, N-1 \rrbracket$ and $t \in \{a_w^s,b_w^s\}$, or $t \in [a_w^s,b_w^s]$ and $i = N$} \right\} \subset \cap_{i = 1}^m \mathcal{F}_{ext}^{s_i,w}.$$
The inclusion in the last displayed equation follows from the fact that the intervals $[a_w^{s_i}, b_w^{s_i}]$  for $i = 1, \dots, m$ are pairwise disjoint for $w \geq W_0$ by the definition of $W_0$.

Taking the limit as $M \rightarrow \infty$ in (\ref{Trunc2}) (using the monotone convergence theorem) we conclude  that for any $w \geq W_0$ and $v \in \{1,2\}$ we have
\begin{equation}\label{S42MCR}
\begin{split}
&\lim_{M \rightarrow \infty} \mathbb{E}_{\mathbb{P}_v} \left[G^M_w(\mathcal{L}^v; \vec{t}, \vec{n}, \vec{z})  \right]  = \mathbb{E}_{\mathbb{P}_v} \Bigg[  \prod_{s \in S^c} \exp \left( -z_s \exp \left( \lambda \mathcal{L}^v_{n_s}(t_s) \right) \right) \times  \\
&\prod_{s \in S}\frac{ \mathbb{P}_{H}^{s,v,w}\left(\mathcal{Q}_{i}(c_w^s) \leq -A_w \mbox{ and } \mathcal{Q}_{i}(d_w^s) \leq -A_w \mbox{ for } i = 1, \dots, N-1 \right) }{F\left(-A_w; a_w^s, b_w^s; c_w^s, d_w^s;\vec{x}^{s,v,w}, \vec{y}^{s,v,w} \right) }  \Bigg].
\end{split}
\end{equation}

We now define for $w \geq W_0$ and $v \in \{1, 2\}$ 
\begin{equation}\label{S42PDef}
\begin{split}
&p_w^v = \mathbb{E}_{\mathbb{P}_v} \Bigg[ \prod_{s \in S^c} \exp \left( -z_s \exp \left( \lambda \mathcal{L}^v_{n_s}(t_s) \right) \right)  \times \\
& \prod_{s \in S}\frac{ \mathbb{P}_{H}^{s,v,w}\left(\mathcal{Q}_{i}(c_w^s) \leq -A_w \mbox{ and } \mathcal{Q}_{i}(d_w^s) \leq -A_w  \mbox{ for } i = 1, \dots, N-1 \right) }{F\left(-A_w; a_w^s, b_w^s; c_w^s, d_w^s ;\vec{x}^{s,v,w}, \vec{y}^{s,v,w} \right) }  \Bigg],
\end{split}
\end{equation}
and for $w < W_0$ we set $p_w^v = 1$. In view of (\ref{Trunc1}) and (\ref{S42MCR}) we clearly have that $p_w^1 = p_w^2$ for all $w \geq 1$. In the next steps we will show that $p_w^v$ satisfy the second part of (\ref{Red1}). \\

{\bf \raggedleft Step 2.} We claim that for all $M \in \mathbb{N}$, $s \in S$ and $v \in \{1, 2\}$ we have $\mathbb{P}_v$-almost surely
\begin{equation}\label{S42R1}
\begin{split}
&\lim_{ w \rightarrow \infty} {\bf 1 }_{E^v_M} \cdot \frac{ \mathbb{E}_{H}^{s,v,w, -\infty}\left[ \prod_{i = 1}^{N-1}{\bf 1}\{\mathcal{Q}_{i}(c_w^s) \leq -A_w \} \cdot {\bf 1}\{ \mathcal{Q}_{i}(d_w^s)\leq -A_w  \} \cdot W_H^{s,v,w}(\mathcal{Q}) \right] }{F\left(-A_w; a_w^s, b_w^s; c_w^s, d_w^s;\vec{x}^{s,v,w}, \vec{y}^{s,v,w} \right)} = \\
& {\bf 1 }_{E^v_M} \cdot e^{ - z_s \exp \left(\lambda \mathcal{L}^v_{N}(t_s)\right) } ,\mbox{ where } E^v_M = \left \{ \sup_{ t \in [a, b]} \max_{i = 1,\dots, N} |\mathcal{L}^v_{i}(t)| \leq M \right\} \mbox{ and }\\
&W_H^{s,v,w}(\mathcal{Q}) =  \exp \left(  - \int_{a_w^s}^{b_w^s} H \left(\mathcal{L}_N^v(u) - \mathcal{Q}_{N-1}(u) \right)du \right).
\end{split}
\end{equation}
In the above equation we have written $\mathbb{P}_{H}^{s,v,w, -\infty}$ in place of $\mathbb{P}_{H}^{1, N-1, a_w^s, b_w^s,\vec{x}^{s,v,w},\vec{y}^{s,v,w}, \infty, -\infty}$ to ease the notation, and we write $\mathbb{E}_H^{s,v,w, -\infty}$ for the expectation operator with respect to this measure. We also mention that $\mathcal{Q}$ is distributed according to $\mathbb{P}_{H}^{s,v,w, -\infty}$. We will prove (\ref{S42R1}) in the steps below. Here we assume its validity and conclude the proof of (\ref{Red1}).

From Definition \ref{DefAvoidingLaw} we know that 
\begin{equation*}
\begin{split}
&\frac{ \mathbb{P}_{H}^{s,v,w}\left(\mathcal{Q}_{i}(c_w^s) \leq -A_w \mbox{ and } \mathcal{Q}_{i}(d_w^s) \leq -A_w \mbox{ for } i = 1, \dots, N-1 \right) }{F\left(-A_w; a_w^s, b_w^s; c_w^s, d_w^s;\vec{x}^{s,v,w}, \vec{y}^{s,v,w} \right) } = \\
& \frac{1}{\mathbb{E}^{s,v,w,-\infty}_{H} \left[ W_H^{s,v,w}(\mathcal{Q}) \right]} \cdot \frac{ \mathbb{E}_{H}^{s,v,w, -\infty}\left[ \prod_{i = 1}^{N-1}{\bf 1}\{\mathcal{Q}_{i}(c_w^s) \leq -A_w \} \cdot {\bf 1}\{ \mathcal{Q}_{i}(d_w^s)\leq -A_w  \} \cdot W_H^{s,v,w}(\mathcal{Q}) \right] }{F\left(-A_w; a_w^s, b_w^s; c_w^s, d_w^s;\vec{x}^{s,v,w}, \vec{y}^{s,v,w} \right) }.
\end{split}
\end{equation*}
In addition, from Lemma \ref{S3Simple} we know that $\mathbb{P}_v$-almost surely we have
$$\lim_{ w \rightarrow \infty} {\bf 1 }_{E^v_M} \cdot \mathbb{E}^{s,v,w,-\infty}_{H} \left[ W_H^{s,v,w}(\mathcal{Q}) \right] = {\bf 1 }_{E^v_M}.$$
Combining the last two statements with (\ref{S42R1}) we conclude that for all $M \in \mathbb{N}$, $s\in S$ and $v \in \{1, 2\}$ we have $\mathbb{P}_v$-almost surely
$$\lim_{ w \rightarrow \infty} {\bf 1 }_{E^v_M} \cdot \frac{ \mathbb{P}_{H}^{s,v,w}\left(\mathcal{Q}_{i}(c_w^s) \leq -A_w \mbox{ and } \mathcal{Q}_{i}(d_w^s) \leq -A_w \mbox{ for } i = 1, \dots, N-1 \right) }{F\left(-A_w; a_w^s, b_w^s; c_w^s, d_w^s;\vec{x}^{s,v,w}, \vec{y}^{s,v,w} \right)  } = {\bf 1 }_{E^v_M} \cdot e^{ - z_s \exp \left(\lambda \mathcal{L}^v_{N}(t_s)\right) } .$$

Since $\mathbb{P}_v \left( \cup_{M = 1}^\infty E^v_M \right) = 1$ we conclude that $\mathbb{P}_v$-almost surely
\begin{equation}\label{S4E5V2}
\begin{split}
&\prod_{i = 1}^k\exp \left( -z_i \exp \left( \lambda \mathcal{L}^v_{n_i}(t_i) \right) \right) = \lim_{ w \rightarrow \infty}  \prod_{s \in S^c} \exp \left( -z_s \exp \left( \lambda \mathcal{L}^v_{n_s}(t_s) \right) \right) \times \\
& \prod_{s \in S}\frac{ \mathbb{P}_{H}^{s,v,w}\left(\mathcal{Q}_{i}(c_w^s) \leq -A_w \mbox{ and } \mathcal{Q}_{i}(d_w^s) \leq -A_w \mbox{ for } i = 1, \dots, N-1 \right) }{F\left(-A_w; a_w^s, b_w^s; c_w^s, d_w^s;\vec{x}^{s,v,w}, \vec{y}^{s,v,w} \right)  } .
\end{split}
\end{equation}
From Lemma \ref{MCLfg} (applied to $a = a_w^s$, $b = b_w^s$, $\vec{x} = \vec{x}^{s,v,w}$, $\vec{y} = \vec{y}^{s,v,w}$, $f^t = f^b = \infty$, $g^t = \mathcal{L}^v_N[a_w^s, b_w^s]$ and $g^b = -\infty$) we know that for each $s\in S$ and $v \in \{1, 2\}$
\begin{equation}\label{S4E6V2}
 \frac{ \mathbb{P}_{H}^{s,v,w}\left(\mathcal{Q}_{i}(c_w^s) \leq -A_w \mbox{ and } \mathcal{Q}_{i}(d_w^s) \leq -A_w \mbox{ for } i = 1, \dots, N-1 \right) }{F\left(-A_w; a_w^s, b_w^s; c_w^s, d_w^s;\vec{x}^{s,v,w}, \vec{y}^{s,v,w} \right)  } \leq 1.
\end{equation}
We mention that we used in (\ref{S4E6V2}) that $H$ is convex, as required for the application of Lemma \ref{MCLfg}. 

Equation (\ref{S4E6V2}) implies that the terms on the right side of (\ref{S4E5V2}) are all in $[0,1]$. Combining (\ref{S42PDef}), (\ref{S4E5V2}), (\ref{S4E6V2}) and the bounded convergence theorem, we conclude the second part of (\ref{Red1}). \\

{\bf \raggedleft Step 3.} In this step we prove (\ref{S42R1}). For $w \geq W_0$ and $s\in S$ we denote the events
\begin{enumerate}
\item $\tilde{A}^s_w = \{ -A_w - 1 \leq \mathcal{Q}_i(t) \leq M + 1 \mbox{ for $a_w^s \leq t \leq c_w^s $ and $i = 1,\dots, N-1$} \}$;
\item $\tilde{D}^s_w = \{ -A_w - 1 \leq \mathcal{Q}_i(t) \leq M + 1 \mbox{ for $d_w^s \leq t \leq b_w^s$ and $i = 1,\dots, N-1$} \}$;
\item $\tilde{E}^s_w = \{  -A_w - 2V_w \leq \mathcal{Q}_i(t) \leq -A_w + 2V_w \mbox{ for $c_w^s \leq t \leq d_w^s$ and $i = 1,\dots, N-1$} \}$;
\item $\tilde{F}^s_w = \{\mathcal{Q}_i(c_w^s) \leq -A_w \mbox{ and } \mathcal{Q}_i(d_w^s) \leq -A_w \mbox{ for  $i = 1,\dots, N-1$} \},$
\end{enumerate}
where we recall that $V_w = w^{-1/3}$ from (\ref{S42Par}). We claim that $\mathbb{P}_v$-almost surely
\begin{equation}\label{S42R2}
\lim_{ w \rightarrow \infty} {\bf 1 }_{E^v_M} \cdot \frac{\mathbb{P}^{s,v,w,-\infty}_{H} \left[ {\bf 1}_{\tilde{A}^s_w \cap \tilde{D}^s_w \cap \tilde{E}^s_w \cap \tilde{F}^s_w}  \cdot W_H^{s,v,w}(\mathcal{Q}) \right]}{\mathbb{P}^{s,v,w,-\infty}_{H} \left( \tilde{A}^s_w \cap \tilde{D}^s_w \cap \tilde{E}^s_w \cap \tilde{F}^s_w\right)} = {\bf 1 }_{E^v_M} \cdot e^{ - z_s \exp \left( \lambda \mathcal{L}^v_{N}(t_s)\right) }.
\end{equation} 
We will prove (\ref{S42R2}) in the next step. Here we assume its validity and conclude the proof of (\ref{S42R1}). \\

We observe that by definition we have
\begin{equation*}
\begin{split}
 \frac{ \mathbb{E}_{H}^{s,v,w, -\infty}\hspace{-1mm}\left[ \prod_{i = 1}^{N-1}{\bf 1}\{\mathcal{Q}_{i}(c_w^s) \leq -A_w \} \cdot {\bf 1}\{ \mathcal{Q}_{i}(d_w^s)\leq -A_w  \} \cdot W_H^{s,v,w}(\mathcal{Q}) \right] }{F\left(-A_w; a_w^s, b_w^s; c_w^s, d_w^s;\vec{x}^{s,v,w}, \vec{y}^{s,v,w} \right) } =  \frac{ \mathbb{E}_{H}^{s,v,w, -\infty}\left[  {\bf 1}_{\tilde{F}^s_w} \cdot W_H^{s,v,w}(\mathcal{Q})  \right] }{ \mathbb{P}_{H}^{s,v,w, -\infty} ( \tilde{F}^s_w) },
\end{split}
\end{equation*}
which together with the fact that $W_H^{s,v,w} \in [0,1]$ (as $H(\cdot) \in [0, \infty)$) implies 
\begin{equation}\label{S42R1V1}
\begin{split}
&\liminf_{ w \rightarrow \infty} {\bf 1 }_{E^v_M} \cdot \frac{\mathbb{E}_{H}^{s,v,w, -\infty}\hspace{-0.5mm}\left[  {\hspace{-0.5mm} \bf 1}_{\tilde{A}^s_w \cap \tilde{D}^s_w \cap \tilde{E}^s_w \cap \tilde{F}^s_w} \cdot W_H^{s,v,w}(\mathcal{Q}) \hspace{-0.5mm} \right]}{\mathbb{P}_{H}^{s,v,w, -\infty} \left( \hspace{-0.5mm} \tilde{A}^s_w \cap \tilde{D}^s_w \cap \tilde{E}^s_w \cap \tilde{F}^s_w\right)}\cdot \frac{\mathbb{P}_{H}^{s,v,w, -\infty}\hspace{-1mm}\left(\hspace{-0.5mm} \tilde{A}^s_w \cap \tilde{D}^s_w \cap \tilde{E}^s_w \cap \tilde{F}^s_w\hspace{-0.5mm} \right)}{\mathbb{P}_{H}^{s,v,w, -\infty}(\tilde{F}^s_w)}  \leq \\
&\liminf_{ w \rightarrow \infty} {\bf 1 }_{E^v_M} \cdot \frac{\mathbb{E}_{H}^{s,v,w, -\infty} \left[ {\bf 1}_{\tilde{F}^s_w} \cdot W_H^{s,v,w}(\mathcal{Q}) \right]}{\mathbb{P}_{H}^{s,v,w, -\infty} ( \tilde{F}^s_w) } \leq  \limsup_{ w \rightarrow \infty} {\bf 1 }_{E^v_M} \cdot \frac{\mathbb{E}_{H}^{s,v,w, -\infty}\left[  {\bf 1}_{\tilde{F}^s_w} \cdot W_H^{s,v,w}(\mathcal{Q}) \right]}{\mathbb{P}_{H}^{s,v,w, -\infty} ( \tilde{F}^s_w) } \leq  \\
&\limsup_{ w \rightarrow \infty} {\bf 1 }_{E^v_M} \cdot  \frac{\mathbb{E}_{H}^{s,v,w, -\infty} \hspace{-0.5mm}\left[  \hspace{-0.5mm}  {\bf 1}_{\tilde{A}^s_w \cap \tilde{D}^s_w \cap \tilde{E}^s_w \cap \tilde{F}^s_w} \cdot W_H^{s,v,w}(\mathcal{Q}) \hspace{-0.5mm } \right] + \mathbb{P}_{H}^{s,v,w, -\infty} \hspace{-1mm} \left( \hspace{-0.5mm}(\tilde{A}^s_w \cap \tilde{D}^s_w \cap \tilde{E}^s_w)^c \cap \tilde{F}_w \hspace{-0.5mm}\right)}{\mathbb{P}_{H}^{s,v,w, -\infty}  \left( \tilde{A}_w \cap \tilde{D}_w \cap \tilde{E}_w \cap \tilde{F}_w\right)}.
\end{split}
\end{equation}

On the other hand, by Corollary \ref{BridgeLemmaV2} we know that we can find constants $C, c > 0$ (depending on $H, \beta_s, M, N$) such that for all $w \geq W_0$ we have on $E^v_M$ 
\begin{equation}\label{S42R2V1}
\frac{\mathbb{P}_{H}^{s,v,w, -\infty}  \left( (\tilde{A}^s_w \cap \tilde{D}^s_w \cap \tilde{E}^s_w)^c \cap \tilde{F}_w\right) }{\mathbb{P}_{H}^{s,v,w, -\infty}  ( \tilde{F}_w)} \leq  C e^{-c w^{1/12}}.
\end{equation} 
Combining (\ref{S42R2}) with (\ref{S42R1V1}) and (\ref{S42R2V1}) we conclude (\ref{S42R1}).\\

{\bf \raggedleft Step 4.} In this step we prove (\ref{S42R2}). Let us fix $\epsilon_1, \epsilon_2 \in (0,1)$. By our assumption that $H$ is a $\lambda$-exponential Hamiltonian we know that there exists $R_1 > 0$ such that for $x \geq R_1$ we have 
\begin{equation}\label{S42R2E1}
0 \leq H(x) \leq e^{(1 + \epsilon_1) \lambda x}.
\end{equation}
From the fact that $H:[-\infty, \infty) \rightarrow [0, \infty)$ is continuous we conclude that we can find a constant $R_2 > 0$ such that for all $x \in [-\infty, \infty)$ we have
\begin{equation}\label{S42R2E1.5}
0 \leq H(x) \leq R_2 + e^{(1 + \epsilon_1) \lambda x}.
\end{equation}

Let us fix $\omega \in E_M^v$ and note that by the continuity of $\mathcal{L}^v_N$ on $[a, b]$ we can find $\delta_1$ (depending on $\omega$) such that for 
\begin{equation}\label{S42R2E2}
\sup_{x,y \in [a, b], |x - y| \leq \delta_1} \left| \mathcal{L}^v_N(x) - \mathcal{L}^v_N(y) \right| < \epsilon_2.
\end{equation}
Let $W_1 \geq W_0$ be sufficiently large (depending on $\omega$) so that for $w \geq W_1$ we have
\begin{equation}\label{S42R2E3}
\begin{split}
& 2 w^{-2} \cdot \sup_{ x \in [-\infty, M+A_w + 1] } H(x) \leq \epsilon_1, \hspace{5mm} \max_{s \in S} (b^s_w - a^s_w) \leq \delta_1 \\
& \left|  e^{\pm 2\lambda V_w} - 1 \right| \leq \epsilon_1, \hspace{5mm}  \sup_{x \in [-M - 2, M +2]}  \left| \frac{H(x + A_w)}{H(A_w)} - e^{\lambda x} \right| \leq \epsilon_1.
\end{split}
\end{equation}
The latter inequalities hold for all large enough $w$, in view of the fact that $H$ is a $\lambda$-exponential Hamiltonian, (\ref{S42R2E1.5}) and (\ref{S42Par}) , where the latter gives $A_w = \lambda^{-1} \log w$, $V_w = w^{-1/3}$, $b^s_w - a^s_w = 2 w^{-2} + 2\beta_s /H(A_w)$.

We claim that we following inequalities for $w \geq W_1$ and $\omega \in E^v_M$ on the event $\tilde{A}^s_w \cap \tilde{D}^s_w \cap \tilde{E}^s_w \cap \tilde{F}^s_w$
\begin{equation}\label{S42R2E4}
\begin{split}
& W_H^{s,v,w}(\mathcal{Q}) \geq   \exp \left( - \epsilon_1 (1 + 2\beta_s)  -  2 \beta_s (1 + \epsilon_1) \exp( \lambda \epsilon_2) \cdot  \exp  \left( \lambda  \mathcal{L}_N^v(t_s)  \right)  \right),
\end{split}
\end{equation}
and 
\begin{equation}\label{S42R2E5}
\begin{split}
&  W_H^{s,v,w}(\mathcal{Q}) \leq   \exp \left(  2\beta_s \epsilon_1 -  2\beta_s  ( 1 - \epsilon_1)  \exp( - \lambda \epsilon_2) \cdot \exp  \left( \lambda   \mathcal{L}_N^v(t_s)  \right)  \right).
\end{split}
\end{equation}
We will prove (\ref{S42R2E4}) and (\ref{S42R2E5}) in the next (and final step). Here we assume their validity and conclude the proof of (\ref{S42R2}).\\

Multiplying both sides of (\ref{S42R2E4}) and (\ref{S42R2E5}) by ${\bf 1}_{\tilde{A}^s_w \cap \tilde{D}^s_w \cap \tilde{E}^s_w \cap \tilde{F}^s_w} $ and taking $\mathbb{E}^{s,v,w, -\infty}_{H}$-expectations on both sides we conclude that for $\omega \in E_M^v$ we have
\begin{equation*}
\begin{split}
& \mathbb{P}^{s,v,w,-\infty}_{H} \left( \tilde{A}^s_w \cap \tilde{D}^s_w \cap \tilde{E}^s_w \cap \tilde{F}^s_w\right) \cdot \exp \left( - \epsilon_1 (1 + 2\beta_s)  -  2 \beta_s (1 + \epsilon_1) \exp( \lambda \epsilon_2) \cdot  \exp  \left( \lambda  \mathcal{L}_N^v(t_s)  \right)  \right) \leq \\
&  \mathbb{E}^{s,v,w, -\infty}_{H} \left[ {\bf 1}_{\tilde{A}^s_w \cap \tilde{D}^s_w \cap \tilde{E}^s_w \cap \tilde{F}^s_w}   W_H^{s,v,w}(\mathcal{Q}) \right] \leq  \\
& \mathbb{P}^{s,v,w, -\infty}_{H} \left( \tilde{A}^s_w \cap \tilde{D}^s_w \cap \tilde{E}^s_w \cap \tilde{F}^s_w\right)  \cdot   \exp \left(  2\beta_s \epsilon_1 -  2\beta_s  ( 1 - \epsilon_1)  \exp( - \lambda \epsilon_2) \cdot \exp  \left( \lambda   \mathcal{L}_N^v(t_s)  \right)  \right).
\end{split}
\end{equation*}
Dividing the above expressions by $\mathbb{P}^{s,v,w, -\infty}_{H} \left( \tilde{A}^s_w \cap \tilde{D}^s_w \cap \tilde{E}^s_w \cap \tilde{F}^s_w\right) $ and letting $w \rightarrow \infty$ we get
\begin{equation*}
\begin{split}
&{\bf 1 }_{E^v_M} \cdot  \exp \left( - \epsilon_1 (1 + 2\beta_s)  -  2 \beta_s (1 + \epsilon_1) \exp( \lambda \epsilon_2) \cdot  \exp  \left( \lambda  \mathcal{L}_N^v(t_s)  \right)  \right) \leq \\
& \liminf_{w \rightarrow \infty} {\bf 1 }_{E^v_M} \cdot \frac{\mathbb{P}^{s,v,w,-\infty}_{H} \left[ {\bf 1}_{\tilde{A}^s_w \cap \tilde{D}^s_w \cap \tilde{E}^s_w \cap \tilde{F}^s_w}  \cdot W_H^{s,v,w}(\mathcal{Q}) \right]}{\mathbb{P}^{s,v,w,-\infty}_{H} \left( \tilde{A}^s_w \cap \tilde{D}^s_w \cap \tilde{E}^s_w \cap \tilde{F}^s_w\right)} \leq \\
& \limsup_{w \rightarrow \infty}  {\bf 1 }_{E^v_M} \cdot \frac{\mathbb{P}^{s,v,w,-\infty}_{H} \left[ {\bf 1}_{\tilde{A}^s_w \cap \tilde{D}^s_w \cap \tilde{E}^s_w \cap \tilde{F}^s_w}  \cdot W_H^{s,v,w}(\mathcal{Q}) \right]}{\mathbb{P}^{s,v,w,-\infty}_{H} \left( \tilde{A}^s_w \cap \tilde{D}^s_w \cap \tilde{E}^s_w \cap \tilde{F}^s_w\right)}   \leq \\
&{\bf 1 }_{E^v_M} \cdot   \exp \left(  2\beta_s \epsilon_1 -  2\beta_s  ( 1 - \epsilon_1)  \exp( - \lambda \epsilon_2) \cdot \exp  \left( \lambda   \mathcal{L}_N^v(t_s)  \right)  \right).
\end{split}
\end{equation*}
Since the above holds for all $\epsilon_1, \epsilon_2 \in (0,1)$ we conclude that 
$$ \lim_{w \rightarrow \infty} {\bf 1 }_{E^v_M} \cdot \frac{\mathbb{P}^{s,v,w,-\infty}_{H} \left[ {\bf 1}_{\tilde{A}^s_w \cap \tilde{D}^s_w \cap \tilde{E}^s_w \cap \tilde{F}^s_w}  \cdot W_H^{s,v,w}(\mathcal{Q}) \right]}{\mathbb{P}^{s,v,w,-\infty}_{H} \left( \tilde{A}^s_w \cap \tilde{D}^s_w \cap \tilde{E}^s_w \cap \tilde{F}^s_w\right)}= {\bf 1 }_{E^v_M}\cdot e^{-2\beta_s \exp \left(  \lambda \mathcal{L}_N^v(t_s) \right)},$$
which is the same as (\ref{S42R2}) once we recall from (\ref{S42Par}) that $\beta_s = z_s/2$. \\

{\bf \raggedleft Step 5.} In this final step we show that (\ref{S42R2E4}) and (\ref{S42R2E5}) hold for $w \geq W_1$ and $\omega \in E^v_M$ on the event $\tilde{A}^s_w \cap \tilde{D}^s_w \cap \tilde{E}^s_w \cap \tilde{F}^s_w$. In the sequel we assume that we have fixed $w \geq W_1$, $\omega \in E^v_M$ and an elementary outcome $\mathcal{Q}$ in $\tilde{A}_w \cap \tilde{D}_w \cap \tilde{E}_w \cap \tilde{F}_w$.

We observe that we have the following tower of inequalities
\begin{equation*}
\begin{split}
& W_H^{s,v,w}(\mathcal{Q}) =   \exp \left( -\int_{a_w^s}^{b_w^s} H \left( \mathcal{L}_N^v(u) - \mathcal{Q}_{N-1}(u) \right)du \right) \geq  \\
&  \exp \left(  - 2 w^{-2} \cdot \sup_{ x \in [-\infty, M+A_w + 1] } H(x) - \int_{c^s_w}^{d^s_w} H \left( \mathcal{L}_N^v(u) - \mathcal{Q}_{N-1}(u)\right)du \right) \geq \\
&\exp \left(  - \epsilon_1 -  2 \beta_s \cdot \sup_{u \in [c^s_w, d^s_w]} H \left( \mathcal{L}_N^v(u) -\mathcal{Q}_{N-1}(u)\right)/H(A_w) \right) \geq \\
&  \exp \left( - \epsilon_1 (1 + 2\beta_s)  -  2 \beta_s \cdot \sup_{u \in [c^s_w, d^s_w]} \exp  \left( \lambda \cdot ( \mathcal{L}_N^v(u) - \mathcal{Q}_{N-1}(u) - A_w) \right)\right)  \geq \\
&  \exp \left( - \epsilon_1 (1 + 2\beta_s)  -  2 \beta_s \cdot  \exp  \left( \lambda \cdot ( \epsilon_2 + \mathcal{L}_N^v(t_s) + 2 V_w) \right)\right)  \geq \\
& \exp \left( - \epsilon_1 (1 + 2\beta_s)  -  2 \beta_s (1 + \epsilon_1) e^{\lambda \epsilon_2} \cdot  \exp  \left( \lambda  \mathcal{L}_N^v(t_s)  \right)\right) .
\end{split}
\end{equation*}
Let us elaborate on the last equation briefly. The first line is just the definition of $W_H^{s,v,w}(\mathcal{Q})$ in (\ref{S42R1}). In going from the first to the second line we used the fact that for $\omega \in E_M^v$ we have $|\mathcal{L}_N^v(u)| \leq M$ and on $\tilde{A}^s_w \cap \tilde{D}^s_w$ we have $\mathcal{Q}_{N-1}(u) \geq -A_w - 1$ on $[a^s_w, c^s_w] \cup [d^s_w ,b^s_w]$, while $b^s_w - d^s_w = c^s_w - a^s_w = w^{-2}$. In going from the second to the third line we used the first inequality on the first line of (\ref{S42R2E3}) and the fact that $d^s_w - c^s_w = 2\beta_s/H(A_w)$. In going from the third to the fourth line we used the second inequality on the second line of (\ref{S42R2E3} and the fact that on $\tilde{E}^s_w$ we have $- A_w + 2V_w \leq\mathcal{Q}_{N-1}(u) \leq -A_w + 2V_w$ for $u \in [c^s_w, d^s_w]$ and hence $\mathcal{L}_N^v(u) - \mathcal{Q}_{N-1}(u) - A_w \in [-M-2, M+2]$ on $[c^s_w, d^s_w]$. In going from the fourth to the fifth line we used that on $\tilde{E}^s_w$ we have $\mathcal{Q}_{N-1}(u) \geq -A_w - 2V_w$ and also that $\mathcal{L}_N^v(t_s) + \epsilon_2 \geq \mathcal{L}_N^v(u)$ for $u \in [c^s_w, d^s_w]$ as follows from the second inequality on the first line of (\ref{S42R2E3}). In going from the fifth to the sixth line we used the first inequality in the second line of (\ref{S42R2E3}). The above tower implies (\ref{S42R2E4}).

We next observe that we have the following tower of inequalities  
\begin{equation*}
\begin{split}
&  W_H^{s,v,w}(\mathcal{Q}) \leq    \exp \left(  - \int_{c^s_w}^{d^s_w} H \left( \mathcal{L}_N^v(u) - \mathcal{Q}_{N-1}(u) \right)du \right) \leq  \\
&   \exp \left(  -  2\beta_s \cdot \inf_{ u \in [c^s_w, d^s_w]} H \left( \mathcal{L}_N^v(u) - \mathcal{Q}_{N-1}(u) \right)/H(A_w) \right) \leq  \\
&   \exp \left(  2\beta_s \epsilon_1 -  2\beta_s \cdot \inf_{ u \in [c^s_w, d^s_w]} \exp  \left( \lambda \cdot ( \mathcal{L}_N^v(u) - \mathcal{Q}_{N-1}(u) - A_w) \right) \right) \leq  \\
&   \exp \left(  2\beta_s \epsilon_1 -  2\beta_s \cdot \exp  \left( \lambda \cdot ( \mathcal{L}_N^v(t_s) - \epsilon_2 - 2 V_w) \right) \right) \leq  \\
&   \exp \left(  2\beta_s \epsilon_1 -  2\beta_s  ( 1 - \epsilon_1)  e^{-\lambda \epsilon_2} \cdot \exp  \left( \lambda   \mathcal{L}_N^v(t_s)  \right) \right).
\end{split}
\end{equation*}
Let us elaborate the last equation briefly. The first inequality follows from the definition of $W_H^{s,v,w}(\mathcal{Q})$ in (\ref{S42R1}) and the fact that $H \geq 0$. In going from the first to the second line we used that $d^s_w - c^s_w = 2\beta_s/ H(A_w)$. In going from the second to the third line we used the second inequality on the second line of (\ref{S42R2E3}). The latter inequality is applicable since $|\mathcal{L}_N^v(u)| \leq M$  for $\omega \in E_M^v$ and also on $\tilde{E}^s_w$ we have $- A_w + 2V_w \leq\mathcal{Q}_{N-1}(u) \leq -A_w + 2V_w$ for $u \in [c^s_w, d^s_w]$ and hence $\mathcal{L}_N^v(u) - \mathcal{Q}_{N-1}(u) - A_w \in [-M-2, M+2]$ on $[c^s_w, d^s_w]$. In going from the third to the fourth line we used that on $\tilde{E}^s_w$ we have $\mathcal{Q}_{N-1}(u) \leq -A_w + 2V_w$ and also that $\mathcal{L}_{N}^v(t_s) - \epsilon_2 \leq \mathcal{L}_N^v(u)$ for $u \in [c^s_w, d^s_w]$ as follows from the second inequality on the first line of (\ref{S42R2E3}). In the last inequality we used the first inequality in the second line of (\ref{S42R2E3}). The above tower implies (\ref{S42R2E5}).

As both (\ref{S42R2E4}) and (\ref{S42R2E5}) were established, we have finished our final step. This completes the proof of Proposition \ref{PropMain}.

%----------------------------------------------------------------------------------------------------------------------------------------------------------------------------------------------------------------------------
%
%     Section 4.3
%
%----------------------------------------------------------------------------------------------------------------------------------------------------------------------------------------------------------------------------
\subsection{Proof of Corollary \ref{CorMain2}}\label{Section4.3} In this section we give the proof of Corollary \ref{CorMain2}. We will use the same notation as in the statement of the corollary and Section \ref{Section4.2} above.\\

The proof is by contradiction and we assume that for every
$k\in \mathbb{N}$,  $t_1 < t_2 < \cdots < t_k$ with $t_i \in \Lambda$ and $x_1, \dots, x_k \in \mathbb{R}$ we have
\begin{equation}\label{S43E1}
\mathbb{P}_1 \left( \mathcal{L}^1_1(t_1) \leq x_1, \dots,\mathcal{L}^1_1(t_k) \leq x_k  \right) = \mathbb{P}_2 \left( \mathcal{L}^2_1(t_1) \leq x_1, \dots,\mathcal{L}^2_1(t_k) \leq x_k  \right).
\end{equation}
We know that the projection of $\mathcal{L}^2$ to the top $N_1$ curves is a $\Sigma_1$-indexed line ensemble on $\Lambda$ that satisfies the partial $H$-Brownian Gibbs property, cf. Remark \ref{RPBGP}. By our assumption above we have that this line ensemble under $\mathbb{P}_2$ has the same top curve distribution as $\mathcal{L}^1$ under $\mathbb{P}_1$ and so by Theorem \ref{ThmMain} we conclude that for any $a < b$ with $a, b \in \Lambda$ we have that $\pi^{\llbracket 1, N_1 \rrbracket}_{[a,b]}(\mathcal{L}^1)$ under $\mathbb{P}_1$ has the same distribution as $\pi^{\llbracket 1, N_1 \rrbracket}_{[a,b]}(\mathcal{L}^2)$ under $\mathbb{P}_2$ as $\Sigma_1$-indexed line ensembles (recall that $\pi^{\llbracket 1, N_1 \rrbracket}_{[a,b]}$ was defined in (\ref{ProjBox})). This allows us to repeat the arguments in Step 1 of Section \ref{Section4.2} and we let $p^v_w$ be as in (\ref{S42PDef}) for the case $k = 1$, $t_1 = (b+a)/2$, $ N- 1 =  N_1$, $S = \{1\}$, $z_1 \in \mathbb{N}$ and $W_0$ sufficiently large so that $ a^1_w, b^1_w \in [a,b]$ for $w \geq W_0$. In particular, $p_w^1 = p_w^2$ for $w \geq W_0$, where for $v \in \{1,2\}$ 
\begin{equation}\label{Obs2S43}
\begin{split}
p^v_w = \mathbb{E}_{\mathbb{P}_v} \left[ \frac{ \mathbb{P}_{H}^{1,v,w}\left(\mathcal{Q}_{i}(c_w^1) \leq -A_w \mbox{ and } \mathcal{Q}_{i}(d_w^1) \leq -A_w  \mbox{ for } i = 1, \dots, N-1 \right) }{F\left(-A_w; a_w^1, b_w^1; c_w^1, d_w^1 ;\vec{x}^{1,v,w}, \vec{y}^{1,v,w} \right) }  \right],
\end{split}
\end{equation}
and if $v = 1$ we have $\mathcal{L}_{N_1 + 1}^1[ a_w^1, b_w^1] = - \infty$ (here we used that $\mathcal{L}^1$ satisfies the $H$-Brownian Gibbs rather than the partial $H$-Brownian Gibbs property). In particular, we have by definition that
$$ \frac{ \mathbb{P}_{H}^{1,1,w}\left(\mathcal{Q}_{i}(c_w^1) \leq -A_w \mbox{ and } \mathcal{Q}_{i}(d_w^1) \leq -A_w  \mbox{ for } i = 1, \dots, N-1 \right) }{F\left(-A_w; a_w^1, b_w^1; c_w^1, d_w^1 ;\vec{x}^{1,v,w}, \vec{y}^{1,v,w} \right) } = 1,$$
and so $p_w^1 = p_w^2 = 1$ for all $w \geq W_0$. On the other hand, by repeating the arguments in Steps 2-5 of Section \ref{Section4.2} we have that
$$\lim_{w \rightarrow \infty} p^2_w = \mathbb{E}_{\mathbb{P}_2} \left[ e^{ - z_1 \exp \left(\lambda \mathcal{L}^2_{N_1 + 1}(t_1)\right)}  \right].$$
This shows that 
$$\mathbb{E}_{\mathbb{P}_2} \left[ e^{ - z_1 \exp \left(\lambda \mathcal{L}^2_{N_1 + 1}(t_1)\right) } \right] = 1,$$
which is our desired contradiction. Hence (\ref{S43E1}) cannot hold for every $k\in \mathbb{N}$,  $t_1 < t_2 < \cdots < t_k$ with $t_i \in \Lambda$ and $x_1, \dots, x_k \in \mathbb{R}$, which is what we wanted to prove.

%-------------------------------------------------------------------------------------------------------------------------------------------------------------------------------------------------
% Section 5
%
%-------------------------------------------------------------------------------------------------------------------------------------------------------------------------------------------------
\section{Proofs of lemmas from Sections \ref{Section2.3} and \ref{Section3.4}}\label{Section5} In this section we prove the lemmas from Sections \ref{Section2.3} and \ref{Section3.4}. Lemmas \ref{LemmaMeasExp} and \ref{PropFD} are proved in Section \ref{Section5.1}. Lemmas \ref{MCLxy} and \ref{MCLfg} are proved in Section \ref{Section5.3} after we establish a few results in Section \ref{Section5.2}.

%-------------------------------------------------------------------------------------------------------------------------------------------------------------------------------------------------
% Section 5.1
%
%-------------------------------------------------------------------------------------------------------------------------------------------------------------------------------------------------
\subsection{Proofs of Lemmas \ref{LemmaMeasExp} and \ref{PropFD}}\label{Section5.1} In this section we present the proofs of Lemmas \ref{LemmaMeasExp} and \ref{PropFD}, whose statements are recalled here for the reader's convenience. Below we will use the notation from Section \ref{Section2}, in particular Definition \ref{ContFunInf}. 

\begin{lemma}\label{S5LemmaMeasExp} Assume the same notation as in Definitions \ref{DefAvoidingLaw} and \ref{ContFunInf} and suppose that $h : C( \llbracket k_1, k_2 \rrbracket \times [a,b]) \rightarrow \mathbb{R}$ is a bounded Borel-measurable function. Let $k = k_2 - k_1 + 1$, and
\begin{equation*}
\begin{split}
&S = \{ (\vec{x}, \vec{y}, f,g) \in \mathbb{R}^k \times  \mathbb{R}^k \times C^+([a,b]) \times C^-([a,b]) \}, \\
\end{split}
\end{equation*}
where $S$ is endowed with the product topology and corresponding Borel $\sigma$-algebra. Then the function $G_h: S \rightarrow \mathbb{R}$, given by
\begin{equation}\label{S5MeasExpFun}
\begin{split}
&G_h(\vec{x}, \vec{y}, f,g) := \mathbb{E}_{H}^{k_1, k_2, a,b,\vec{x}, \vec{y},f,g}[h(\mathcal{Q})],
\end{split}
\end{equation}
is bounded and measurable.
\end{lemma}
\begin{proof} 

For clarity we split the proof into three steps. In the first step we reduce the proof of the lemma to showing that $G_h$ is bounded and measurable for a large enough class of bounded measurable functions, by using a monotone class argument. In the second step, we reduce the proof of the lemma to showing that $G_h$ is bounded and continuous if $h$ is bounded and continuous, which is established in the third and final step.\\

{\bf \raggedleft Step 1.} By the definition of $\mathbb{P}_{H}^{k_1, k_2, a,b,\vec{x}, \vec{y}, f , g}$ (see Definition \ref{DefAvoidingLaw}) we know that
$$G_h(\vec{x}, \vec{y}, f,g) = \frac{F_h(\vec{x}, \vec{y}, f,g)}{F_1(\vec{x}, \vec{y}, f,g)}.$$
In the above equation $1$ stands for the constant function that is equal to $1$ and
\begin{equation}\label{S11E1}
\begin{split}
F_h(\vec{x}, \vec{y},f,g)  =  \mathbb{E}_{free}^{k_1, k_2, a,b,\vec{x}, \vec{y}} \left[ h(\mathcal{Q}) \cdot \exp \left( - \sum_{i = k_1 - 1}^{k_2} \int_a^b H \left( \mathcal{Q}_{i+1}(u) - \mathcal{Q}_i(u)\right)du \right)  \right],
\end{split}
\end{equation}
where $\mathcal{Q} = (\mathcal{Q}_{k_1}, \dots, \mathcal{Q}_{k_2})$ is distributed according to $\mathbb{P}^{k_1, k_2, a,b, \vec{x},\vec{y}}_{free}$ and $\mathcal{Q}_{k_1- 1} = f$, while $\mathcal{Q}_{k_2 +1} = g$.

Let $\mathcal{H}$ denote the set of bounded measurable functions $h$ such that $G_h$ is measurable. Fix $K \in \mathbb{N}$ and $n_1, \dots,n_K \in \llbracket k_1, k_2 \rrbracket$, $t_1, \dots, t_K \in [a,b]$ and $z_1, \dots, z_K \in \mathbb{R}$. We define with this data the function $\tilde{h}: C( \llbracket 1, k \rrbracket \times [a,b]) \rightarrow \mathbb{R}$ through
\begin{equation}\label{S2Step}
\tilde{h}(Q) = \prod_{i = 1}^K {\bf 1} \{ Q(n_i, t_i) \leq z_i \}.
\end{equation}
We claim that the function $G_{\tilde{h}}$ is measurable, i.e. $\tilde{h} \in \mathcal{H}$. We will establish this claim the next step. For now we assume its validity and conclude the proof of the lemma.\\

It is clear that $\mathcal{H}$ is closed under linear combinations (by linearity of the expectation). Furthermore, if $h_n \in \mathcal{H}$ is an increasing sequence of non-negative measurable functions that increase to a bounded function $h$ then $h \in \mathcal{H}$ by the monotone convergence theorem. Finally, in view of our earlier claim we know that ${\bf 1}_{A} \in \mathcal{H}$ for any set $A \in \mathcal{A}$, where $\mathcal{A}$ is the $\pi$-system of sets of the form
$$ \{ Q \in C( \llbracket 1, k \rrbracket \times [a,b]) : Q(n_i, t_i) \leq z_i \mbox{ for $i = 1, \dots, K$}\}.$$
By the monotone class theorem, see e.g. \cite[Theorem 5.2.2]{Durrett}, we have that $\mathcal{H}$ contains all bounded measurable functions with respect to $\sigma(\mathcal{A})$, and the latter is $\mathcal{C}_{\llbracket k_1, k_2 \rrbracket}$ in view of \cite[Lemma 3.1]{DM20}. This proves the measurability of $G_h$ in (\ref{S5MeasExpFun}) for any bounded measurable $h$ and as it is clearly bounded by $\|h\|_{\infty}$ we conclude that it is bounded and measurable.\\

{\bf \raggedleft Step 2.} In this step we show that if $\tilde{h}$ is as in (\ref{S2Step}) then $G_{\tilde{h}}$ is measurable.

We claim that $F_h$ as in (\ref{S11E1}) is bounded and continuous if $h$ is bounded and continuous. This statement will be established in the next step. For now we assume its validity and conclude the proof of the lemma.\\

Observe that when $h \equiv 1$ is the constant function we have that $F_1(\vec{x}, \vec{y}, f,g) > 0$ and so if $h$ is continuous we see that $G_h$ is the ratio of two continuous functions, with the denominator being positive. This means that $G_h$ is continuous and thus measurable. 

Let us define for $r \in \mathbb{R}$ and $n \in \mathbb{N}$ the function
$$q_n(x;r) = \begin{cases} 0 &\mbox{ if $x > r + n^{-1}$} \\ 1 - n(x - r) &\mbox{ if $x \in [r, r+ n^{-1}]$} \\ 1 &\mbox{ if $x < r$},   \end{cases}$$
and then set
$$\tilde{h}_n(Q) = \prod_{i = 1}^K q_n(Q(n_i,t_i);z_i).$$
From our earlier discussion, we know $G_{h_n}$ are measurable for all $n \in \mathbb{N}$ and bounded by $1$. By the bounded convergence theorem we conclude that $G_{\tilde{h}}(\vec{x}, \vec{y}, f,g) = \lim_{n \rightarrow \infty} G_{\tilde{h}_n}(\vec{x}, \vec{y}, f,g)$ is also measurable, as desired.\\

{\bf  \raggedleft Step 3.} In this step we prove that $F_h$ as in (\ref{S11E1}) is continuous and bounded if $h$ is continuous and bounded. Fix some points $\vec{x}^{\infty}, \vec{y}^\infty \in \mathbb{R}^k$ and $f \in C^+([a,b])$ and $g \in C^-([a,b])$ and suppose that we are given any sequences $\vec{x}^n, \vec{y}^n \in \mathbb{R}^k$, $f^n \in C^+([a,b])$, $g^n \in C^-([a,b])$ such that 
$$\lim_{n \rightarrow \infty} \vec{x}^n = \vec{x}^\infty, \hspace{3mm}\lim_{n \rightarrow \infty} \vec{y}^n = \vec{y}^\infty, \hspace{3mm} \lim_{n \rightarrow \infty} f^n = f^{\infty}, \hspace{3mm} \lim_{n \rightarrow \infty} g^n = g^\infty.$$

Since $H(x) \geq 0$ we see from (\ref{S11E1}) that $F_h(\vec{x}, \vec{y}, f,g)$ is bounded by $\|h\|_\infty$. Thus we only need to prove that 
\begin{equation}\label{S2Limit}
\lim_{n \rightarrow \infty} F_h(\vec{x}^n, \vec{y}^n, f^n, g^n) =  F_h(\vec{x}^\infty, \vec{y}^\infty, f^{\infty}, g^{\infty}).
\end{equation}

Let $(\Omega, \mathcal{F}, \mathbb{P})$ denote a probability space that supports $k = k_2 - k_1 + 1$ i.i.d. Brownian bridges $\{ \tilde{B}_i \}_{i = k_1}^{k_2}$ (from $\tilde{B}_i(0) = 0$ to $\tilde{B}_i(1) = 0$) with diffusion parameter $1$. We denote the expectation operator with respect to $\mathbb{P}$ by $\mathbb{E}$. We also define for $n \in \mathbb{N} \cup \{ \infty\}$ and $i = k_1, \dots, k_2$ the bridges
\begin{equation}\label{S2BM1}
B_i^n(t) = (b-a)^{1/2} \cdot \tilde{B}_i \left( \frac{t - a}{b-a} \right) + \left(\frac{b-t}{b-a} \right) \cdot x^n_{i-k_1+1} + \left( \frac{t- a}{b-a}\right) \cdot y^n_{i - k_1 + 1}.
\end{equation}
By the definition of $F_h(\vec{x}, \vec{y}, f,g)$ in (\ref{S11E1}) we know that 
\begin{equation}\label{S2BM2}
F_h(\vec{x}, \vec{y}, f,g) = \mathbb{E} \left[ h(\mathcal{B}^n) \cdot \exp \left( - \sum_{i = k_1 - 1}^{k_2} \int_a^b H \left( B^n_{i+1}(u) - B^n_i(u)\right)du \right)  \right],
\end{equation}
where $\mathcal{B}^n = (B_{k_1}^n, \dots, B_{k_2}^n)$ is the line ensemble formed by the Brownian bridges from (\ref{S2BM1}), $B^n_{k_1-1} = f^n$ and $B^n_{k_2 + 1} = g^n$. By the continuity of $h$ we know that $\mathbb{P}$-almost surely
$$\lim_{n \rightarrow \infty}  h(\mathcal{B}^n)  =  h(\mathcal{B}^\infty).$$
By the continuity of $H$ and the bounded convergence theorem we conclude
$$\lim_{n \rightarrow \infty} \sum_{i = k_1 - 1}^{k_2} \int_a^b H \left( B^n_{i+1}(u) - B^n_i(u)\right)du = \sum_{i = k_1 - 1}^{k_2} \int_a^b H \left( B^\infty_{i+1}(u) - B^\infty_i(u)\right)du.$$
Consequently, we know that $\mathbb{P}$-almost surely
$$ \lim_{n \rightarrow \infty}  \exp \left( - \sum_{i = k_1 - 1}^{k_2} \int_a^b H \left( B^n_{i+1}(u) - B^n_i(u)\right)du \right) = \exp \left( - \sum_{i = k_1 - 1}^{k_2} \int_a^b H \left( B^\infty_{i+1}(u) - B^\infty_i(u)\right)du \right).$$
Combining the last few statements we conclude that the random variables in the expectation on the right side of (\ref{S2BM2}) converge $\mathbb{P}$-almost surely to 
$$h(\mathcal{B}^\infty) \cdot \exp \left( - \sum_{i = k_1 - 1}^{k_2} \int_a^b H \left( B^\infty_{i+1}(u) - B^\infty_i(u)\right)du \right) ,$$
and are bounded by $\|h\|_\infty$ (here we used again that $H(x) \geq 0$). We thus conclude that (\ref{S2Limit}) holds by the bounded convergence theorem. This suffices for the proof.
\end{proof}

\begin{lemma}\label{S5PropFD}  
 Let $\Sigma = \llbracket 1, N \rrbracket$ with $N \in \mathbb{N}$ or $N =\infty$, and let $\Lambda \subset \mathbb{R}$ be an interval. Suppose that $\mathcal{L}^1$ and $\mathcal{L}^2$ are $\Sigma$-indexed line ensembles on $\Lambda$ with laws $\mathbb{P}_1$ and $\mathbb{P}_2$ respectively. Suppose further that for every $k\in \mathbb{N}$,  $ t_1 < t_2 < \cdots < t_k$ with $t_i \in \Lambda^o$ (the interior of $\Lambda$) for $i = 1, \dots, k$; $n_1, \dots, n_k \in \Sigma$ and $z_1, \dots, z_k \in \mathbb{N}$ we have
\begin{equation}\label{S5FDStrict}
\mathbb{E}_{\mathbb{P}_1} \left[ \prod_{ i =1}^k \exp \left( - z_i \exp \left( \mathcal{L}^1_{n_i}(t_i)\right) \right) \right] =\mathbb{E}_{\mathbb{P}_2} \left[ \prod_{ i =1}^k \exp \left( - z_i \exp \left( \mathcal{L}^2_{n_i}(t_i)\right) \right)  \right].
\end{equation}
Then we have that $\mathbb{P}_1 = \mathbb{P}_2$.
\end{lemma}
\begin{proof} We first claim that for every $k\in \mathbb{N}$,  $ t_1 < t_2 < \cdots < t_k$ with $t_i \in \Lambda^o$ for $i = 1, \dots, k$ and $n_1, \dots, n_k \in \Sigma$ we have that the random vectors $(X_1^1, \dots, X_k^1)$ and $(X_1^2, \dots, X_k^2)$, given by 
$$X_i^1 =  \mathcal{L}^1_{n_i}(t_i) \mbox{ and } X_i^2 =  \mathcal{L}^2_{n_i}(t_i) \mbox{ for $i = 1, \dots, k$},$$
with $\mathcal{L}^1$, $\mathcal{L}^2$ having laws $\mathbb{P}_1, \mathbb{P}_2$, respectively, have the same distribution. If true, then \cite[Proposition 3.3]{DM20} would imply that $\mathbb{P}_1 = \mathbb{P}_2$. 

To show that $(X_1^1, \dots, X_k^1)$ and $(X_1^2, \dots, X_k^2)$ have the same distribution, it suffices to show that the vectors $(Y_1^1, \dots, Y_k^1)$ and $(Y_1^2, \dots, Y_k^2)$, given by
$$Y_i^1 = \exp \left( - \exp (X_i^1) \right) \mbox{ and } Y_i^2= \exp \left( - \exp (X_i^2) \right)  \mbox{ for $i = 1, \dots, k$},$$
have the same distribution. The latter is true since the function $\exp( - \exp(x))$ is a bijective diffeomorphism between $\mathbb{R}$ and $(0,1)$. In particular, by \cite[Theorem 1.2]{Billing} it suffices to show that for any bounded continuous function $f: [0,1]^k \rightarrow \mathbb{R}$ we have 
\begin{equation}\label{S2FDRed}
\mathbb{E}_{\mathbb{P}_1} \left[  f(Y_1^1, \dots, Y_k^1)  \right] =\mathbb{E}_{\mathbb{P}_2} \left[  f(Y_1^2, \dots, Y_k^2) \right].
\end{equation}
Using (\ref{S5FDStrict}) and the linearity of expectation we have for any polynomial $P: [0,1]^k \rightarrow \mathbb{R}$ that
\begin{equation}\label{S2FDRed2}
\mathbb{E}_{\mathbb{P}_1} \left[  P(Y_1^1, \dots, Y_k^1)  \right] =\mathbb{E}_{\mathbb{P}_2} \left[  P(Y_1^2, \dots, Y_k^2) \right].
\end{equation}
Moreover, by the Stone-Weierstrass theorem, see e.g. \cite[Theorem 7.26]{Rudin}, we know that there exists a sequence $P_n$ converging uniformly to $f$ on $[0,1]^k$. By the bounded convergence theorem
$$\mathbb{E}_{\mathbb{P}_1} \left[  f(Y_1^1, \dots, Y_k^1)  \right] = \hspace{-1mm}\lim_{n \rightarrow \infty} \mathbb{E}_{\mathbb{P}_1} \left[  P_n(Y_1^1, \dots, Y_k^1)  \right] \mbox{, } \mathbb{E}_{\mathbb{P}_2} \left[  f(Y_1^2, \dots, Y_k^2)  \right] = \hspace{-1mm} \lim_{n \rightarrow \infty} \mathbb{E}_{\mathbb{P}_2} \left[  P_n(Y_1^2, \dots, Y_k^2) \right] \hspace{-1mm} .$$
Combining the last statement and (\ref{S2FDRed2}) we conclude (\ref{S2FDRed}). This suffices for the proof.
\end{proof}

%-------------------------------------------------------------------------------------------------------------------------------------------------------------------------------------------------
% Section 5.2
%
%-------------------------------------------------------------------------------------------------------------------------------------------------------------------------------------------------
\subsection{Discrete approximations of $H$-Brownian bridges }\label{Section5.2} In this section we show that the $H$-Brownian bridges from Definition \ref{DefAvoidingLaw} can be approximated by sequences of random discrete line ensembles. The precise statement is given as Lemma \ref{lem:RW} below, and will be used in the proofs of Lemmas \ref{MCLxy} and \ref{MCLfg} in Section \ref{Section5.3}. We begin by introducing some relevant notation.\\

Let $X_i$ be i.i.d. random variables such that $\mathbb{P}(X_1 = - 1) = \mathbb{P}(X_1 = 0) = \mathbb{P}(X_1 = 1) = 1/3$. In addition, we let $S_N = X_1 + \cdots + X_N$ and for $z \in \llbracket -N, N\rrbracket$ we let $S^{(N,z)} = \{S_m^{(N,z)}\}_{m = 0}^N$ denote the process $\{S_m \}_{m = 0}^N$ with law conditioned so that $S_N = z$. We extend the definition of $S^{(N,z)}_t$ to non-integer values of $t$ by linear interpolation. 

We summarize some useful notation in the following definition.
\begin{definition}\label{Grids}
Fix $a,b \in \mathbb{R}$ with $b > a$ and a scaling parameter $n \in \N$. With the latter data we define two quantities $\Delta_n^t = (b-a) / n^2$ and $\Delta_n^x = \sqrt{3 \Delta^t_n / 2}$. Furthermore, we introduce two grids $\R_{n} = (\Delta_n^x) \cdot \Z$ and $\Lambda_{n^2} = \{a + m \cdot \Delta_n^t : m \in \mathbb{Z}\}$. Given $x,y \in \mathbb{R}_n$ with $|x - y| \leq \frac{\Delta_n^x}{\Delta_n^t} \cdot (b-a)$, we define the $C([a,b])$-valued random variable
$$Y(t) = x + \Delta_n^x \cdot S^{(n^2,  (y-x)/ \Delta_n^x)}_{(t -a)/\Delta_n^t}  \mbox{ for $t \in [a,b]$}.$$
As defined, $Y(t)$ is a continuous function on $[a,b]$ such that $Y(a) = x$ and $Y(b) = y$. We denote the law of $Y$ by $\mathbb{P}^{a,b,x,y}_{free, n}$. We let $\Omega(a,b, x,y)$ denote the state space of $Y$ and note that $\Omega(a,b, x,y)$ is finite and the condition $|x - y| \leq \frac{\Delta_n^x}{\Delta_n^t} \cdot (b-a)$ is equivalent to stating that $|\Omega(a,b, x,y)| \geq 1$. Moreover, it is clear from the definition of $Y$ that it is uniformly distributed on $\Omega(a,b, x,y)$. 
\end{definition}

The following result roughly states that the laws $\mathbb{P}^{a,b,x_n,y_n}_{free, n}$ weakly converge to the law of a Brownian bridge as $n \rightarrow \infty$  if the quantities $x,y$ converge. See \cite[Lemma 5.3]{DM20} for the proof.
\begin{lemma}\label{ConvToBridge} Let $x,y, a, b \in \mathbb{R}$ with $a < b$. In addition, for $n \in \mathbb{N}$ let $x_n, y_n \in \mathbb{R}_n$ be such that $|x_n - y_n| \leq \frac{\Delta_n^x}{\Delta_n^t} \cdot (b-a)$ (here we used the notation from Definition \ref{Grids}). Suppose that $x_n \rightarrow x$ and $y_n \rightarrow y$ as $n \rightarrow \infty$. If $Y^n$ are random variables with laws $\mathbb{P}_{free,n}^{a,b,x_n,y_n}$, then $Y^n$ converge weakly to $\mathbb{P}_{free}^{1,1,a,b, x,y}$ as $n \rightarrow \infty$, where we recall that $\mathbb{P}_{free}^{1,1,a,b, x,y}$ was introduced just before Definition \ref{DefAvoidingLaw}.
\end{lemma}

We next introduce the following discrete analogue of Definition \ref{DefAvoidingLaw}.
\begin{definition}\label{Grids2} Continue with the same notation as in Definition \ref{Grids}. Let us fix $k \in \mathbb{N}$ and a continuous function $H: [-\infty, \infty) \rightarrow [0, \infty)$.  In addition, let $\vec{x}, \vec{y} \in \R_n^k$ be such that $|x_i  -y_i| \leq \frac{\Delta_n^x}{\Delta_n^t}\cdot(b-a)$. Put $\Omega(a,b, \vec{x}, \vec{y}) = \prod_{i = 1}^k \Omega(a,b, x_i, y_i),$
and note that an element in $\Omega(a,b, \vec{x}, \vec{y})$ is a $k$-tuple $\mathcal{Y}^n = (Y^n_1, \dots, Y^n_k)$ of continuous functions in $C([a,b])$. Suppose that $f:[a,b] \rightarrow (-\infty, \infty]$ and $g: [a,b] \rightarrow [-\infty,\infty)$ are two continuous functions. By analogy with Definition \ref{DefAvoidingLaw} we define the following weight function on $\Omega(a,b, \vec{x}, \vec{y}) $
\begin{equation}\label{S5weight}
W^{1,k,a,b,\vec{x}, \vec{y},f,g}_{H, n}(\mathcal{Y}^n) = \exp \left( \hspace{-1mm} -  \sum_{i = 0}^{k} \sum_{s \in [a,b] \cap \Lambda_{n^2}} \Delta^t_n \cdot  H \left({Y}^n_{i+1}(s) - {Y}^n_i(s)\right)   \right), \mbox{ where }
\end{equation}
 $Y^n_0 = f, Y^n_{k+1} = g$. Note that by our assumptions on $H$ we have that $W^{1,k,a,b,\vec{x}, \vec{y},f,g}_{H, n}(\mathcal{Y}^n) \in (0, 1]$.

With the above data, we define a probability distribution function on $\Omega(a,b, \vec{x}, \vec{y})$ via 
$$\mathbb{P}_{H,n}^{1,k, a,b, \vec{x}, \vec{y},f,g} (\mathcal{Y}^n) = \frac{W^{1,k,a,b,\vec{x}, \vec{y},f,g}_{H, n}(\mathcal{Y}^n)}{Z^{1,k,a,b,\vec{x}, \vec{y},f,g}_{H, n}}, \mbox{ where }Z^{1,k,a,b,\vec{x}, \vec{y},f,g}_{H, n} = \sum_{\mathcal{Y}^n \in \Omega(a,b, \vec{x}, \vec{y})}W^{1,k,a,b,\vec{x}, \vec{y},f,g}_{H, n}(\mathcal{Y}^n) .$$
Note that $\mathcal{Y}^n$ is a $\llbracket 1, k \rrbracket$-indexed line ensemble on $[a,b]$ in the sense of Definition \ref{DefLE}, defined on the probability space $(\Omega(a,b, \vec{x}, \vec{y}), \mathcal{F}_{disc}, \mathbb{P}_{H,n}^{1,k, a,b, \vec{x}, \vec{y},f,g} )$ (here $\mathcal{F}_{disc}$ is the discrete $\sigma$-algebra). 
\end{definition}

With the above notation in place we are ready to state the main result of the section.
\begin{lemma}\label{lem:RW}
Fix $k \in \N$ and $a,b \in \mathbb{R}$ with $a < b$ and assume the same notation as in Definition \ref{Grids2}. Suppose that $f : [a,b] \to (-\infty, +\infty], g : [a,b] \to [-\infty, +\infty)$ are continuous functions. Fix $\vec{x}, \vec{y} \in \mathbb{R}^k$ and let $\vec{x}^n, \vec{y}^n \in \R_{n}^k$ be such that $\lim_{n \to \infty} \vec{x}^n = \vec x$, $\lim_{n \to \infty} \vec{y}^n = \vec y$. Then there exists $N_0 \in \mathbb{N}$ such that $\mathbb{P}_{H, n}^{1,k,a,b, \vec{x}^n, \vec{y}^n, f, g}$ are well-defined for $n \geq N_0$. Moreover, if $\mathcal{Y}^n$ are $\llbracket 1, k \rrbracket$-indexed line ensembles with laws $\mathbb{P}_{H, n}^{1,k,a,b, \vec{x}^n, \vec{y}^n, f, g}$, then $\mathcal{Y}^n$ converge weakly to $\mathbb{P}_{H}^{1, k,a,b, \vec{x}, \vec{y}, f, g}$ as $n \to \infty$.
\end{lemma}
\begin{proof} The proof below is inspired by \cite[Lemma 5.5]{DM20}.

Observe that $\mathbb{P}_{H, n}^{1,k,a,b, \vec{x}^n, \vec{y}^n, f, g}$ are well-defined if $|x^n_i  -y^n_i| \leq \frac{\Delta_n^x}{\Delta_n^t}\cdot(b-a) = n \sqrt{3(b-a)/2}$ for all $i = 1, \dots, k$, which clearly holds for all large enough $n$ as $\lim_{n \to \infty} \vec{x}^n = \vec x$, and $\lim_{n \to \infty} \vec{y}^n = \vec y$. This proves the existence of $N_0$ as in the statement of the lemma.

Let $\Lambda = [a,b]$ and $\Sigma = \llbracket 1, k\rrbracket$, we need to show that for any bounded continuous function $F : C(\Sigma \times \Lambda) \to \R$ we have
\begin{equation}\label{RWLimit}
\lim_{n \to \infty} \E\left[F(\mathcal{Y}^{n})\right] = \E\left[F(\mathcal{Q})\right],
\end{equation}
where $\mathcal{Q}$ is a $\Sigma$-indexed line ensemble, whose distribution is $\mathbb{P}_{H}^{1,k,a,b, \vec{x}, \vec{y}, f, g}$.

From Definition \ref{Grids2} we have for $n \geq N_0$ that
\begin{equation}\label{RW}
\E[F( \mathcal{Y}^n)] = \frac{\E\left[F(\CL^{n}) W_{H,n}(\CL^{n}) \right]}{\E\left[W_{H,n}(\CL^{n}) \right]},
\end{equation}
where $\CL^{n} = (\mathcal{L}_1, \dots, \mathcal{L}_k)$ is a line ensemble of $k$ independent random walk bridges with $\mathcal{L}_i$ having distribution $\mathbb{P}_{free, n}^{a,b,x_i^n,y_i^n}$ and we have written $W_{H,n}$ in place of $W^{1,k,a,b,\vec{x}, \vec{y},f,g}_{H, n}$ to ease the notation. For concreteness, we will denote the law of $\CL^n$ by $\mathbb{P}_{free, n}^{1,k,a,b,\vec x^n,\vec y^n}$. On the other hand, By Definition \ref{DefAvoidingLaw} we also have
\begin{equation}\label{S5BB}
\E[F(\mathcal{Q})] = \frac{\E\left[F(\CL) W_H(\CL) \right]}{\E\left[W_H(\CL) \right]},
\end{equation}
where $\CL = (\mathcal{L}_1, \dots, \mathcal{L}_k)$ is a $\llbracket1,k \rrbracket$-indexed line ensemble of independent Brownian bridges with distribution $\mathbb{P}_{free}^{1,k,a,b,\vec x,\vec y}$ and we have written $W_H$ in place of $W_H^{k_1, k_2, a,b, \vec{x}, \vec{y}, f,g}$ to ease the notation. 

In view of (\ref{RW}) and (\ref{S5BB}) we see that to prove (\ref{RWLimit}) it suffices to prove that for any bounded continuous function $F : C(\Sigma \times \Lambda) \to \R$ we have
\begin{equation}\label{RWLimitRed}
\lim_{n \to \infty} \E\left[F(\CL^{n})W_{H,n}(\CL^{n}) \right] = \E\left[F(\CL)  W_H(\CL)\right].
\end{equation}

By Lemma \ref{ConvToBridge} we know that $\mathcal{L}^n \implies \mathcal{L}$ as $n \rightarrow \infty.$ In addition, using that $C([a,b])$ with the uniform topology is separable, see e.g. \cite[Example 1.3, pp 11]{Billing}, we know that $C(\Sigma \times \Lambda) $ is also separable. In particular we can apply the Skorohod Representation Theorem, see \cite[Theorem 6.7]{Billing}, from which we conclude that there exists a probability space $(\Omega, \mathcal{F}, \mathbb{P})$, which supports $C(\Sigma \times [a,b])$-valued random variables $\mathcal{L}^n$ and a $C(\Sigma \times \Lambda)$-valued random variable $\mathcal{L}$ such that $\mathcal{L}^n \rightarrow \mathcal{L}$ for every $\omega \in \Omega$ and such that under $\mathbb{P}$ the law of $\mathcal{L}^n$ is $\mathbb{P}_{free, n}^{1,k,a,b,\vec x^n,\vec y^n}$, while under $\mathbb{P}$ the law of $\mathcal{L}$ is $\mathbb{P}_{free}^{1,k,a,b,\vec x,\vec y}$. Since $\mathcal{L}^n \rightarrow \mathcal{L}$ and $H$ is continuous we conclude by the bounded convergence theorem that for any $\omega \in \Omega$ we have
\begin{equation*}
\lim_{n \to \infty} \sum_{i = 0}^{k} \sum_{s \in [a,b] \cap \Lambda_{n^2}} \Delta^t_n \cdot  H \left(\CL^n_{i+1}(s) - \CL^n_i(s)\right)    = \sum_{i =0}^{k} \int_a^b H \left( \mathcal{L}_{i+1}(u) - \mathcal{L}_i(u)\right)du.
\end{equation*}
The last equation and the continuity of $F$ imply that for any $\omega \in \Omega$ we have
\begin{equation*}
\lim_{n \to \infty}F(\CL^{n})W_{H,n}(\CL^{n}) = F(\CL)  W_H(\CL).
\end{equation*}
The last equation and the the bounded convergence theorem (here we used that $W_{H,n} \in [0, 1]$ and that $F$ is bounded) imply (\ref{RWLimitRed}). This suffices for the proof.
\end{proof}

%-------------------------------------------------------------------------------------------------------------------------------------------------------------------------------------------------
% Section 5.3
%
%-------------------------------------------------------------------------------------------------------------------------------------------------------------------------------------------------
\subsection{Proofs of Lemmas \ref{MCLxy} and \ref{MCLfg} }\label{Section5.3} The main result of this section is as follows.
\begin{lemma}\label{MCLxyfg} Assume the same notation as in Definition \ref{DefAvoidingLaw} and that $H$ is convex. Fix $k \in \mathbb{N}$, $a < b$ and four continuous functions $f^t, f^b:[a,b] \rightarrow [-\infty, \infty)$, $g^t, g^b: [a,b] \rightarrow [-\infty, \infty)$ such that $f^t(x) \geq f^b(x)$ and $g^t(x) \geq g^b(x)$ for all $x \in [a,b]$. We also fix $\vec{x}^b, \vec{y}^b, \vec{x}^t, \vec{y}^t \in \mathbb{R}^k$ such that $x_i^b \leq x_i^t$, $y_i^b \leq y_i^t$ for $i = 1,\dots, k$. Then there exists a probability space $(\Omega, \mathcal{F}, \mathbb{P})$, which supports two $\llbracket 1, k \rrbracket$-indexed line ensembles $\mathcal{L}^t$ and $\mathcal{L}^b$ on $[a,b]$ such that the law of $\mathcal{L}^{t}$ {\big (}resp. $\mathcal{L}^b${\big )} under $\mathbb{P}$ is $\mathbb{P}_{H}^{1,k,a,b, \vec{x}^t, \vec{y}^t, f^t, g^t}$ {\big (}resp. $\mathbb{P}_{H}^{1,k,a,b, \vec{x}^b, \vec{y}^b, f^b, g^b}${\big )} and such that $\mathbb{P}$-almost surely we have $\mathcal{L}_i^t(x) \geq \mathcal{L}^b_i(x)$ for all $i = 1,\dots, k$ and $x \in [a,b]$.
\end{lemma}

It is clear that Lemmas \ref{MCLxy} and \ref{MCLfg} both follow from Lemma \ref{MCLxyfg}. The reason we keep the statements of the two lemmas separate earlier in the paper is that it makes their application a bit more transparent in the main body of text.

\begin{proof}[Proof of Lemma \ref{MCLxyfg}] The proof below is inspired by the proofs of \cite[Lemmas 2.6 and 2.7]{CorHamK} and \cite[Lemma 5.6]{DM20}. We assume the same notation as in Lemma \ref{MCLxyfg} and also Definition \ref{Grids2}. Specifically, we fix $\Sigma = \llbracket 1, k\rrbracket$ and $\Lambda = [a,b]$. For clarity we split the proof into five steps.\\

{\bf \raggedleft Step 1.} We choose any sequences $\vec{x}^n, \vec{y}^n, \vec{u}^n, \vec{v}^n \in \R_{n}^k$ such that for each $n \in \mathbb{N}$ we have $x_i^n \leq u_i^n$, $y_i^n \leq v_i^n$ for $i = 1, \dots,k$ and also such that $\lim_{n \to \infty} \vec{x}^n = \vec{x}^b$, $\lim_{n \to \infty} \vec{y}^n = \vec{y}^b$, $\lim_{n \rightarrow \infty}\vec{u}^n = \vec{x}^t$ and $\lim_{n\rightarrow \infty} \vec{v}^n = \vec{y}^t$. It follows from Lemma \ref{lem:RW} that there exists $N_0 \in \mathbb{N}$ such that if $n \geq N_0$ we have that $\mathbb{P}_{H, n}^{1,k,a,b, \vec{x}^n, \vec{y}^n, f^b, g^b}$ and $\mathbb{P}_{H, n}^{1,k,a,b, \vec{u}^n, \vec{v}^n, f^t, g^t}$ are well-defined. 

We claim that we can construct sequences of probability spaces $(\Omega_n, \mathcal{F}_n, \mathbb{P}_n)$ for $n \geq N_0$ that support $\llbracket 1, k \rrbracket$-indexed line ensembles $\mathcal{Y}^n$ and $\mathcal{Z}^n$, whose laws are $\mathbb{P}_{H, n}^{1,k,a,b, \vec{x}^n, \vec{y}^n, f^b, g^b}$ and  $\mathbb{P}_{H, n}^{1,k,a,b, \vec{u}^n, \vec{v}^n, f^t, g^t}$ respectively, such that for each $\omega \in \Omega_n$ we have
\begin{equation}\label{S52Red1}
\mathcal{Y}^n(\omega)(i,x) \leq \mathcal{Z}^n(\omega)(i,x) \mbox{ for $i = 1, \dots, k$ and $x \in [a,b]$}.
\end{equation}
We show (\ref{S52Red1}) in the next step. Here we assume its validity and conclude the proof of the lemma.\\

It follows from Lemma \ref{lem:RW} that $\mathcal{Y}^n$ converge weakly to $\mathbb{P}_{H}^{1,k,a,b, \vec{x}^b, \vec{y}^b, f^b, g^b}$ and $\mathcal{Z}^n$ converge weakly to $\mathbb{P}_{H}^{1,k,a,b, \vec{x}^t, \vec{y}^t, f^t, g^t}$ as $n \to \infty$. In particular, the latter sequences of measures are relatively compact, which by the separability and completeness of $C(\llbracket 1, k \rrbracket \times [a,b])$ implies that these sequences are tight, cf. \cite[Theorem 5.2]{Billing}. In particular, the sequence of random variables $(\mathcal{Y}^n, \mathcal{Z}^n)$ on $(\Omega_n, \mathcal{F}_n, \mathbb{P}_n)$ (viewed as $C(\llbracket 1, k \rrbracket \times [a,b]) \times C(\llbracket 1, k \rrbracket \times [a,b])$-valued random variables with the product topology and corresponding Borel $\sigma$-algebra) are also tight. 

By Prohorov's theorem, see \cite[Theorem 5.1]{Billing}, we conclude that the sequence of laws of $(\mathcal{Y}^n, \mathcal{Z}^n)$ is relatively compact. Let $n_m$ be a subsequence such that $(\mathcal{Y}^{n_m}, \mathcal{Z}^{n_m})$ converges weakly. By the Skorohod Representation Theorem, see \cite[Theorem 6.7]{Billing}, we conclude that there exists a probability space $(\Omega, \mathcal{F}, \mathbb{P})$, which supports $C(\Sigma \times [a,b])$-valued random variables $\mathcal{Y}^{n_m}, \mathcal{Z}^{n_m}$ and $\mathcal{Y}, \mathcal{Z}$ such that
\begin{enumerate}
\item  $\mathcal{Y}^{n_m} \rightarrow \mathcal{Y}$ for every $\omega \in \Omega$ as $m \rightarrow \infty$; 
\item $\mathcal{Z}^{n_m} \rightarrow \mathcal{Z}$  for every $\omega \in \Omega$ as $m \rightarrow \infty$; 
\item  under $\mathbb{P}$ the law of $\mathcal{Y}^{n_m}$ is $\mathbb{P}_{H, n_m}^{1,k,a,b,\vec x^{n_m},\vec y^{n_m}, f^b, g^b}$;
\item   under $\mathbb{P}$ the law of $\mathcal{Z}^{n_m}$ is $\mathbb{P}_{H, n_m}^{1,k,a,b,\vec u^{n_m},\vec v^{n_m}, f^t, g^t}$;
\item  $\mathbb{P}$-almost surely we have $\mathcal{Y}^{n_m}(i,x) \leq \mathcal{Z}^{n_m}(i,x) \mbox{ for $m \geq 1$, $i = 1, \dots, k$ and $x \in [a,b]$}$.
\end{enumerate}
Since $\mathcal{Y}^n$ converge weakly to $\mathbb{P}_{H}^{1,k,a,b, \vec{x}^b, \vec{y}^b, f^b, g^b}$ and $\mathcal{Z}^n$ converge weakly to $\mathbb{P}_{H}^{1,k,a,b, \vec{x}^t, \vec{y}^t, f^t, g^t}$ we conclude that under $\mathbb{P}$ the variables $\mathcal{Y}$ and $\mathcal{Z}$ have laws $\mathbb{P}_{H}^{1,k,a,b, \vec{x}^b, \vec{y}^b, f^b, g^b}$ and $\mathbb{P}_{H}^{1,k,a,b, \vec{x}^t, \vec{y}^t, f^t, g^t}$ respectively. Also conditions (1), (2) and (5) above imply that $\mathbb{P}$-almost surely we have 
$$\mathcal{Y}(i,x) \leq \mathcal{Z}(i,x) \mbox{ for $i = 1, \dots, k$ and $x \in [a,b]$}.$$
Consequently, taking the above probability space $(\Omega, \mathcal{F}, \mathbb{P})$ and setting $(\mathcal{L}^t, \mathcal{L}^b) =(\mathcal{Y}, \mathcal{Z})$ we obtain the statement of the lemma.\\

{\bf \raggedleft Step 2.} In this step we prove (\ref{S52Red1}). Let $Y_n$ and $Z_n$ denote the (finite) sets of possible elements in $C(\llbracket 1, k \rrbracket \times [a,b])$ that the line ensembles $\mathcal{Y}^n$ and $\mathcal{Z}^n$ can take with positive probability. We will construct a continuous time Markov chain $(A_t, B_t)$ taking values in $Y_n \times  Z_n$, such that:
\begin{enumerate}
\item $A_t$ and $B_t$ are each Markov in their own filtration;
\item $A_t$ is irreducible and has invariant measure $\mathbb{P}_{H, n}^{1,k,a,b, \vec{x}^n, \vec{y}^n, f^b, g^b}$;
\item $B_t$ is irreducible and has invariant measure $\mathbb{P}_{H, n}^{1,k,a,b, \vec{u}^n, \vec{v}^n, f^t, g^t}$;
\item for every $t \geq 0$ we have $A_t(i,x) \leq B_t(i,x)$ for $i \in \llbracket 1, k \rrbracket$ and $x \in [a,b]$.
\end{enumerate}
We will construct the Markov chain $(A_t, B_t)$ in the next step. Here we assume we have such a construction and conclude the proof of (\ref{S52Red1}).\\

From \cite[Theorems 3.5.3 and 3.6.3]{Norris} we know that $A_N$ weakly converges to $\mathbb{P}_{H, n}^{1,k,a,b, \vec{x}^n, \vec{y}^n, f^b, g^b}$ and $B_N$ weakly converges to $\mathbb{P}_{H, n}^{1,k,a,b, \vec{u}^n, \vec{v}^n, f^t, g^t}$ as $N \rightarrow \infty$. In particular, we see that $A_N, B_N$ are tight and then so is the sequence $(A_N,B_N)$. By Prohorov's theorem, see \cite[Theorem 5.1]{Billing}, we conclude that the sequence of laws of $(A_N, B_N)$ is relatively compact. Let $N_m$ be a subsequence such that  $(A_{N_m}, B_{N_m})$ converge weakly. By the Skorohod Representation Theorem, see \cite[Theorem 6.7]{Billing}, we conclude that there exists a probability space $(\Omega, \mathcal{F}, \mathbb{P})$, which supports $C(\Sigma \times [a,b])$-valued random variables $\mathcal{A}_m, \mathcal{B}_m$ and $\mathcal{A}, \mathcal{B}$ such that
\begin{itemize}
\item  $\mathcal{A}_{m} \rightarrow \mathcal{A}$ for every $\omega \in \Omega$ as $m \rightarrow \infty$; 
\item $\mathcal{B}_m \rightarrow \mathcal{B}$  for every $\omega \in \Omega$ as $m \rightarrow \infty$; 
\item  under $\mathbb{P}$ the law of $(\mathcal{A}_m, \mathcal{B}_m)$ is the same as that of $(A_{N_m}, B_{N_m})$.
\end{itemize}
The weak convergence of $A_N, B_N$ implies that $\mathcal{A}$ under $\mathbb{P}$ has law $\mathbb{P}_{H, n}^{1,k,a,b, \vec{x}^n, \vec{y}^n, f^b, g^b}$ and $\mathcal{B}$ under $\mathbb{P}$ has law $\mathbb{P}_{H, n}^{1,k,a,b, \vec{u}^n, \vec{v}^n, f^t, g^t}$. Furthermore, the fourth condition in the beginning of the step shows that $\mathcal{A}(i,x) \leq \mathcal{B}(i,x)$ for $i \in \llbracket 1, k \rrbracket$ and $x \in [a,b]$. Consequently, we can take $(\Omega_n, \mathcal{F}_n, \mathbb{P}_n)$ to be the above space $(\Omega, \mathcal{F}, \mathbb{P})$ and set $(\mathcal{Y}^n, \mathcal{Z}^n)= (\mathcal{A}, \mathcal{B})$. This proves (\ref{S52Red1}).\\

{\bf \raggedleft Step 3.} In this step we construct the chain $(A_t, B_t)$, satisfying the four conditions in the beginning of Step 2. We first describe the initial state of the Markov chain $(A_0, B_0)$. Notice that if $y \in Y_n$ there is a natural way to encode $y(i,x)$ for $i \in \llbracket 1, k \rrbracket$ by a list of $n^2$ symbols $\{-1, 0, 1\}$, where the $j$-th symbol is precisely 
$$\frac{y(i, a + j \cdot \Delta^t_n) - y(i, a + (j-1) \cdot \Delta^t_n)}{\Delta_n^x}.$$
 We define the lexicographic ordering on the set of all such lists of symbols (where of course $1 > 0 > -1$). If we look at $y(i,x)$, we see that there is a maximal sequence of $n^2$ symbols, which consists of 
$\left\lfloor \frac{1}{2} \cdot \left( \frac{y_i^n - x_i^n}{ \Delta_x}+ n^2 \right)\right\rfloor$
symbols $1$, followed by a $0$ if 
$ \frac{1}{2} \cdot \left( \frac{y_i^n - x_i^n}{ \Delta_x}+ n^2 \right)\not \in \mathbb{Z},$
 followed by 
$\left\lfloor \frac{1}{2} \cdot \left( \frac{x_i^n - y_i^n}{ \Delta_x}+ n^2 \right) \right\rfloor$
 symbols $-1$. We call the curve corresponding to this list $y^{max}(i,\cdot)$. One directly checks that $y^{max}=(y^{\max}(1, \cdot), \cdots, y^{\max}(k, \cdot)) \in Y_n.$ In showing the last statement, we implicitly used that $n \geq N_0$ so that $\Omega(a,b, \vec{x}^n, \vec{y}^n)$ as in Definition \ref{Grids2} is non-empty. 

 We analogously define $z^{max} \in Z_n$ by replacing everywhere above $x_i^n, y_i^n$ with $u_i^n, v_i^n$ respectively. Again one needs to use that $n \geq N_0$ so that $\Omega(a,b, \vec{u}^n, \vec{v}^n)$ is non-empty.  One further checks directly that $y^{max}(i,x) \leq z^{max}(i,x)$ for all $i \in \llbracket 1, k \rrbracket$ and $x\in [a,b]$. The state $(y^{max}, z^{max})$ is the initial state of our chain.\\

We next describe the dynamics. For each point $r \in \Lambda_{n^2} \cap (a,b)$ and $i \in \llbracket 1,k \rrbracket$ we have an independent Poisson process $R(i,r)$, with rate $1$. In addition, we assume that we have countable collections of uniform $\{-1, 0, 1\}$ random variables $\delta(i,r,m)$ and uniform $(0,1)$ random variables $U(i,r,\delta,m)$, indexed by $r \in \Lambda_{n^2} \cap (a,b)$, $i \in \llbracket 1,k \rrbracket$, $\delta \in \{-1, 0, 1\}$ and $m \in \mathbb{N}$. These uniform random variables and the Poisson processes are assumed to be all independent from each other. 

When $R(i,r)$ jumps at time $T$, we fix $\delta = \delta(i,r,R_T(i,r))$ and we update $(A_{T-},B_{T-})$ as follows. We erase the part of $A_{T-}(i,x)$ (resp. $B_{T-}(i,x)$) for $x \in [r - \Delta_n^t, r+\Delta_n^t]$ and replace that piece with two linear pieces connecting the points $\left(r-\Delta_n^t, A_{T-}(i, r- \Delta_n^t)\right)$ and $\left(r+\Delta_n^t, A_{T-}(i, r+ \Delta_n^t)\right)$ with $\left(r, A_{T-}(i, r) + \delta \cdot \Delta_n^x\right)$ (resp. $\left(r-\Delta_n^t, B_{T-}(i, r- \Delta_n^t)\right)$ and $\left(r+\Delta_n^t, B_{T-}(i, r+ \Delta_n^t)\right)$ with $\left(r, B_{T-}(i, r) + \delta \cdot \Delta_n^x\right)$). Call the resulting $C(\Sigma \times [a,b])$-valued elements $A_T'$ and $B_T'$. If $A_T'$ is in $Y_n$ we set $A_T$ to it if 
$$U(i,r,\delta, R_T(i,r)) \leq \frac{W^{1,k,a,b,\vec{x}^n, \vec{y}^n, f^{b}, g^b}_{H,n}(A_T')}{W^{1,k,a,b,\vec{x}^n, \vec{y}^n, f^{b}, g^b}_{H,n}(A_{T-})},$$
and otherwise we set $A_T = A_{T-}$. We remark that the ratio on the right side above is the same as ratio with $W^{1,k,a,b,\vec{x}^n, \vec{y}^n, f^{b}, g^b}_{H,n}$ replaced with $\mathbb{P}_{H, n}^{1,k,a,b, \vec{x}^n, \vec{y}^n, f^b, g^b}$, and recall that $W^{1,k,a,b,\vec{x}, \vec{y}, f, g}_{H,n}$ was defined in (\ref{S5weight}). Similarly, if $B_T'$ is in $Z_n$ we set $B_T$ to it if 
$$U(i,r,\delta, R_T(i,r)) \leq \frac{W^{1,k,a,b,\vec{x}^n, \vec{y}^n, f^{t}, g^t}_{H,n}(B_T')}{W^{1,k,a,b,\vec{x}^n, \vec{y}^n, f^{t}, g^t}_{H,n}(B_{T-})},$$
and otherwise we set $B_T = B_{T-}$. We remark that the ratio on the right side above is the same as ratio with $W^{1,k,a,b,\vec{x}^n, \vec{y}^n, f^{t}, g^t}_{H,n}$ replaced with $\mathbb{P}_{H, n}^{1,k,a,b, \vec{x}^n, \vec{y}^n, f^t, g^t}$.

This defines the dynamics of our process $(A_t, B_t)$. It is clear from the above definition that $(A_t, B_t)$ is a Markov chain and that $A_t$ and $B_t$ are individually Markov in their own filtration. This proves that condition (1) in the beginning of Step 2 hold. In the following steps we show that $(A_t, B_t)$ satisfies conditions (2), (3) and (4) in the beginning of Step 2 as well. \\

{\bf \raggedleft Step 4.} We first show that the measure $\mathbb{P}_{H, n}^{1,k,a,b, \vec{x}^n, \vec{y}^n, f^b, g^b}$ on $Y_n$ is invariant with respect to the $A_t$ dynamics from Step 3. In order to do this, it suffices to show that for any $r \in \Lambda_{n^2} \cap (a,b)$ and $i \in \llbracket 1,k \rrbracket$ if the $R(i,r)$ process jumps at time $T$ and $ A_{T-}$ has distribution $\mathbb{P}_{H, n}^{1,k,a,b, \vec{x}^n, \vec{y}^n, f^b, g^b}$ so does $A_T$. Setting $p(y) := \mathbb{P}_{H, n}^{1,k,a,b, \vec{x}^n, \vec{y}^n, f^b, g^b}(y)$, we need to show that for each $ y \in Y_n$ we have
\begin{equation}\label{S5S4E1}
\mathbb{P}(A_T = y) = p(y).
\end{equation}
We prove (\ref{S5S4E1}) by considering different cases, depending on the structure of $y$ near $y(i,r)$ -- see Figure \ref{S5_1}.\\
\begin{figure}[ht]
\begin{center}
  \includegraphics[scale = 0.6]{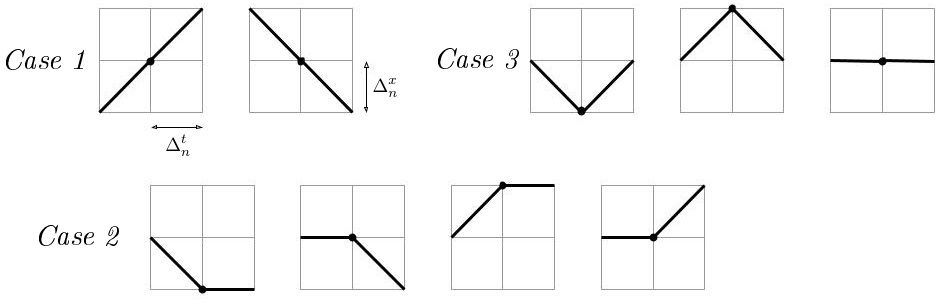}
  \vspace{-2mm}
  \caption{The picture shows the possible local structures of $y(i, \cdot)$ near the point $(i,r)$. The dot represents $y(i,r)$.}
  \label{S5_1}
  \end{center}
\end{figure}

{\em \raggedleft Case 1.} Suppose that $|y(i, r - \Delta_n^t) - y(i, r + \Delta_n^t)| = 2 \Delta_n^x$. We have two possibilities for $y$ and as they are quite similar, we only prove (\ref{S5S4E1}) when $y(i,r) = y(i, r \pm \Delta_n^t) \mp \Delta_n^x$. Notice that in this case we have that $A_T = y$ can only be achieved if $A_{T-} = y$. Moreover, any $\delta \in \{-1, 0, 1\}$ results in $A_T = y$ and so we conclude that 
$$\mathbb{P}(A_T = y)  = \mathbb{P}(A_{T-} = y) = p(y).$$

{\em \raggedleft Case 2.} Suppose that $|y(i, r - \Delta_n^t) - y(i, r + \Delta_n^t)| =  \Delta_n^x$. We have four possibilities for $y$ and as they are quite similar, we only prove (\ref{S5S4E1}) when $y(i,r) = y(i, r - \Delta_n^t) + \Delta_n^x = y(i, r+ \Delta_n^t)$. We let $y^-$ denote the element in $Y_n$ that is obtained from $y$ by erasing the part of $y(i,x)$ between $[r-\Delta_n^t, r+\Delta_n^t]$ and replacing it with two linear pieces connecting $(r- \Delta_n^t, y(i, r-\Delta_n^t))$ and $(r + \Delta_n^t, y(i, r+\Delta_n^t))$ with $(r, y(i, r) - \Delta_n^x)$. Notice that in this case we have that $A_T = y$ can only be achieved if $A_{T-} = y$ or $A_{T-} = y^-$. If $A_{T-} = y$ then we either must have $\delta = 0,1$ (and then $A_{T} = A_{T-}$ with probability $1$) or we must have $\delta = -1$ and also $p(y^{-})/ p(y) < U(i,r, -1, R_T(i,r))$. If $A_{T-} = y^-$ then we must have $\delta = 1$ and $p(y)/p(y^-) \geq U(i,r, 1, R_T(i,r))$. If $p(y) \geq p(y^-)$ then we have
$$\mathbb{P}(A_T = y) = p(y) \cdot \frac{2}{3} + p(y) \cdot \frac{1}{3} \cdot \left( 1 - \frac{p(y^{-})}{ p(y)} \right) + p(y^{-}) \cdot \frac{1}{3} \cdot 1 = p(y).$$
If $p(y) \leq p(y^-)$ then we have
$$\mathbb{P}(A_T = y) = p(y) \cdot \frac{2}{3} + p(y) \cdot \frac{1}{3} \cdot 0+ p(y^{-}) \cdot \frac{1}{3} \cdot \frac{p(y)}{ p(y^-)}  = p(y).$$
We conclude in this case that (\ref{S5S4E1}) holds.\\

{\em \raggedleft Case 3.} Suppose that $|y(i, r - \Delta_n^t) - y(i, r + \Delta_n^t)| =  0$. We have three possibilities for $y$ and as they are quite similar, we only prove (\ref{S5S4E1}) when $y(i,r) = y(i, r - \Delta_n^t) = y(i, r+ \Delta_n^t)$. We let $y^{\pm}$ denote the element in $Y_n$ that is obtained from $y$ by erasing the part of $y(i,x)$ between $[r-\Delta_n^t, r+\Delta_n^t]$ and replacing it with two linear pieces connecting $(r- \Delta_n^t, y(i, r-\Delta_n^t))$ and $(r + \Delta_n^t, y(i, r+\Delta_n^t))$ with $(r, y(i, r) \pm \Delta_n^x)$. Notice that in this case we have that $A_T = y$ can only be achieved if $A_{T-} = y$, $A_{T-} = y^-$ or $A_{T-} = y^+$. If $A_{T-} = y$ then we either must have $\delta = 0$ (and then $A_{T} = A_{T-}$ with probability $1$), or $\delta = -1$ and also $p(y^{-})/ p(y) < U(i,r, -1, R_T(i,r))$, or $\delta = 1$ and also $p(y^{+})/ p(y) < U(i,r, 1, R_T(i,r))$. If $A_{T-} = y^-$ then we must have $\delta = 1$ and $p(y)/p(y^-) \geq U(i,r, 1, R_T(i,r))$. If $ A_{T-} = y^+$ then we must have $\delta = -1$ and $p(y)/p(y^+) \geq U(i,r, -1, R_T(i,r))$. Below we compute $\mathbb{P}(A_T = y)$ in the different cases, corresponding to the possible orderings of $p(y), p(y^-)$ and $p(y^+)$.

If $p(y^-) \leq p(y) \leq p(y^+)$ then
$$\mathbb{P}(A_T = y)  = p(y) \cdot \frac{1}{3} + p(y) \cdot \frac{1}{3} \cdot \left( 1 - \frac{p(y^{-})}{ p(y)} \right) + p(y) \cdot \frac{1}{3} \cdot 0 + p(y^-) \cdot \frac{1}{3} \cdot 1 + p(y^+) \cdot \frac{1}{3} \cdot \frac{p(y)}{p(y^+)} = p(y).$$
If $p(y^-) \leq p(y^+) \leq p(y)$ then
$$\mathbb{P}(A_T = y)  = p(y) \cdot \frac{1}{3} + p(y) \cdot \frac{1}{3} \cdot \left( 1 - \frac{p(y^{-})}{ p(y)} \right) + p(y) \cdot \frac{1}{3} \cdot \left( 1 - \frac{p(y^{+})}{ p(y)} \right) + p(y^-) \cdot \frac{1}{3} \cdot 1 + p(y^+) \cdot \frac{1}{3} \cdot 1 = p(y).$$
If $p(y) \leq p(y^-) \leq p(y^+)$ or $p(y) \leq p(y^+) \leq p(y^-)$ we have
$$\mathbb{P}(A_T = y)  = p(y) \cdot \frac{1}{3} + p(y) \cdot \frac{1}{3} \cdot 0 + p(y) \cdot \frac{1}{3} \cdot 0 + p(y^-) \cdot \frac{1}{3} \cdot \frac{p(y)}{p(y^-)} + p(y^+) \cdot \frac{1}{3} \cdot \frac{p(y)}{p(y^+)} = p(y).$$
If $p(y^+) \leq p(y) \leq p(y^-)$ then
$$\mathbb{P}(A_T = y)  = p(y) \cdot \frac{1}{3} + p(y) \cdot \frac{1}{3} \cdot \left( 1 - \frac{p(y^{+})}{ p(y)} \right) + p(y) \cdot \frac{1}{3} \cdot 0 + p(y^+) \cdot \frac{1}{3} \cdot 1 + p(y^-) \cdot \frac{1}{3} \cdot \frac{p(y)}{p(y^-)} = p(y).$$
If $p(y^+) \leq p(y^-) \leq p(y)$ then
$$\mathbb{P}(A_T = y)  = p(y) \cdot \frac{1}{3} + p(y) \cdot \frac{1}{3} \cdot \left( 1 - \frac{p(y^{+})}{ p(y)} \right) + p(y) \cdot \frac{1}{3} \cdot \left( 1 - \frac{p(y^{-})}{ p(y)} \right) + p(y^+) \cdot \frac{1}{3} \cdot 1 + p(y^-) \cdot \frac{1}{3} \cdot 1 = p(y).$$

Summarizing our work above, we see that in all cases we get (\ref{S5S4E1}). This means that the measure $\mathbb{P}_{H, n}^{1,k,a,b, \vec{x}^n, \vec{y}^n, f^b, g^b}$ on $Y_n$ is invariant with respect to the $A_t$ dynamics from Step 3. An analogous argument shows that the measure $\mathbb{P}_{H, n}^{1,k,a,b, \vec{u}^n, \vec{v}^n, f^t, g^t}$ on $Z_n$ is invariant with respect to the $B_t$ dynamics from Step 3. Consequently, $(A_t, B_t)$ satisfies conditions (2) and (3) in the beginning of Step 2.\\

{\bf \raggedleft Step 5.} In this step we show that $(A_t, B_t)$ satisfies condition (4) in the beginning of Step 2. By construction, we know that $A_t(i,x) \leq B_t(i,x)$ for all $i \in \llbracket 1, k\rrbracket$ and $x \in [a,b]$ when $t =0$. What remains to be seen is that the update rule, explained in Step 3, maintains this property for all $t \geq 0$. \\

For the sake of contradiction, suppose that $A_{T-} \in Y_n$, $B_{T-}\in Z_n$ are such that $A_{T-}(i,x) \leq B_{T-}(i,x)$ for all $i \in \llbracket 1, k\rrbracket$ and $x \in [a,b]$, but that after the $(r,i)$-th Poisson process has jumped at time $T$ we no longer have that $A_{T}(i,x) \leq B_{T}(i,x)$ for all $i \in \llbracket 1, k\rrbracket$ and $x \in [a,b]$. By the formulation of the dynamics, the latter implies that $A_T(i,r) > B_T(i,r)$, and is only possible if $A_{T-}(i,r) = B_{T-}(i,r)$. In particular, we distinguish two cases: {\em Case A}. $\delta = \delta(r,i) = 1$ and $A_{T}(i,r) = A_{T-}(i,r) + \Delta_n^x$, while $B_{T}(i,r) = B_{T-}(i,r)$ or {\em Case B}. $\delta = -1$ and $A_{T}(i,r) = A_{T-}(i,r)$, while $B_{T}(i,r) = B_{T-}(i,r) - \Delta_n^x$. \\

{\em Case A.} By direct inspection one observes that the conditions $A_{T}(i,r) = A_{T-}(i,r) + \Delta_n^x$ and $B_{T-}(i,r) = A_{T-}(i,r)$ imply that $B_{T}' \in Z_n$ (see Step 3 for the definition of $B_T'$). Consequently, in order for $B_T(i,r) = B_{T-}(i,r)$ we must have that 
$$ \frac{W^{1,k,a,b,\vec{x}^n, \vec{y}^n, f^{t}, g^t}_{H,n}(B_T')}{W^{1,k,a,b,\vec{x}^n, \vec{y}^n, f^{t}, g^t}_{H,n}(B_T)} < U(i,r,1, R_T(i,r)) \leq \frac{W^{1,k,a,b,\vec{x}^n, \vec{y}^n, f^{b}, g^b}_{H,n}(A_T')}{W^{1,k,a,b,\vec{x}^n, \vec{y}^n, f^{b}, g^b}_{H,n}(A_T)}.$$
Recalling the definition of $W^{1,k,a,b,\vec{x}, \vec{y}, f, g}_{H,n}$ from (\ref{S5weight}), we see that the latter inequality implies 
\begin{equation*}
\begin{split}
&H(B_{T-}(i+1,r) - B_{T-}(i,r)) + H(B_{T-}(i,r) - B_{T-}(i-1,r)) - \\
&H(B_T'(i+1,r) - B_T'(i,r)) -  H(B_T'(i,r) - B_T'(i-1,r))  < \\
&  H(A_{T-}(i+1,r) - A_{T-}(i,r)) +  H(A_{T-}(i,r) - A_{T-}(i-1,r))- \\
& H(A_T'(i+1,r) - A_T'(i,r)) -  H(A_T'(i,r) - A_T'(i-1,r)),
\end{split}
\end{equation*}
where $A_{T-}(0,r) = A_T'(0,r) = f^b(r)$, $A_{T-}(k+1, r) = A_T'(k+1,r) = g^b(r)$, $B_{T-}(0,r) = B_T'(0,r) = f^t(r)$ and $B_{T-}(k+1, r) = B_T'(k+1,r) = g^t(r)$. 

Let us set $X: =B_{T-}(i,r) = A_{T-}(i,r)$, $Y_t : = B_{T-}(i+1,r) = B_T'(i+1,r) $, $ Z_t :=B_{T-}(i-1,r)= B_T'(i-1,r)$, $Y_b : = A_{T-}(i+1,r) = A_T'(i+1,r) $, $ Z_b :=A_{T-}(i-1,r)=A_T'(i-1,r)$ and note that then the above inequality is equivalent to
\begin{equation}\label{S5Earlier1}
\begin{split}
&H(Y_t- X) + H(X- Z_t) -  H(Y_t- X - \Delta_n^x) -  H(X +\Delta_n^x- Z_t)  < \\
& H(Y_b- X) +  H(X- Z_b)-  H(Y_b - X - \Delta_n^x) -  H(X + \Delta_n^x- Z_b).
\end{split}
\end{equation}
However, by the convexity of $H$ and the fact that $Y_t \geq Y_b$ and $Z_t \geq Z_b$ we have
$$H(Y_t- X) -  H(Y_t- X - \Delta_n^x) \geq H(Y_b- X) -  H(Y_b- X - \Delta_n^x) \mbox{ and }$$
$$H(X + \Delta_n^x- Z_b) - H(X- Z_b) \geq H(X + \Delta_n^x- Z_t) - H(X- Z_t),$$
which contradicts (\ref{S5Earlier1}).\\

{\em Case B.} By direct inspection one observes that the conditions $B_{T}(i,r) = B_{T-}(i,r) - \Delta_n^x$ and $B_{T-}(i,r) = A_{T-}(i,r)$ imply that $A_{T}' \in Z_n$. Consequently, in order for $A_T(i,r) = A_{T-}(i,r)$ we must have that 
$$\frac{W^{1,k,a,b,\vec{x}^n, \vec{y}^n, f^{b}, g^b}_{H,n}(A_T')}{W^{1,k,a,b,\vec{x}^n, \vec{y}^n, f^{b}, g^b}_{H,n}(A_T)} < U(i,r,-1, R_T(i,r)) \leq   \frac{W^{1,k,a,b,\vec{x}^n, \vec{y}^n, f^{t}, g^t}_{H,n}(B_T')}{W^{1,k,a,b,\vec{x}^n, \vec{y}^n, f^{t}, g^t}_{H,n}(B_T)}.$$
Recalling the definition of $W^{1,k,a,b,\vec{x}, \vec{y}, f, g}_{H,n}$ from (\ref{S5weight}), we see that the latter inequality implies 
\begin{equation*}
\begin{split}
&  H(A_{T-}(i+1,r) - A_{T-}(i,r)) +  H(A_{T-}(i,r) - A_{T-}(i-1,r))- \\
& H(A_T'(i+1,r) - A_T'(i,r)) -  H(A_T'(i,r) - A_T'(i-1,r)) < \\
&H(B_{T-}(i+1,r) - B_{T-}(i,r)) + H(B_{T-}(i,r) - B_{T-}(i-1,r)) - \\
&H(B_T'(i+1,r) - B_T'(i,r)) -  H(B_T'(i,r) - B_T'(i-1,r)),
\end{split}
\end{equation*}
where as before $A_{T-}(0,r) = A_T'(0,r) = f^b(r)$, $A_{T-}(k+1, r) = A_T'(k+1,r) = g^b(r)$, $B_{T-}(0,r) = B_T'(0,r) = f^t(r)$ and $B_{T-}(k+1, r) = B_T'(k+1,r) = g^t(r)$. 

Let us set $X: =B_{T-}(i,r) = A_{T-}(i,r)$, $Y_t : = B_{T-}(i+1,r) = B_T'(i+1,r) $, $ Z_t :=B_{T-}(i-1,r)= B_T'(i-1,r)$, $Y_b : = A_{T-}(i+1,r) = A_T'(i+1,r) $, $ Z_b :=A_{T-}(i-1,r)=A_T'(i-1,r)$ and note that then the above inequality is equivalent to
\begin{equation}\label{S5Earlier2}
\begin{split}
&  H(Y_b - X) +  H(X- Z_b)-  H(Y_b - X + \Delta_n^x) -  H(X - \Delta_n^x - Z_b) < \\
& H(Y_t - X) + H(X - Z_t) - H(Y_t- X + \Delta_n^x) -  H(X - \Delta_n^x - Z_t),
\end{split}
\end{equation}
However, by the convexity of $H$ and the fact that $Y_t \geq Y_b$ and $Z_t \geq Z_b$ we have
$$H(X- Z_b)-  H(X - \Delta_n^x - Z_b) \geq H(X - Z_t)  -  H(X - \Delta_n^x - Z_t) \mbox{ and }$$
$$H(Y_t- X + \Delta_n^x)  - H(Y_t - X) \geq H(Y_b- X + \Delta_n^x)  - H(Y_b - X) ,$$
which contradicts (\ref{S5Earlier2}).\\

Overall, we see that we reach a contradiction in both cases. This means that $(A_t, B_t)$ satisfies all four conditions in Step 2, which concludes the proof of the lemma.
\end{proof}

\bibliographystyle{alpha}
\bibliography{PD}

\end{document}